\documentclass{amsart}

\usepackage{amsfonts, amsmath,amscd, amssymb, latexsym, mathrsfs, stmaryrd, verbatim, wasysym }
\usepackage{enumerate}
\usepackage{graphicx}
\usepackage[all]{xy}
\newtheorem{theorem}{Theorem}

\newtheorem{definition}[theorem]{Definition}

\newtheorem{corollary}[theorem]{Corollary}
\newtheorem{example}[theorem]{Example}
\newtheorem{lemma}[theorem]{Lemma}
\newtheorem{proposition}[theorem]{Proposition}
\newtheorem{remark}[theorem]{Remark}

\newtheorem{assumption}[theorem]{Assumption}
\numberwithin{equation}{section}
\numberwithin{theorem}{section}
 \usepackage{color}
\definecolor{darkgreen}{cmyk}{1,0,1,.2}
\definecolor{m}{rgb}{1,0.1,1}
\long\def\red#1{\textcolor {red}{#1}}
\long\def\blue#1{\textcolor {blue}{#1}}

\newcommand{\K}{\mathbb{K}}

\DeclareMathOperator{\supp}{supp}   %support
 \DeclareMathOperator{\tr}{tr}

\DeclareMathOperator{\Ind}{Ind}

\DeclareMathOperator{\const}{const}

\newcommand{\forget}[1]{}

\def  \nuint {\raise10pt\hbox{$\nu$}\kern-6pt\int}

\newcommand\Tr{\operatorname{Tr}}

\def\N{\mathcal N}
\def \L{\mathcal L}
\def \A{\mathcal A}

\def \P{\mathcal P}

\newcommand\G{\mathcal G}
\newcommand\Q{\mathcal Q}
\newcommand\R{\mathcal R}
\newcommand\V{\mathcal V}

\def \L {{\cal L}}
\def \Sp {{\cal S}}
\newcommand\B{\mathcal B}
\newcommand\Bi{\B^\infty}
\def \J{\mathcal J}

\def \Ch {{\rm Ch}}
\def\Id{{\rm Id}}

\newcommand\cyl{\operatorname{cyl}}
\newcommand\ha{\frac12}

\newcommand\D{\mathcal D}
\newcommand\Di{D\kern-6pt/}
\newcommand\cDi{{\mathcal D}\kern-6pt/}
\newcommand\spi{S\kern-6pt/}
\newcommand \cspi{\Sp\kern-6pt/}

\newcommand\CC{\mathbb C}

\def \cal {\mathcal}

\def \K {{\cal K}}

\newcommand\NN{\mathbb N}

\newcommand\RR{\mathbb R}
\newcommand\ZZ{\mathbb Z}

\newcommand\pa{\partial}
\newcommand\Ker{\operatorname{Ker}}

\def\K{{\mathcal K}}

{\catcode`@=11\global\let\c@equation=\c@theorem}

\newcommand{\bint}{\sideset{^{\mathrm{b}\!\!\!}}{_{M}}\int}

\allowdisplaybreaks[2]
\date{}
 \usepackage{color}
\definecolor{darkgreen}{cmyk}{1,0,1,.2}
\definecolor{m}{rgb}{1,0.1,1}
\long\def\red#1{\textcolor {red}{#1}}
\long\def\blue#1{\textcolor {blue}{#1}}
\title[Higher genera]{Higher genera for proper actions of Lie groups, Part 2:\\  the case  of manifolds with boundary.}
\author{Paolo Piazza}
\address{Dipartimento di Matematica, Sapienza Universit\`a di Roma}
\email{piazza@mat.uniroma1.it}
\author{Hessel B. Posthuma}
\address{University of Amsterdam }
\email{H.B.Posthuma@uva.nl}


\begin{document}

\begin{abstract}
Let $G$ be a finitely connected Lie group and let $K$ be a maximal compact subgroup. Let $M$ be a cocompact $G$-proper manifold 
with boundary, endowed with a $G$-invariant metric which is of product type near the boundary.
Under additional
assumptions on $G$, for example that it satisfies the Rapid Decay condition and is such that $G/K$ has nonpositive sectional curvature,
we define  higher Atiyah-Patodi-Singer $C^*$-indices associated to elements $[\varphi]\in H^*_{{\rm diff}} (G)$ and to
a generalized $G$-equivariant Dirac operator $D$ on $M$ with $L^2$-invertible boundary operator $D_\partial$. We then establish
a higher index formula for these $C^*$-indices and use it in order to introduce higher genera for $M$, thus generalizing to
manifolds with boundary the results that we have established in Part 1. Our results
apply in particular to a semisimple Lie group $G$. We use crucially  the pairing 
between suitable  relative cyclic cohomology groups and relative K-theory groups.
\end{abstract}

\maketitle

\tableofcontents

\section*{Introduction}
This is the second part of a work devoted to the study of higher $C^*$-indices on $G$-proper manifolds and to the related notion of higher genera.
Let $G$ be a finitely connected Lie group. Let $M$ be a cocompact $G$-proper manifold without boundary. Let $D$ be
$G$-equivariant Dirac operator on $M$, acting
on the sections of a $G$-equivariant vector bundle $E$.
Higher $C^*$-indices for $D$ are defined by pairing the $C^*$-index class associated to  $D$ with suitable cyclic cocycles. Our main concern 
has been and will be  with cyclic cocycles associated to group cocycles $\varphi\in Z^k_{{\rm diff}} (G)$.
In part 1
we established sufficient conditions on $G$ ensuring that the cyclic cocycle 
$\tau^M_\varphi$ associated to $\varphi$, initially
defined on the algebra of $G$-invariant smoothing kernels of  $G$-compact support, $\mathcal{A}^c_G (M,E)$, extend to 
a dense holomorphically closed subalgebra of the Roe $C^*$-algebra $C^*(M,E)^G$. 
Our sufficient conditions are satisfied, for example, by Lie groups $G$ satisfying the Rapid Decay condition
and such that $G/K$, with $K$ a maximal compact  subgroup, has nonpositive sectional curvature. Connected semisimple Lie groups do satisfy these
last two conditions and constitute one of the main example to which our results apply. 
More precisely, what we established in \cite{PP-akt} is that under the stated assumptions on $G$ there exists a subalgebra $\mathcal{A}^\infty_G (M,E)$ of $C^*(M,E)^G$
 which is dense and holomorphically closed
and with the property that each 
$\tau^M_\varphi$, $\varphi\in Z^k_{{\rm diff}} (G)$, extends from 
$\mathcal{A}^c_G (M,E)$ to  
$\mathcal{A}^\infty_G (M,E)$. It then follows that the homomorphism $$K_* (\mathcal{A}_G^c (M,E))\xrightarrow{\langle \tau^M_\varphi,\cdot\rangle} \CC$$ defined by $\tau^M_\varphi$
through the pairing  $$HC^* (\mathcal{A}^c_G (M,E))\times K_* (\mathcal{A}^c_G (M,E))\to\CC$$ extends to a homomorphism
$$K_* (\mathcal{A}^\infty_G (M,E))\equiv K_* (C^*(M,E)^G)\rightarrow\CC.$$
The 
$C^*$-higher index  associated to $\varphi\in Z^k_{{\rm diff}} (G)$, $\Ind_\varphi (D)$,  is defined as the value that this last homomorphism attains  on the $C^*$-index class $\Ind (D)\in K_* (C^* (M,E)^E)$.
 With $M$  even dimensional and $\varphi\in Z^{2p}_{{\rm diff}} (G)$, the index theorem
of Pflaum, Posthuma and Tang \cite{PPT} was then applied in order to give a geometric formula for this higher $C^*$-index:
$$\Ind_\varphi (D) = \int_M \chi {\rm AS}(M) \wedge \omega_\varphi$$
with $\chi$ a suitable cut-off function,  ${\rm AS}(M)$ a normalization of the  Atiyah-Singer integrand, $\omega_\varphi$ a $G$-invariant 
closed differential form naturally associated to $\varphi$. %The normalization
%in  ${\rm AS}(M)$ is obtained by multiplying the usual Atiyah-Singer form by the  the constant  $\frac{1}{(2\pi i)^p}$ in degree $\dim M-p$.
This $C^*$-index formula allowed, in turn, to define
higher genera on $M$.
Most prominent
are the higher $\widehat{A}$-genera,
$$  \widehat{A}(M,\varphi):= \int_M \chi \widehat{A}(M) \wedge \omega_\varphi$$
and the higher signatures
$$\sigma (M,\varphi):=  \int_M \chi L(M) \wedge \omega_\varphi\,.$$
Properties of the $C^*$-index class allow to prove that the higher $\widehat{A}$-genera are obstructions to the existence
of a $G$-invariant metric of positive scalar curvature  if $M$ admits a $G$-invariant spin structure
and that the numbers $\sigma (M,\varphi)$ are $G$-proper homotopy invariants.
%Notice that for these last properties to hold it is crucial to have a higher $C^*$-index theorem.

\medskip
{\it The goal of the present article is to extend these results to $G$-proper manifolds with boundary.}

\medskip
Crucial to our analysis will be the proof of a higher $C^*$-Atiyah-Patodi-Singer index formula on cocompact $G$-proper manifolds with boundary.
We are now going to state this key result.
First, let us discuss briefly our geometric set-up.
We are given an even dimensional  cocompact $G$-proper manifold with boundary $M_0$ endowed with a  $G$-invariant riemannian metric $g$
which is of product  type near the boundary: $g=dt^2 + g_{\partial}$ in a neighbourhood of  $\partial M$. We consider $D$,  a $G$-equivariant Dirac-type operator, acting on the sections
of a $G$-equivariant complex vector bundle $E$ and with product structure near the boundary. We consider the
associated manifold with cylindrical ends $M$; we  extend all the structures from $M_0$ to $M$.  We shall always assume that $\vert \pi_0 (G) \vert <\infty$.

\begin{theorem}\label{intro:maintheorem}
Assume that
the boundary operator $D_\partial$ is $L^2$-invertible. Then
\begin{enumerate}\label{theo:intro-main}
\item There is a well-defined $C^*$-index class $\Ind (D)\in K_0 (C^* (M_0\subset M, E)^G)\simeq K_0 (C^* (M_0,E)^G)$.
\item If $G$ satisfies the RD condition then there exists a dense holomorphically closed
subalgebra $\mathcal{A}^\infty_G (M,E)\subset C^* (M_0\subset M, E)^G$ and a smooth representative of the index class
$\Ind_\infty (D)\in K_0 (\mathcal{A}^\infty_G (M,E))$.
\item if $\varphi\in Z^{2p}_{{\rm diff}}(G)$ is a group cocycle of polynomial growth then we can define a cyclic cocycle $\tau^M_\varphi$
on $\mathcal{A}^\infty_G (M,E)$ and thus a higher $C^*$-index $\Ind_{\varphi}(D):= (-1)^p \frac{2p !}{p!}\langle [\tau^M_\varphi], \Ind_\infty (D) \rangle$
\item under the above assumptions there is a well-defined higher eta invariant $\eta_\varphi (D_\partial)$ and the following higher Atiyah-Patodi-Singer index formula holds:
\begin{equation}\label{intro:mainformula}
\Ind_{\varphi}(D)= \int_M \chi {\rm AS}(M) \wedge \omega_\varphi - \frac{1}{2}\eta_\varphi (D_\partial)\equiv  \int_{M_0} \chi {\rm AS}(M_0) \wedge \omega_\varphi - \frac{1}{2}\eta_\varphi (D_\partial) \,.
\end{equation}
\end{enumerate}
\end{theorem}

The above theorem fits into a long series of results generalizing the seminal work of Atiyah-Patodi and Singer
\cite{APS1}. 
We shall not attempt to reconstruct the history of this particular branch of index theory here but would like to
mention the recent preprints  \cite{HWW1,HWW2}, devoted precisely to $G$-proper manifolds with boundary;
in these papers  the index class $\Ind (D)\in K_0 (C^* (M_0 \subset M,E)^G)$ is also
defined and index formulae are proved for the pairing of it with certain 0-cocycles. We shall  comment further
on this point later in the paper.

\smallskip
Going back to the statement of theorem \ref{intro:maintheorem}, we have proved in part 1 that if $ G$ is a  Lie groups satisfying the Rapid Decay condition
and such that $G/K$, with $K$ a maximal compact  subgroup, has nonpositive sectional curvature,
then for {\it any} $\alpha\in H^*_{{\rm diff}} (G)$ there exists a representative $\varphi\in Z^*_{{\rm diff}} (G)$
which is of polynomial growth. For these Lie groups, for example {\it semisimple Lie groups}, all the conclusions of
the above theorem hold, provided of course that $D_\partial$ is $L^2$-invertible.

\smallskip
We use the above theorem in order to introduce and study higher $\widehat{A}$-genera and higher signatures 
on a $G$-proper manifold with boundary, with $G$ satisfying the assumptions
of theorem \ref{intro:maintheorem}.
For example, if $(M,g)$ admits a $G$-invariant spin structure
and is such that $g_{\partial}$ is of positive scalar curvature, then the numbers 
\begin{equation}\label{intro:higher -aroof}
\left\{\left( \int_{M_0} \chi \widehat{A}(M_0) \wedge \omega_\varphi -\frac{1}{2} \widehat{\eta}_\varphi (D_\partial)\right),\;\;[\varphi]\in H^{2*}_{{\rm diff}} (G)\right\},
\end{equation}
where $\widehat{\eta}_\varphi(D_\partial):=(2\pi i)^p\eta_\varphi(D_\partial),~[\varphi]\in H^{2p}_{{\rm diff}} (G)$,
are, by definition, the higher $\widehat{A}$-genera of $M$. They are obstructions to the existence of an isotopy 
from $g_{\partial}$ to a $G$-metric of positive scalar curvature $h_{\partial}$ which admits an extension 
to a $G$-metric $h$ on $M$ which is of product type and of positive scalar curvature on all of $M$.

\smallskip
For the proof of theorem \ref{theo:intro-main} we use
an extension of the $b$-calculus of Melrose and, most crucially,
the interplay between the pairing 
\begin{equation}\label{intro:pairing-absolute}
HC^* (\mathcal{A}^\infty_G (M,E))\times K_* (\mathcal{A}^\infty_G (M,E))\to \CC
\end{equation}
used in the definition of $\langle [\tau^M_\varphi], \Ind_\infty (D) \rangle$, and a {\it relative}
pairing 
\begin{equation}\label{intro:pairing-relative}
HC^* ({}^b \mathcal{A}^\infty_G (M,E), {}^b \mathcal{A}^\infty_{G,\RR^+} (\overline{N_+ \partial M},E))\times K_* ({}^b \mathcal{A}^\infty_G (M,E), 
{}^b \mathcal{A}^\infty_{G,\RR^+} (\overline{N_+ \partial M},E))\to\CC\end{equation}
for algebras of $b$-smoothing operators fitting into a short exact sequence of algebras
\begin{equation}\label{intro:short-exact-sequence}
0\rightarrow \mathcal{A}^\infty_G (M,E)\rightarrow {}^b \mathcal{A}^\infty_G (M,E)\rightarrow {}^b \mathcal{A}^\infty_{G,\RR^+} (\overline{N_+ \partial M},E)
\rightarrow 0\,.\end{equation}
Crucial in the proof of the index formula is the definition of a relative cyclic cocycle $(\tau_\varphi^{(M,r)},\sigma_\varphi)$
and a relative smooth index class $\Ind_\infty (D,D_\partial)$ such that
\begin{equation}\label{crucial-step}
\langle [\tau^M_\varphi], \Ind_\infty (D) \rangle = \langle [\tau_\varphi^{(M,r)},\sigma_\varphi], \Ind_\infty (D,D_\partial)\rangle 
\end{equation}
with the pairing in  \eqref{intro:pairing-relative} used
on the  right hand side. In \eqref{crucial-step} $\tau_\varphi^{(M,r)}$ is a cyclic {\it cochain} on ${}^b \mathcal{A}^\infty_G (M,E)$
defined through a regularization \`a la Melrose;
the higher eta invariant appearing in the statement of theorem \ref{theo:intro-main}
is defined in terms of the cocycle $\sigma_\varphi$, which is therefore called the {\it eta cocycle}. 
The proof of the fact that $(\tau_\varphi^{(M,r)},\sigma_\varphi)$ is indeed a relative cocycle employs (an extension of)
the $b$-trace
formula for a commutator, due to Melrose. Notice that the relative cocycle is initially defined on algebras of
$G$-compactly supported kernels and its extension to the larger algebras appearing in 
\eqref{intro:short-exact-sequence} does require a finer analysis  with respect to the closed case.

\smallskip
Having established the crucial step \eqref{crucial-step}, with the left hand side being  $\Ind_\varphi (D)$ up to 
the constant $(-1)^p \frac{2p !}{p!}$,
 the index formula   \eqref{intro:mainformula}  is proved by writing 
down explicitly the relative pairing and  applying a Geztler rescaling argument to the operator $uD$ as $u\downarrow 0^+$.
The above technique was initiated in \cite{mp} where a higher APS index formula for the Godbillon-Vey cocycle on a foliated bundle
with boundary
was proved; this technique  has been also used successfully in the higher APS index theorem on Galois coverings,
see \cite{GMPi}. Relative cyclic cohomology and relative K-theory played a key role 
also  in  work of Lesch-Moscovici-Pflaum \cite{LMP2}.

\medskip
\noindent
Notice that even though  the logical path leading to theorem \ref{theo:intro-main}
was rather clear from the outset, there are a number of technical issues that need to be resolved when we pass from Galois coverings to $G$-proper manifolds. 
Sorting out these issues
 is 
certainly one  contribution of this article; moreover, in the process  we needed
to clarify certain aspects of the theory in the case of {\it closed} $G$-proper manifolds as well. We believe that these 
aspects have
their intrinsic interest.

\medskip
\noindent
{\bf The article is organized as follows.} In Section \ref{sect:closed-case} we elaborate on the results proved 
in Part 1 in the closed case: in particular we give a heat-kernel proof of the index formula on $G$-proper
manifolds proved by the
second author with Pflaum and Tang and of its sharpening into a $C^*$-index formula, proved in Part 1 of this work. 
This requires, among other things, to
 prove that if $G$ satisfies the RD condition with respect to a length function $L$,
then the heat kernel and the Connes-Moscovici projector belong to the algebra 
of kernels $\mathcal{A}_G^\infty (M,E)$
associated to the dense holomorphically closed algebra $H^\infty_L (G)\subset C^*_r G$. In Section 
\ref{sect:more-on-heat} we give an alternative approach to this last result, using the Cauchy-integral 
representation of the heat operator.  We also establish results that will be crucial later on for
manifolds with boundary.  Section \ref{section:geometric} is devoted to some geometric
preliminaries on $G$-proper manifolds
with boundary and on Dirac-type operators on them. We also present  a few examples. In Section \ref{sect:b-c*-index}
we prove that under the $L^2$-invertibility assumption on the boundary operator, there is a well-defined
APS $C^*$-index class $\Ind_{C^*} (D)\in K_* (C^* (M_0\subset M,E)^G)$, where we recall that $M$ denotes the manifolds with cylindrical ends
associated to $M_0$ and $C^* (M_0\subset M,E)^G$ denotes the associated 
Roe algebra. This class can in fact be defined in a number of equivalent ways; we use $b$-calculus techniques,
given that they will be crucial in the proof of the higher index formula. In Section \ref{sect:smooth-index-class}, under the RD
assumption on $G$,
we prove the existence of a {\it smooth} APS index class in the K-theory of a dense holomorphically
closed subalgebra of  $C^* (M_0\subset M,E)^G$. We also prove that this class corresponds under excision
to a relative {\it smooth} index class.  In Section \ref{sect:relative-cocycles} we define the relative cyclic cocycle
associated to $\tau_\varphi^M$. Finally in Section \ref{sect:aps} we establish the higher APS $C^*$-index formula.
Section  \ref{sect:genera} is devoted to the definition of higher genera on $G$-proper manifolds with boundary and the study of some of their geometric properties.
We close the paper with an Appendix containing results known to the experts but for which we could
not find precise and citable references.

\medskip
\noindent
{\bf Acknowledgements.} Part of this work was done during visits of the first author to University of Amsterdam
and of the second author
to Sapienza Universit\`a di Roma. Financial support for these visits was provided by the grant {\it PRIN, Spazi di moduli e teoria di Lie (MIUR, Ministero Istruzione Universit\`a  Ricerca, Italy)} and by NWO TOP
grant no. 613.001.302..\\
It is a pleasure to thank  Pierre Albin, Sacha Gorokhovsky, Bernhard Hanke and Sylvie Paycha for useful discussions.

\section{Getzler rescaling and higher indices in the closed case}\label{sect:closed-case}
A fundamental step in the analysis of the  higher genera defined in Part 1, is the higher index formula established by the
second author with Pflaum and  Tang, see \cite{PPT}. The proof of the higher index formula  in \cite{PPT} 
is based  on the algebraic index theorem of Nest and Tsygan \cite{NT-algebraic}. The main goal of this  section is to
give a proof of this formula using  Getzler rescaling techniques and to upgrade it to a higher $C^*$-index formula
under suitable assumptions on $G$.

\subsection{Three algebras of kernels}\label{subsect:3algebras}
Let $M$  be a closed smooth manifold carrying a smooth proper action of a Lie group $G$ by $(g,x)\mapsto gx$, $g\in G$, $x\in M$. We assume that $G$ has  finitely many connected components and that the quotient is compact. We  let $E$ be an equivariant complex vector bundle over $M$. We choose an invariant complete Riemannian metric denoted by $h$, with associated distance function denoted by $d_{M}(x,y)$ for $x,y\in M$, and volume form $d{\rm vol}(x)$.

Assume that $G$ satisfies the Rapid Decay (RD) condition.
The aim of the current section is to introduce a chain of algebras
\begin{equation}
\label{coa}
\mathcal{A}^c_G(M,E)\subset\mathcal{A}^{\exp}_G(M,E)\subset\mathcal{A}_G^\infty(M,E)\subset C^* (M,E)^G,
\end{equation}
all subalgebras of the Roe $C^*$-algebra $C^* (M,E)^G$, --defined below--, which all can be thought of as algebras of integral kernels.

The first algebra $\mathcal{A}^c_G (M,E)$ on the left is the algebra of smooth kernels of $G$-compact support, i.e., 
\[
\mathcal{A}^c_G (M,E):=\{A\in C^{\infty}(M\times M,E\boxtimes E^*)^{G},~\pi(\supp(A))\subset (M\times M)\slash G~\mbox{compact}\},
\]
where $\pi:M\times M\to (M\times M)\slash G$ denotes the projection with respect to the diagonal action. It is  well known that
$\mathcal{A}^c_G (M,E)$ is a Fr\'echet algebra
under the convolution product $*$.
Each element $A\in \mathcal{A}_{G}^{c}(M,E)$ defines an equivariant linear operator 
$S_A :C^\infty_c (M,E)\to C^\infty_c (M,E)$, the integral operator associated to the kernel $A$. One proves that
 $S_{A_1}\circ S_{A_2}=S_{A_1 * A_2}$ and that $S_A$ extends to an equivariant bounded operator 
on $L^2 (M,E)$. The corresponding subalgebra  of $\mathcal{B}(L^2 (M,E))$ is denoted
by  $\mathcal{S}_G^c (M,E)$; by definition
\begin{equation}\label{sGc}
\mathcal{S}_G^c (M,E):=\{S_A, A\in  \mathcal{A}_{G}^{c}(M,E)\}.
\end{equation}
We shall often identify $\mathcal{S}_G^c (M,E)$ with $\mathcal{A}_{G}^{c}(M,E)$. Recall also that 
$\mathcal{S}_G^c (M,E)$ is a subalgebra of the Roe algebra $C^*$-algebra $C^* (M,E)^G$
defined as the closure of the bounded $G$-equivariant operators on $L^2 (M,E)$ that are locally compact and of finite propagation.

The second algebra in the chain \eqref{coa} is given by exponentially rapidly decreasing invariant kernels:
\[
\mathcal{A}^{\exp}_G(M,E):=\left\{A\in C^{\infty}(M\times M,E\boxtimes E^*)^{G},~\sup_{x,y\in M}\left| e^{qd_M(x,y)}\nabla^{m}_{x}\nabla^{n}_{y}A(x,y)\right| < C_q,~\mbox{for all}~ q,m,n\in\mathbb{N}\right\}.
\]
It is proved in Appendix \ref{red} that these kernels indeed form an algebra under convolution, in fact a Fr\'echet algebra with 
the obvious seminorms.

It is a well-known fact \cite{abels} that, because the $G$-action on $M$ is proper and $G$ has finitely many
connected components, there exists a global slice: this is a compact submanifold
$S\subset M$ on which the $G$-action restricts to an action of maximal compact subgroup $K\subset G$, such that 
the natural map
\[
G\times_K S\to M,\quad [g,x]\mapsto g\cdot x,
\]
is a diffeomorphism. This decomposition of $M$ induces an isomorphism
\begin{equation}
\label{algebra-slice}
\mathcal{A}_G^c (M,E)\cong \left(C^\infty_c(G)\hat{\otimes}\Psi^{-\infty}(S,E|_S)\right)^{K\times K},
\end{equation}
where $\Psi^{-\infty}(S)$ denotes the algebra of smoothing operators on $S$. Here $\hat{\otimes}$ denotes the completed tensor product of 
the Fr\'echet algebras on both sides. For this we identify $\Psi^{-\infty}(S)\cong C^\infty(S\times S)$ and we use the $C^\infty$-topologies of uniform convergence 
of derivatives, written as $||~||_\alpha$, where $\alpha$  is a multi-index labeling the derivatives. Remark that because $S$ is compact, $\Psi^{-\infty}(S)$ is  nuclear and 
therefore all topological tensor products $\hat{\otimes}$ coincide.

Similarly, for the exponentially rapidly decaying kernels we have the decomposition
\[
\mathcal{A}^{\exp}_G(M,E)\cong \left(\mathcal{A}^{\exp}(G)\hat{\otimes} \Psi^{-\infty}(S,E|_S)\right)^{K\times K},
\]
where, to define $\mathcal{A}^{\exp}(G)$, the group $G$ is viewed as a $G$ space with respect to left multiplication. Furthermore, by compactness of $K$ and $S$, the $G$-invariant metric on $M$ is equivalent to a metric 
induced by a left-invariant metric on $G$; we write $L(g)$ for the function measuring the distance from the unit element $e\in G$. Then we have
\[
\mathcal{A}^{\exp}(G):=\left\{f\in C^\infty(G),~\sup_{g\in G}e^{q L(g)}|Df(g)|<C ~\forall q,~D\in U(\mathfrak{g})\right\}
\]

Finally, we introduce
\begin{equation}\label{ainfty-with-functions} 
\mathcal{A}^\infty_G(M,E):=\left(H^\infty_L(G)\hat{\otimes}\Psi^{-\infty}(S,E|_S)\right)^{K\times K},
\end{equation}
where 
\[
H^\infty_L(G):=\left\{f\in L^2(G),~ g\mapsto (1+L(g))^k f(g) \in L^2(G)~\forall k\right\}
\]
with $L$ the length function on $G$. This is a Fr\'echet space with seminorms 
\begin{equation}\label{nu-k}\nu_k (f):= \left( \int_G |f (g)|^2 
 (1+L(g))^{2k} dg \right)^{\frac{1}{2}}\,.
 \end{equation}

Next, recall that in the projective tensor product \eqref{algebra-slice} we can  write elements as $A=\sum_i\lambda_if_i\hat{\otimes}a_i$
with $\sum_i|\lambda_i|<\infty$ and $\{f_i\}$ and $\{k_i\}$ null sequences in $C^\infty_c(G)$ and $\Psi^{-\infty}(S)$. By defining
\[
\Phi_A(g):=\sum_i\lambda_if_i(g)a_i\in\Psi^{-\infty}(S),\quad g\in G,
\]
we obtain a map from $G$ to $\Psi^{-\infty}(S)$ invariant under $K\times K$ action:
\[
\Phi_A(k_1gk_2^{-1})=\rho(k_1)\circ\Phi(g)\circ\rho(k_2^{-1}),\quad k_1,k_2\in K,~g\in G.
\]
Under this identification, the product is defined by combining the composition in $\Psi^{-\infty}(S)$ with convolution:
\begin{equation}
\label{conv-op}
(\Phi_1*\Phi_2)(g)=\int_G\left(\Phi_1(gh^{-1})\circ\Phi_2(h)\right)dh.
\end{equation}
For the following, recall that $||T||_\alpha$ denotes the semi-norm of a smoothing operator $T\in\Psi^{-\infty}(S)$.
\begin{proposition}\label{prop:alternative-description}
The map $A\to \Phi_A$ induces isomorphisms
\begin{align*}
\mathcal{A}^c_G(M,E)&\cong\left\{\Phi:G\to \Psi^{-\infty}(S),~\mbox{smooth, compactly supported and }
~K\times K~\mbox{invariant}\right\},\\
\mathcal{A}^{\exp}_G(M,E)&\cong\left\{\Phi:G\to \Psi^{-\infty}(S), \text{ smooth,} ~K\times K~\mbox{invariant} ~and ~ \sup_{g\in G}e^{q L(g)}||D\Phi(g)||_{\alpha}<C ~\forall q, \alpha,k,~D\in U(\mathfrak{g})\right\},\\
\mathcal{A}^\infty_G(M,E)&\cong\left\{\Phi:G\to \Psi^{-\infty}(S),~K\times K~\mbox{invariant}~and~g\mapsto (1+L(g))^\ell||\Phi(g)||_{\alpha}\in L^2(G)~\forall \alpha,k,\ell\right\}
\end{align*}
Consequently
$$\mathcal{A}^c_G(M,E)\subset\mathcal{A}^{\exp}_G(M,E)\subset\mathcal{A}_G^\infty(M,E)\,.$$
\end{proposition}
\begin{proof}
First remark that the map $A\mapsto\Phi_A$ maps the algebraic tensor product $\mathcal{A}\otimes\Psi^{-\infty}(S)$, where $\mathcal{A}$ can be any
of the three choices $C^\infty_c(G)$, $\mathcal{A}^{\exp}(G)$ or $H^\infty_L(G)$, into the algebras on the right hand side of the proposition above.
It therefore suffices to show that the topologies on the algebras on the right hand side induce the topologies on the tensor products on the left hand side.
For the first isomorphism this follows from the well-known isomorphisms of Fr\'ech\`et spaces
\[
C^\infty_c(G\times S\times S)\cong C^\infty_c(G,C^\infty(S\times S))\cong C^\infty_c(G)\hat{\otimes}C^\infty(S\times S),
\]
and implementing $K\times K$ invariance. For the second algebra, we remark that the isomorphisms above still hold true when we drop the compact support, 
under which the families of seminorms 
\[
\sup_{x,y\in M}\left| e^{qd_M(x,y)}\nabla^{m}_{x}\nabla^{n}_{y}A(x,y)\right|\quad \mbox{and}\quad \sup_{g\in G}e^{qd_G(e,g)}||D\Phi_A(g)||_{\alpha,k}
\]
are equivalent. 

For the final isomorphism we observe that for $A=\sum_i f_i\otimes a_i\in H^\infty_L(G)\otimes\Psi^{-\infty}(S)$ we have
\begin{align*}
||\Phi_A||_{k,\ell}=||\sum_i f_i\otimes a_i||_{k,\ell} \leq \sum_i||f_i(1+L)^\ell||_{L^2(G)}||a_i||_k.
\end{align*}
Taking the infimum over all ways of writing $A=\sum_if_i\otimes a_i$, we conclude that 
$||\Phi_A||_{k,\ell}\leq \inf \sum_i ||f_i||_\ell||a_i||_k$. Because this infimum defines exactly the seminorms appearing in the projective tensor product, this shows
that the map $A\mapsto \Phi_A$ extends to the completion and gives exactly the third isomorphism. 
\end{proof}

\subsection{Index classes of $G$-equivariant Dirac operators}
Suppose now that $M$ is an even-dimensional closed manifold equipped with a proper, co-compact action of $G$. Let $D$ 
be an odd $\ZZ_2$-graded Dirac operator, equivariant with respect to the $G$-action.
Recall, first of all,  the classical Connes-Skandalis idempotent. Let $Q$ be a $G$-equivariant parametrix
of $G$-compact support with remainders $S_\pm$; consider the $2\times 2$ matrix
\begin{equation}\label{CS-projector}
P_{Q}:= \left(\begin{array}{cc} S_{+}^2 & S_{+}  (I+S_{+}) Q\\ S_{-} D^+ &
I-S_{-}^2 \end{array} \right).
\end{equation}
This produces a well-defined class 
\begin{equation}\label{CS-class}
\operatorname{Ind_c} (D):= [P_{Q}] - [e_1]\in K_0 (\mathcal{A}_G^c (M,E))\equiv 
K_0 (\mathcal{S}_G^c (M,E))\;\;\text{with}\;\;e_1:=\left( \begin{array}{cc} 0 & 0 \\ 0&1
\end{array} \right)
\end{equation}
%The $C^*$-index associated to $D$ is the class $\Ind_{C^*(M,E)} (D)\in K_0 (C^* (M,E)^G)$ obtained by taking the 
%image of the class  $\operatorname{Ind_c} (D)$  under the homomorphism  $\iota_* : K_* (\mathcal{A}_G^{c} (M,E))\rightarrow K_0 (C^* 
%(M,E)^G)$.

\begin{definition}
The $C^*$-index associated to $D$ is the class $\Ind_{C^*(M,E)} (D)\in K_0 (C^* (M,E)^G)$ obtained by 
considering $ [P_{Q}] - [e_1]$ as a formal difference of idempotents with entries in $C^* (M,E)^G$,
under the continuous inclusion $\iota : \mathcal{A}_G^{c} (M,E)\equiv \mathcal{S}^c_G (M,E)\hookrightarrow C^* (M,E)^G)$.\\
We shall also denote this class more simply by $\Ind_{C^*} (D)$.
%taking the 
%image of the class  $\operatorname{Ind_c} (D)$  under the homomorphism  $\iota_* : K_* (\mathcal{A}_G^{c} (M,E))\rightarrow K_0 (C^* (M,E)^G)$.
\end{definition}

\noindent
One can also give a definition of $\Ind_{C^*(M,E)} (D)\in K_0 (C^* (M,E)^G)$ using Coarse Index Theory, see
\cite{hr-book};
the compatibility of the two definitions is proved in \cite[Proposition 2.1]{PS-Stolz}.

\medskip
\noindent
There is another way of defining this class, due to Connes--Moscovici \cite{cm}, which uses the parametrix
\begin{equation}
 Q
:= \frac{I-\exp(-\frac{1}{2} D^- D^+)}{D^- D^+} D^+
\end{equation}
with
$I-Q D^+ = \exp(-\frac{1}{2} D^- D^+)$, $I-D^+ Q =  \exp(-\frac{1}{2} D^+ D^-)$. This particular choice of parametrix produces
 the idempotent
\begin{equation}
\label{cm-idempotent1}
V(D)=\left( \begin{array}{cc} e^{-D^- D^+} & e^{-\frac{1}{2}D^- D^+}
\left( \frac{I- e^{-D^- D^+}}{D^- D^+} \right) D^-\\
e^{-\frac{1}{2}D^+ D^-}D^+& I- e^{-D^+ D^-}
\end{array} \right)
\end{equation}
The following Proposition, well known to the experts, clarifies in which algebra
this idempotent lives. As we could not find a detailed proof, we supply one in the Appendix.
\begin{proposition}\label{prop:CM-rapid}
The idempotent $V(D)$  is an element in  $M_{2\times 2} (\mathcal{A}_G^{{\rm exp}} (M,E))$
(with the identity adjoined).
\end{proposition}

\noindent
We can thus define
\begin{equation}\label{exp-index}
\Ind_{{\rm exp}} (D)=[V_{D}]-[e_1]\in K_0 (\mathcal{A}_G^{{\rm exp}} (M,E))
\end{equation}
Using the inclusion $\iota: \mathcal{A}_G^{{\rm exp}} (M,E)\hookrightarrow C^*(M,E)^G$ we 
can express the $C^*(M,E)^G$-index class as the image of $[V_{D}]-[e_1]$ in  $K_0 (C^* (M,E)^G)$.
%then have 
%\begin{equation}\label{CM=index}
%\Ind_{C^*(M,E)^G} (D)=\iota_*  [V_{D}]-[e_1]\in K_0 (C^* (M,E)^G)
%\end{equation}
If $G$ satisfies the RD condition then we know that  there exists a continuous inclusion $ \mathcal{A}_G^{{\rm exp}}
 (M,E)\hookrightarrow \mathcal{A}^\infty (M,E)$. In this case, as $\mathcal{A}^\infty (M,E)$ is holomorphically closed
 in $C^*(M,E)$, we have
 \begin{equation}\label{CM=index-bis}
\Ind_{C^*(M,E)^G} (D)\equiv  \Ind_{\mathcal{A}^\infty (M,E)} (D)= [V_{D}]-[e_1]\in K_0 (\mathcal{A}^\infty (M,E))=
K_0 (C^* (M,E)^G)
\end{equation}
The advantage of using $\Ind_{\mathcal{A}^\infty (M,E)} (D)$ is that we can use Getzler rescaling
in order to prove an index theorem for the higher $C^*$-indices defined Subsection \ref{subsect:c*-indeces-closed}.

\begin{remark}
If $G$ is semisimple,
$M$ admits a G-equivariant spin$_c$ structure and $D$ is the spin$_c$ Dirac operator
with values in a G-equivariant vector bundle, the property that $V(D)$ is an element in $M_{2\times 2} (\mathcal{A}_G^{\infty} (M))$ also follows from
Hochs-Song-Tang \cite{HSTang} (building on Hochs-Wang \cite{Hochs-Wang-HC}), where  it is proved that the  Connes-Moscovici projector is
a $2\times 2$-matrix with entries   in the dense holomorphically closed subalgebra
of $C^* (M,E)^G$
constructed through the slice theorem using the Harish-Chandra subalgebra $\mathcal{C}(G)$ (and we know that $\mathcal{C}(G)$
is contained in $H^\infty_L (G)$ when $G$ is semisimple).
\end{remark}
 In the next section we shall see yet another approach 
to these results.

\subsection{Cyclic cocycles and the Van Est map}\label{section:van-est}$\;$\\
Here we briefly recall the construction of cyclic cocycles from smooth group cohomology \cite{ppt1} and the identification of the van Est map
as the pull-back along the classifying map $M\to G\slash K$ \cite{PP-akt}. 
For any Lie group $G$, denote by $(C^\bullet_{{\rm diff},\lambda}(G),\delta)$ the cochain complex of smooth, homogeneous group cochains
\[
C^k_{{\rm diff},\lambda}(G):=\{c\in C^\infty(G^{\times(k+1)},\mathbb{C}),~c(gg_0,\ldots,gg_k)=c(g_0,\ldots,g_k),~c(g_0,\ldots,g_k)=(-1)^kc(g_k,g_0,\ldots,g_{k-1})\},
\]
equipped with the standard differential. Also associated to $G$ is the cyclic cochain complex $C^\bullet_\lambda(C^\infty_c(G))$ of the convolution algebra $C^\infty_c(G)$. In \cite{ppt1} an explicit morphism
$C^k_{{\rm diff},\lambda}(G)\to C^\bullet_\lambda(C^\infty_c(G))$ of cochain complexes is described.

Now let $M$ be a closed smooth manifold equipped with a proper, cocompact action of $G$. As above, we denote by $\mathcal{A}^c_G(M)$ the algebra of $G$-invariant smoothing
operators with compact $G$-support, with associated cyclic cochain complex $C^\bullet_\lambda(\mathcal{A}^c_G(M))$. In \cite{PPT,PP-akt} morphisms are constructed as in the following commutative diagram: 
\[
\xymatrix{\Omega^\bullet_{\rm inv}(G\slash K)\ar[d]&&\ar[ll]C^\bullet_{{\rm diff},\lambda}(G)\ar[dll]\ar[drr]\ar[rr]&&\ar[d]C^\bullet_{\lambda}(C^\infty_c(G))\\
\Omega^\bullet_{\rm inv}(M)&&&&C_\lambda^\bullet(\mathcal{A}^c_G(M))}
\]
In this diagram, the triangle on the left gives the standard van Est theory involving the invariant differential forms on $M$ and $G\slash K$, where $K$ is a maximal compact subgroup of $G$. 
Explicitly, the left diagonal map is given as follows. Fix a cut-off function $\chi$ for the 
action of $G$, that is, $\chi\in C^\infty_c (M)$ and $\int_G \chi (g^{-1} x) dg=1$ $\forall x\in M$. Associated to a cochain $\varphi\in C^k_{{\rm diff},\lambda}(G)$ is a smooth function $f^\varphi\in C^\infty(M^{\times(k+1)})$ given by
\begin{equation}
\label{ac}
f^\varphi(x_0,\ldots,x_k):=\int_{G^{\times(k+1)}}\chi(g_0^{-1}x_0)\cdots\chi(g_k^{-1}x_k)\varphi(g_0,\ldots,g_k)d\mu(g_0)\cdots d\mu(g_k).
\end{equation}
With this notation, the invariant differential form $\omega^\chi_\varphi\in\Omega^k_{\rm inv}(M)$ associated to $\varphi$ is defined as
\[
\omega^\chi_\varphi:=(d_1\cdots d_k f^\varphi)|_\Delta
\]
where $d_i$ means taking the differential in the i'th variable, and $\Delta:M\to M^{\times(k+1)}$ is the diagonal embedding.

The other diagonal map is given by writing, for convenience only, $f^\varphi=f^\varphi_0\otimes\ldots\otimes f^\varphi_k$ 
 and writing down 
the following cyclic cocycle on $\mathcal{A}^c_G(M)$:
\begin{equation}
\label{af}
\tau^M_\varphi (A_0,\ldots,A_k):={\rm Tr}_\chi (f^\varphi_0\cdot A_0\cdots f^\varphi_k\cdot A_k),
\end{equation}
where ${\rm Tr}_\chi:\mathcal{A}^c_G(M)\to\mathbb{C}$ is the trace given by
\begin{equation}
\label{trace}
{\rm Tr}_\chi (A):=\int_M\chi(x)A(x,x) dx.
\end{equation}
Let us now rewrite this cyclic cocycle in a more convenient form, using the global slice $G\times_KS\cong M$ which identifies elements
 $A\in \mathcal{A}^c_G(M)$ with $K\times K$ equivariant maps $\Phi_A:G\to\Psi^{-\infty}(S)$ by defining
 \[
 \Phi_A(g,s_1,s_2):=A(s_1,gs_2).
 \]
Using, as in \cite[Lemma 2.1]{PP-akt}, 
a family of cut-off functions $\{\chi_\epsilon\}_{\epsilon>0}$ which converges in the limit $\epsilon\downarrow 0$ to the characteristic 
function on $S$ in the distributional sense, the trace \eqref{trace} can be written as
\[
{\rm Tr}_\chi (A)=\lim_{\epsilon\downarrow 0}\int_M\chi_\epsilon(x)A(x,x) dx={\rm Tr}_S(\Phi_A(e)),
\]
where ${\rm Tr}_S$ denotes the trace on $\Psi^{-\infty}(S)$, given by integration of the kernel over the diagonal.
We can work out the cyclic cocycle $\tau_\varphi^M$ in \eqref{af} in similar fashion to obtain:
\begin{align*}
\tau^M_\varphi (A_0,\ldots,A_k)&=\int_{G^{\times k}}{\rm Tr}_S
\left(\Phi_{A_0}((g_1\cdots g_k)^{-1})\circ \Phi_{A_1}(g_1)\circ \ldots\circ \Phi_{A_k}(g_k)\right)
\varphi(e,g_1,g_1g_2,\ldots,g_1\cdots g_k)dg_1\cdots dg_k
\end{align*}
When $M=G$, i.e., $S=\{pt.\}$, this formula reduces to the usual cyclic cocycle 
\[
\tau^G_\varphi(a_0,\ldots,a_k):=\int_{G^{\times k}}a_0((g_1\cdots g_k)^{-1})a_1(g_1)\cdots a_k(g_k)\varphi(e,g_1,g_1g_2,\ldots,g_1\cdots g_k)dg_1\cdots dg_k.
\]
on the convolution algebra $C^\infty_c(G)$. This association $\varphi\mapsto\tau^G_\varphi$ describes the horizontal map in the upper right corner of the diagram.

\subsection{Getzler rescaling}\label{PPT-PP}
$\;$\\
We now return to the setting of a $G$-equivariant %spin-
Dirac operator $D$ on an even dimensional manifold. Let us consider
the index class $\Ind_c (D)\in K_0 (\mathcal{A}_G^c(M))$ defined for example through the Connes-Skandalis
projector associated to a $G$-compact parametrix. Given a smooth group cocycle $\varphi$ of degree $2p$, we 
can pair the associated cyclic cocycle $\tau^M_\varphi\in C^{2p}(\mathcal{A}_G^c(M))$ with the index class
 $\Ind_c (D)$ and so define the higher index
\[
{\rm Ind}_{c,\varphi}(D):=(-1)^p \frac{2p !}{p!}\,\left<\tau_\varphi^M,{\rm Ind}_c (D)\right>.
\]
It is important to notice that this is {\em not} a $C^*$-higher index. In \cite{ppt1},  a higher index theorem is proved giving a topological formula for the outcome of this pairing using the algebraic index theorem in deformation quantization.  Here we shall give
a proof using heat kernels and Getzler rescaling which is better adapted for the case of manifolds with boundary that we are aiming for. We shall treat explicitly the spin case but the proof can be generalized to any Dirac
operator (associated to a unitary Clifford action and to a Clifford connection) in the usual fashion.
%\textcolor{cyan}{FROM PAOLO: but in order to make sense of the pairings appearing below,
%don't we need to extend 
%the  triangular diagram of the previous page to $\mathcal{A}_G^{{\rm exp}} (M,E)$ ?
%In particular should not we make  an assumption on the growth of  $c\in C^{2k}_{{\rm dif},\lambda}(G)$, for example 
%that it is at most of exponential growth  ?}

\begin{theorem}\label{theo:short-time}
Let $\varphi\in C^{2p}_{{\rm diff},\lambda}(G)$, and consider the Connes--Moscovici projector $[V(D)]-[e_1]\in K_0 (\mathcal{A}_G^{{\rm exp}} (M,E))$. Assume that the associated cyclic cocycle $\tau_\varphi^M$ 
extends from $\mathcal{A}_G^{{\rm c}} (M,E)$ to $\mathcal{A}_G^{{\rm exp}} (M,E)$.
Then the following identities hold true:
\[
\left<\tau^M_\varphi, [V(D)]-[e_1]\right>=\lim_{t\downarrow 0}\left<\tau^M_\varphi,[V(tD)]-[e_1]\right>=
\frac{p!}{2p !}\frac{(-1)^p}{(2\pi i)^p} \,\int_M \chi\hat{A}(M)\wedge\omega_\varphi.
\]
%with $c_p=\frac{(-1)^{\dim M/2}}{(2\pi i)^p} \frac{p !}{2p!}$.
%in \eqref{as-to-dr}.\\
%\textcolor{cyan}{FROM PAOLO: as already remarked, we have not really defined this pairing; of course we could simply make the assumption that $\tau_c $ extends from $\mathcal{A}_G^{{\rm c}} (M,E) $ to $\mathcal{A}_G^{{\rm exp}} (M,E)$, this would be more in line with the higher $C^*$ index theorem....}
\end{theorem}

\begin{proof}
The proof of this theorem follows by adapting the original proof of the localized index formula in \cite{cm} to our setting. As in that proof, we use Getzler's symbol calculus for 
pseudodifferential operators acting on spinor bundles, now in the invariant setting. In particular, Getzler's fundamental trace formula \cite[Thm 3.7]{Getzler} for the trace of an invariant 
smoothing operator acting on the sections of the spinor bundle now becomes:
\[
{\rm Tr}^\chi_s(P)=\frac{1}{(2\pi)^{\dim M}}\int_{T^*M}\chi(x){\rm tr}_s(\sigma_{t^{-1}}(P))(x,\xi) dxd\xi,
\]
where $\sigma_{t^{-1}}(P)$ denotes the rescaled Getzler symbol and ${\rm tr}_s:{\rm Cliff}(T_xM) \to\mathbb{C}$ denotes the Berezin trace. The appearance of the cut-off function $\chi(x)$ is 
in accordance with the formula \eqref{trace} for the trace on the algebra $\mathcal{A}_G^c(M)$ of smoothing operators.

As in \cite[\S 3]{cm} we can write the pairing in terms of this trace as
\[
\lim_{t\downarrow 0}\left<\tau_\varphi,[V(tD)]-[e_1]\right>=\lim_{t\downarrow 0}{\rm Tr}^\chi_s(\Pi(t))=\frac{1}{(2\pi)^{\dim M}}\int_{T^*M}\chi(x){\rm tr}_s(\sigma_{t^{-1}}(\Pi(t)))(x,\xi) dxd\xi,
\]
with $\Pi(t)$ the invariant smoothing operator as in \cite{cm} for our specific Alexander--Spanier cochain $f^\varphi$ given in \eqref{ac}.  For each point $x\in M$, the evaluation of the quantity
\[
\int_{T^*_xM} {\rm tr}_s(\sigma_{t^{-1}}(\Pi(t)))(x,\xi)d\xi
\]
proceeds as in \cite[\S 3]{cm} to give 
\begin{align*}
\frac{1}{(2\pi)^{\dim M}}\int_{T^*M}\chi(x){\rm tr}_s(\sigma_{t^{-1}}(\Pi(t)))(x,\xi) dxd\xi&=\frac{(-1)^p}{(2\pi i)^p} \frac{p !}{2p!}\,\int_M \chi\hat{A}(M)f_0^\varphi df^\varphi_1\wedge df^\varphi_k\\&=\frac{(-1)^p}{(2\pi i)^p} \frac{p !}{2p!}\,\int_M \chi\hat{A}(M)\wedge \omega_\varphi.
\end{align*}
This completes the proof.
\end{proof}

\subsection{The $C^*$-index theorem}\label{subsect:c*-indeces-closed}
In this subsection we will elaborate on  the $C^*$-index theorem that we have established in Part 1.  Our presentation
will be crucial later in the paper, when we shall pass to manifolds with boundary.

\begin{theorem}\label{c*-index-th-part-1}
Assume  that
$G$ has finitely many connected components and satisfies the RD condition.
Let $\varphi\in Z^{2p}_{{\rm dif},\lambda}(G)$ be a cocycle of polynomial growth
and consider the $C^*$-index class $\Ind_{C^*(M,E)} (D)\in K_0 (C^* (M,E)^G)$.
Then the homomorphism  $\langle \tau^M_\varphi,\cdot \rangle: K_0 (\mathcal{A}^c_G (M,E))\to \CC$ defined by the
pairing $HC^* (\mathcal{A}^c_G (M,E))\times K_0 ((\mathcal{A}^c_G (M,E))\to \CC$
extends to a homomorphism $\langle \tau^M_\varphi,\cdot \rangle: K_0 (\mathcal{A}^\infty_G (M,E))= K_0 (C^* (M,E)^G)\to \CC$
and for the higher $C^*$-indices
\begin{equation}\label{def-higher-c*-index}
\Ind_\varphi (D):= (-1)^p \frac{2p !}{p!}\left<\tau^M_\varphi, \Ind_{C^*(M,E)} (D) \right>
\end{equation}
the following $C^*$-higher index formula holds:
\begin{equation}\label{index-formula}
%=\lim_{t\downarrow 0}\left<\tau_c,[V(tD)]-[e_1]\right>
\Ind_\varphi (D)=\frac{1}{(2\pi i)^p}\,\int_M \chi\hat{A}(M)\wedge
{\rm Ch}^\prime (E)\wedge\omega_\varphi.
\end{equation}
\end{theorem}
\begin{proof}
For notation convenience we expunge $E$ from the notation.
We shall freely use the alternative description of $\mathcal{A}^c_G(M)$
and $\mathcal{A}^\infty_G(M)$
given in Proposition \ref{prop:alternative-description}:
\begin{equation*}
\mathcal{A}^c_G(M)=\left\{\Phi:G\to \Psi^{-\infty}(S),~\mbox{smooth, compactly supported and }
~K\times K~\mbox{invariant}\right\}\,,
\end{equation*}
where smoothness is with respect to the seminorms $||~||_{\alpha}$ defining the Fr\'echet structure of $\Psi^{-\infty}(S)$, and
\begin{equation*}
\mathcal{A}^\infty_G(M):=\left\{\Phi:G\to \Psi^{-\infty}(S),~K\times K~\mbox{invariant} \text{ such that }~g\mapsto (1+L(g))^k
p_{\alpha,k}(\Phi(g))\in L^2(G)~\forall \alpha,k\right\}.
\end{equation*}

\begin{lemma}\label{key-lemma-closed}
The map
$\Phi\to ||\Phi(\cdot)||_1$ defines a continuous application 
$\mathcal{A}^\infty_G(M)\to H^\infty_L (G)$. 
\end{lemma}

\begin{proof}
We are considering the composition
$$G\xrightarrow{\Phi} \Psi^{-\infty}(S)\xrightarrow{||\cdot ||_1}\RR$$
and so it suffices to prove that
the map $\Psi^{-\infty}(S)\ni A\to ||A||_1\in\RR$ is continuous. This is well-known but we present the
argument nevertheless. 
%From the Sobolev embedding theorem it is easy to see that 
%an equivalent set of seminorms for the Fr\'echet structure of $\Psi^{-\infty}(S)$ is given by
%the norms $||A_\kappa ||_{L^2\to H^M}$, $M>\dim S/2$, for the 
Consider the operator $A_\kappa$ 
defined by a smoothing kernel  $\kappa$
and  write $A_\kappa$ as $(1+\Delta)^{-M} \circ ((1+\Delta)^{M} A_\kappa)$ for $M>\dim S/2$,
so that
$$||A_\kappa ||_1 \leq || (1+\Delta)^{-M}||_1 ||((1+\Delta)^{M} A_\kappa)||\leq  
|| (1+\Delta)^{-M}||_1 ||((1+\Delta)^{M} A_\kappa)||_2 $$
where the last norm is the Hilbert-Schmidt norm. The trace norm of 
$(1+\Delta)^{-M}$ is a constant whereas we can obviously bound the second factor  on the right hand side 
by  one of the
seminorms defining the  Fr\'echet structure of $\Psi^{-\infty}(S)$;
indeed, the kernel of $(1+\Delta)^{M} A_\kappa$ is $L(x,y):= (1+\Delta_x )\kappa (x,y)$
and we know that $||((1+\Delta)^{M} A_\kappa)||_2 =\int |L(x,y)|^2 dx dy$.
The Lemma is proved.\\
%\textcolor{cyan}{FROM PAOLO: equivalently, we could try to prove that 
%$$\Psi^{-\infty}(S)\ni\kappa \to \Tr (|S_\kappa|)(\equiv ||S_\kappa ||_1)$$
%is continuous. However, I am not sure who is the kernel associated to $|S_\kappa |$?
%Notice that from $||AB||_1 \leq ||A||_2 || B ||_2$ it follows that we could also consider
%the map
%$\Phi\to ||\Phi(\cdot)||_2$ from  
%$\mathcal{A}^\infty_G(M)\to H^\infty(G)$.\\ See the proof below where using the above inequality we would end
%up with $\tau_{|c|}^G ( || \Phi_0 (\cdot )||_2,  \cdots  ,||\Phi_{k-1} (\cdot)||_2)$. This remark 
%could be useful in the b-case.
%}
\end{proof}
We go back to the proof of the Theorem. Let $\Phi_0,\Phi_1,\dots,\Phi_k\in  \mathcal{A}^c_G(M)$. We want to estimate $|\tau_\varphi^M (\Phi_0,\dots,\Phi_k)|$, that is, the absolute value of 
$$\int_{G^{k}} \int_{S^{k+1}} \Phi_0 ((g_1 \cdots g_k)^{-1})(s_0,s_1)\Phi_{1} (g_{1})(s_{1},s_2)\cdots 
\Phi_k (g_k) (s_k,s_0)  \,\varphi(e,g_1,g_1 g_2,\dots,g_1 g_2 \dots g_{k}) ds_0 \cdots ds_k dg_1\cdots dg_k
$$
Bringing the absolute value under the sign of integral, using Lidski's theorem,
$\Tr K=\int_S K(s,s)ds$, and the well-known estimate 
$| \Tr (K) |\leq ||K ||_1$ we have
\begin{align*}
 |\tau_\varphi^M (\Phi_0,\dots,\Phi_k)| &\leq 
\int_{G^{k}} || \Phi_0 ((g_1 \cdots g_k)^{-1})\circ \Phi_1(g_1)(s_1,s_2)\cdots \circ  \Phi_k (g_k)||_1  \,|\varphi(e,g_1,g_1 g_2,\dots,g_1 g_2 \dots g_{k})|  dg_1\cdots dg_k\\
&\leq \int_{G^{k}} || \Phi_0 ((g_1 \cdots g_k)^{-1})||_1 ||\Phi_1(g_1)||_1 \cdots   ||\Phi_k ((g_1 \cdots g_k)^{-1})||_1 \,|\varphi(e,g_1,g_1 g_2,\dots,g_1 g_2 \dots g_{k})|  dg_1\cdots dg_k\\
&= \tau_{|\varphi|}^G ( || \Phi_0 (\cdot)||_1,  \cdots  ,||\Phi_{k} (\cdot)||_1)
\end{align*}
Now, 
under the present assumptions we have proved in Part 1 that 
$$\tau_{|\varphi|}^G ( || \Phi_0 (\cdot )||_1,  \cdots  ,||\Phi_{k} (\cdot)||_1)\leq C \nu_{p+\ell} ( || \Phi_0 (\cdot)||_1)
\cdots \nu_{p+\ell} ( || \Phi_{k} (\cdot)||_1$$
with $p$ given by the hypothesis of polynomial growth of $\varphi$ and $\ell$ given by the RD assumption.
 See Proposition 5.6
in Part 1. 
Given $\Phi\in  \mathcal{A}^c_G(M)$, let us  introduce the following norm:
\begin{equation}\label{norm-31}
||| \Phi |||^2_m := \int_G ||\Phi (g)||^2_{1} (1+L(g))^{2m} dg \,.
\end{equation}
%\red{FROM PAOLO: should we use
%\begin{equation}
%||| \Phi |||^2_m := \int_G ||\Phi (g)||_{1} (1+L(g))^{2m} dg \,.
%\end{equation}
%}
Notice that
$\nu_{p+\ell} ( || \Phi (\cdot)||_1)= ||| \Phi |||_{p+\ell}$.
%with $||\cdot ||_1$ denoting the trace norm. 
The expression $\int_G ||\Phi (g)||^2_{1} (1+L(g))^{2m} dg $ can also be written  for $\Phi\in  \mathcal{A}^\infty_G(M)$
%and the following lemma holds:
%If $\Phi\in \mathcal{A}^\infty_G(M)$ 
%then %there exists 
%for each $m\in\NN$ we have that $|||\Phi |||_m <\infty$.
%Consequently  %there exists 
%given $m\in\NN$ %such that 
%we have that 
%$$\mathcal{A}^\infty_G(M) \subset \overline{\mathcal{A}^c_G(M)}^{|||\cdot|||_m}\,.$$
and we know from the Lemma that
if $\Phi\in \mathcal{A}^\infty_G(M)$ 
then %there exists 
for each $m\in\NN$ we have that $|||\Phi |||_m <\infty$.
Consequently  %there exists 
given $m\in\NN$ %such that 
we have that 
\begin{equation}\label{inclusion-closed}
\mathcal{A}^\infty_G(M) \subset \overline{\mathcal{A}^c_G(M)}^{|||\cdot|||_m}\,.
\end{equation}
Since we have proved that 
$$ |\tau_\varphi^M (\Phi_0,\dots,\Phi_k)| \leq C ||| \Phi_0 |||_{p+\ell} \cdots  ||| \Phi_k |||_{p+\ell}$$
we conclude, finally, that $\tau_\varphi^M$ extends continuously to $\mathcal{A}^\infty_G(M)$.\\
The index formula follows immediately from \eqref{CM=index-bis} and  Theorem \ref{theo:short-time}.
\end{proof}

\section{More on heat kernels and holomorphically closed subalgebras}\label{sect:more-on-heat}

We have proved in Proposition \ref{prop:CM-rapid} that the Connes-Moscovici projector $V(D)$ 
is an element in $M_{2\times 2} (\mathcal{A}_G^{{\rm exp}} (M))$. If
$G$ satisfies the RD condition, then, from the inclusion $\mathcal{A}_G^{{\rm exp}} (M))\subset \mathcal{A}_G^{\infty} (M)$, we 
have deduced that
$V(D)$ is an element in $M_{2\times 2} (\mathcal{A}_G^{\infty} (M))$ and that it is a representative of the $C^*$-index class.
%with $\mathcal{A}_G^{\infty} (M)$ the dense holomorphically closed subalgebra associated through the slice theorem
%to $H^\infty_L (G)$.

In this section we want to present a different approach to these results, based on a direct analysis 
of the right hand side of 
\begin{equation}\label{heat-operator}\exp (-t D^2)=\frac{1}{2\pi i}\int_\gamma e^{-t\lambda} (D^2 - \lambda)^{-1} d\lambda
\end{equation}
This analysis has the advantage of extending easily to manifolds with cylindrical ends which is why we present
it now in the closed case. Moreover,
some of the results needed in our discussion will be used crucially later in the paper
and are, in our opinion, of independent interest.

\medskip
\noindent
For some of our results, as in Part 1, we shall consider 
any  Fr\'echet algebra   $\mathcal{A}(G)$ with a  continuous inclusion
$ C^\infty_c (G)\subset \mathcal{A}(G)$ and with the following two properties:\\
(i) there exists 
 %action by convolution defines  
 a continuous injective map
$\mathcal{A}(G) \hookrightarrow  C^*_r (G)$ which makes  $\mathcal{A}(G)$ a  subalgebra of $C^*_r (G)$;\\
%\item $\mathcal{A}^{{\rm exp}}(G)\subset \mathcal{A}(G)$, where $\mathcal{A}^{{\exp}} (G)$ is the
%algebra of  smooth functions on $G$ of rapid exponential decay
(ii) $\mathcal{A}(G)$ is holomorphically closed
in $C^*_r (G)$\\
%\item  if $f\in \mathcal{A}(G)$ then its action as a convolution operator is bounded on $L^2 (G)$.
We have called such Fr\'echet algebras {\it admissible}.\\
At some point however, we shall assume that $G$ satisfies the RD condition and consider 
the particular admissble algebra $H^\infty_L (G)$.

\subsection{The inverse of a $L^2$-invertible Dirac operator.}\label{subsect:resolvant}$\;$

\noindent
Let $D$ be a $G$-equivariant symmetric Dirac-type operator on a closed proper $G$-manifold $M$. We denote by $E$ the corresponding Clifford module (but in order to lighten the notation, we shall often forget about $E$ in the notation).
 It is well known that $D$ is essentially self-adjoint.  We still denote
by $D$ its unique self-adjoint extension. We keep identifying an algebra of smooth kernels (under convolution)
with the corresponding algebra of integral operators (under composition).

%\begin{definition}
%Let $\mathcal{A}(G)\subset C^*_r G$ be an admissible Fr\'echet algebra.
%We shall additionally assume that 
% $\mathcal{A}(G)\subset C^\infty (G)$.
% Let $\mathcal{A}_G (M)$ be the corresponding
%Fr\'echet algebra of smoothing operators.
%\footnote{\textcolor{cyan}{From Paolo: I am not sure the requirement  $\mathcal{A}(G)\subset C^\infty (G)$
%is really necessary;
% I am adding this condition so at to have true
% smoothing operators but  we could get away with less, for example that 
%  $\mathcal{A}_G (M)$ is a left and right module over the compactly supported pseudos. See the remark below.
%The condition $\mathcal{A}(G)\subset C^\infty (G)$ is satisfied by the HC algebra $\mathcal{C}(G)$ 
% but not, as far as I know, by $H^\infty_L (G)$. To check: assumptions on $G$ needed in order to prove
% that  the HC algebra $\mathcal{C}(G)$ is holomorphically closed. G semisimple OK but more generally ??}}.
%The algebra $\Psi_{G} (M)$ is the subalgebra of $C^* (N)^G$ generated by %$\cup_{m\leq -1}\Psi^m_{G,p} (N)$,
%$\cup_{m\leq -1}\Psi^m_{G,c} (M)$,
%the $G$-equivariant pseudodifferential operators of order $\leq -1$ of $G$ compact support, and the algebra 
%$\mathcal{A}_G (M)\subset C^* (N)^G$.\\
%
%\end{definition}
%
%Notice that 
%$\mathcal{A}_G (M)$ is in a natural way a left and right module on $\cup_{m\leq -1}\Psi^m_{G,c} (M)$; 
%thus $\Psi_{G} (M)$ can be identified with the subalgebra of   $C^* (N)^G$
% equal to the vector sum of 
%$\cup_{m\leq -1}\Psi^m_{G,p} (M)$ and $\mathcal{A}_G (M)$ in $C^* (N)^G$.
%
\begin{proposition}\label{prop:inverse-0}
Let $\mathcal{A}(G)\subset C^*_r G$ be an admissible Fr\'echet algebra.
We shall additionally assume that 
 $\mathcal{A}(G)\subset C^\infty (G)$.
 Let $\mathcal{A}_G (M)$ be the corresponding
Fr\'echet algebra of smoothing operators.
Assume  that
$D$ is $L^2$-invertible. Then 
%$$D^{-1}\in\Psi_{G} (M)\subset C^* (M)^G \,.$$
%More precisely
$$D^{-1}=Q+ Q\circ A$$
with $Q\in \Psi^{-1}_{G,c} (M)$ and $A\in \mathcal{A}_G (M)$, from which it follows that 
$D^{-1}=Q+ C$  with $Q\in \Psi^{-1}_{G,c} (M)$ and $C\in \mathcal{A}_G (M)$.
\end{proposition}

%\begin{proposition}\label{prop:resolvant}
%Let $D$ be as above and consider $D+i\lambda$, with $\lambda\in\RR$. Assume  that
%$D+i\lambda$, $\lambda\in\RR$, is $L^2$-invertible. We have that
%$$(D+i\lambda)^{-1}\in\Psi_{G} (N)\,.$$
%\end{proposition}

\begin{proof}
%\blue{See handwritten notes of June 20 2019 and blackboards of June 19th 2019.}\\
%Following the proof of Proposition 19 in \cite{Lott-torsion} and using a parametrix
%for $D+i\lambda$ we reduce ourselves  to proving that
%if $T\in \mathcal{A}_G (N)$  is such that
%$\Id-T$ is invertible in the unitalization of $C^*(N)^G$ then $\Id-T$   is invertible 
%in the unitalization of $\mathcal{A}_G(N)$. This should follow the fact
%that $\mathcal{A}_G(N)$ is holomorphically closed. 

%\noindent
%\textcolor{m}{$\bullet$$\longrightarrow$ to be made more precise}
 Let $f:\RR\to \RR$ be an odd smooth  bounded function equal to $1/x$ on the $L^2$-spectrum
of the invertible self-adjoint operator $D$. Clearly $D^{-1}=f(D)$, from which it follows that $D^{-1}\in C^* (N)^G$. We want to show
that, in fact, $D^{-1}$ has the particular structure given in the statement of the Proposition.
 %Notice that $f$, and thus its Fourier transform $\hat{f}$, lies in
 %$L^2(\RR)$.  
 First, following closely  \cite{PS-Stolz},
 we have the following 
 
 \begin{lemma}
 For each $\epsilon >0$ we can find a decomposition of smooth functions
 $$f=h_\epsilon + w_\epsilon$$
 with the following properties:\\
 (i) $h_\epsilon$ is an element of $S^{-1}(\RR)$, the symbols of order $-1$, and its Fourier 
 transform $\hat{h}_\epsilon$ is compactly supported;\\
 (ii) $w_\epsilon\in \mathcal{S}$ and  has the
 property that $|x w_\epsilon(x)|_{\infty}<\epsilon$.
 \end{lemma}
 For a detailed proof, see  Proposition   4.18 in \cite{PS-Stolz}.
 
Let now $Q:= h_\epsilon (D)$. Then, from  \cite[Theorem XII.1.3]{taylor}, and the properties of $h_\epsilon$, 
we know that $Q\in \Psi^{-1}_{G,c} (M)$
(see again Proposition   4.18 in \cite{PS-Stolz} for providing the  details that are necessary to
pass from the statement in  \cite{taylor} to the statement given here);
moreover 
\begin{equation}\label{para}
D \circ Q= D\circ (f(D)-w_\epsilon (D))= D\circ f(D)- D\circ w_\epsilon (D)=\Id - (xw_\epsilon)(D)\,.
\end{equation}
Let $\phi_\epsilon (x) : = x w_\epsilon (x)$.
From above we infer that  $Q$, which is of $G$-compact support, is a right parametrix with smoothing remainder $\phi_\epsilon (D)$
and with the norm of the reminder $\phi_\epsilon (D)$  less then $\epsilon$. 
Notice that from \eqref{para} we infer that $(xw_\epsilon)(D)\equiv \phi_\epsilon (D)$, a smoothing operator,
has  $G$-compact support, because 
$$\phi_\epsilon (D)=\Id - D\circ Q$$
and on the right we have a pseudodifferential operator of $G$-compact support. In particular  $\phi_\epsilon (D)\in A_G (M)$. Since $\phi_\epsilon (D)$
has small operator norm, we see  that $\Id - \phi_\epsilon(D)$ is $L^2$-invertible.
If we could show that 
$$(\Id - \phi_\epsilon(D))^{-1}=\Id + A\;\;\text{with}\;\; A\in A_G (M),$$ then
$$ D^{-1} = Q + Q\circ A$$
finishing the proof.
However, we know that $A_G (M)$ is holomophically closed in $C^* (N)^G$; thus 
if $U\in A_G (N)$ and $\Id + U$ is invertible  in $C^* (M)^G$, then $(\Id + U)^{-1}$
is an element of  type $\Id + A$ with $A\in A_G (M)$, see Lauter-Monthubert-Nistor
\cite[(2.2)]{LMN-olo}.
Since we have observed that $ \phi_\epsilon(D)\in A_G (M)$, we conclude  that 
$(\Id - (\phi_\epsilon)(D))^{-1}=\Id + A$ with $A\in A_G (M)$, as required. 
 \end{proof}

\begin{remark}
We have required in Proposition \eqref{prop:inverse-0} that $\mathcal{A}(G)\subset C^\infty (G)$
so as to have $\mathcal{A}_G (M)$ made of integral operators with  smooth kernel.
We can weaken the assumption 
 $\mathcal{A}(G)\subset C^\infty (G)$  and simply  require that the resulting algebra
 of integral operators 
  $\mathcal{A}_G (M)$ is a left and right module over $\Psi^{-1}_{G,c} (M)$.%, with $m\leq -1$.
\end{remark}

\begin{remark}
This particular structure of the inverse of an $L^2$-invertible Dirac operator will be used
in our analysis on manifolds with boundary.
\end{remark}

\subsection{The heat operator}\label{subsect:heat}$\;$\\
Assume that $G$ satisfies the RD condition.
%\blue{We keep considering the problem of showing that the Connes-Moscovici projector
%is an element in the unitalization of $\mathcal{A}^\infty_G (M,E)$. The advantage of the 
%method explained below is that it can be generalized to manifolds with cylindrical ends.}
%
We now consider the heat operator $\exp (-t D^2)$, which we write as
$$\exp (-t D^2)=\frac{1}{2\pi i}\int_\gamma e^{-t\lambda} (D^2 - \lambda)^{-1} d\lambda,$$
with $\gamma$ a suitable path in the complex plane missing the spectrum of $D^2$. 
%For example
%for $C<0$, $C$  real and close to $0$ we can take $\gamma$ as the union of the straight lines
%${\rm Re}(\lambda+C+1)=\pm {\rm Im}(\lambda)$.
We want to show that the bounded operator on the RHS is in fact an element in 
 $\mathcal{A}^\infty_G(M,E)$.
% ; in fact, in any algebra $\mathcal{A}_G (M,E)$ associated an
% admissible Fr\'echet algebra $\mathcal{A}(G)$. 
To this end we shall need the following technical 
\begin{lemma}\label{lemma:composition}
%The convolution product on $\mathcal{A}^\infty_G(M)$ 
The composition $\Psi^0_{G,c}(M)\times \Psi^{-\infty}_{G,c}(M)\to  \Psi^{-\infty}_{G,c}(M)$
extends to a continuous map
\[
\Psi^0_{G,c}(M)\times \mathcal{A}^\infty_G(M)\to \mathcal{A}^\infty_G(M).
\]
\end{lemma}
\begin{proof}
Let $S\subset M$ be a global slice and consider the definition 
\[
\mathcal{A}^\infty_G(M):=\left(H^\infty_L(G)\hat{\otimes}\Psi^{-\infty}_c(S)\right)^{K\times K}
\]
of the algebra $\mathcal{A}^\infty_G(M)$. The diffeomorphism $G\times_KS\cong M$ likewise
induces an isomorphism
\[
\Psi^0_{G,c}(M)\cong\left(\Psi^0_{G,c}(G)\hat{\otimes}\Psi^0_c(S)\right)^{K \times K}.
\]
Here $G$ is regarded as a $G$-space using the action by left translations, so that $\Psi^0_{G,c}(G)$ is the algebra of invariant Pseudo-differential operators on $G$ first described in \cite{cm-homogeneous}.
Enlarging $\mathcal{A}_{G,c}(M)$ to $\mathcal{A}^\infty_G(M)$ only affects the first component over $G$, and it is a basic fact that $\Psi^{-\infty}_c(S)\subset \Psi^0_c(S)$ is an ideal, so we only have to show
that 
\[
\Psi^0_{G,c}(G)\times H^\infty_L(G)\to H^\infty_L(G)
\]
is well-defined and continuous. Recall that a general element in $A\in \Psi^0_{G,c}(G)$ can be written as
\[
A={\rm Op}(a)+K,
\]
with $K\in\Psi^{-\infty}_{G,c}(G)\cong C^\infty_c(G)$ and ${\rm Op}(a)$ is the operator corresponding to a symbol $a\in S^0(\mathfrak{g}^*)$ of order zero, i.e., a smooth function on the dual $\mathfrak{g}^*$ of the Lie algebra $\mathfrak{g}$
 satisfying the symbol estimates $|D^\alpha_\xi a(\xi)|\leq C_\alpha(1+|\xi|)^{-|\alpha|}$. Concretely, ${\rm Op}(a)$ is the operator given by
\[
\left({\rm Op}(a)(f)\right)(g):=\int_G\int_\mathfrak{g^*}\chi(gh^{-1})e^{i\left<\xi,\exp^{-1}(gh^{-1})\right>}a(\xi)d\xi dh,\quad\mbox{for}~ f\in C^\infty_c(G),
\]
where $\chi\in C^\infty_c(g)$ is a cut-off function around the unit $e\in G$.
To verify that composition with $A$ lands in $H^\infty_L(G)$, we first consider the smoothing part $K\in\Psi^{-\infty}_{G,c}(G)$ for which the statement follows automatically
from the fact that $C^\infty_c(G)\subset H^\infty_L(G)$. For the first operator ${\rm Op}(a)$ we use Peetre's inequality
\[
(1+L(gh))^k\leq(1+L(g))^{|k|}(1+L(h))^k,
\]
to get the estimate
\begin{align*}
||{\rm Op}(a)\star f||_{2,k}&\leq ||((1+L)^{|k|}{\rm Op}(a))\star((1+L)^k f)||_2\\
&\leq ||(1+L)^{|k|}{\rm Op}(a)||_{B(L^2(G)}||f||_{2,k}.
\end{align*}
Here the last inequality follows from the fact that $(1+L)^kf\in L^2(G)$, and that the operator $(1+L)^{|k|}{\rm Op}(a)$ with integral kernel
\[
(1+L(g))^{|k|}\chi(g)\int_\mathfrak{g^*}e^{i\left<\xi,\exp^{-1}(gh^{-1})\right>}a(\xi)d\xi,
\]
is just an invariant pseudodifferential operator of order $0$, which act by bounded operators on $L^2(G)$.  
\end{proof}

\begin{remark}
Exactly the same proof establishes the fact that $\mathcal{A}^\infty_G(M)$
is a right and left (continuous) module over $\Psi^{-1}_{G,c} (M)$. This means, in particular, that
Proposition \eqref{prop:inverse-0} does apply to $\mathcal{A}^\infty_G(M)$.
\end{remark}
\bigskip
Let us go back to  
$$
\exp (-t D^2)=\frac{1}{2\pi i}\int_\gamma e^{-t\lambda} (D^2 - \lambda)^{-1} d\lambda$$

\begin{proposition} Assume that $G$ satisfies the RD condition. Then the heat operator $
\exp (-t D^2)$ is an element in $\mathcal{A}^\infty_G (M)$.
\end{proposition}

\begin{proof}
We need to control the resolvent  $(D^2 - \lambda)^{-1}$. For finite values of $\lambda$
we know that  
$$(D^2 - \lambda)^{-1}= B_\lambda + B_\lambda \circ A_\lambda$$
with $B_\lambda \in \Psi^{-2}_{G,c} (M)$ and $A_\lambda\in \mathcal{A}^\infty_G (M)$. This follows
from an obvious modification of Proposition \ref{prop:inverse-0}.
We are now  interested in the behaviour of $(D^2 - \lambda)^{-1}$
 {\it for large} values of  $| \lambda|$.
Following Shubin's monograph \cite{shubin-book}
one can easily develop a pseudodifferential calculus with parameters $\Psi^*_{G,c}(M,\Lambda)$,
with  $\Lambda \subset\CC$
an allowable  conic set  as in \cite{shubin-book}. Notice that our $\gamma$ in \eqref{heat-operator} can be chosen to be
an  allowable conic set, as in \cite{shubin-book}.
%\footnote{\textcolor{cyan}{From Paolo: there is an issue of homogeneity; not any set in $\CC$ is  allowed...}} 
We consider $(D^2 - \lambda)$, with $\lambda\in \Lambda$
and regard it  as a differential operator with parameter. 
Then we know that there exists a parametrix (with parameter) $\{B_\lambda\}\in \Psi^{-2}_{G,c}(M,\Lambda)$ such that 
$$(D^2 - \lambda) \circ B_\lambda=\Id+R_\lambda\;,\quad B_\lambda  \circ (D^2 - \lambda)=\Id+S_\lambda$$
with $R_\lambda, S_\lambda\in \Psi^{-\infty}_{G,c}(N,\Lambda)$. Moreover, 
$B_\lambda$, $R_\lambda$ and $S_\lambda$ have uniform $G$-compact support in $\lambda$.
Notice that proceeding as in \cite{shubin-book, shubin92}
we can give estimates on the norm of these operators in $\Psi^{-2}_{G,c}(M,\Lambda)$ and 
$\Psi^{-\infty}_{G,c}(M,\Lambda)$.
%%the norm
%%of $R_\lambda, S_\lambda$ in the algebra of bounded operators in $L^2$ is smaller than 1 for $\lambda$
%%large enough. 
For $\lambda\in \Lambda$ large, $\lambda$ not in the  $L^2$-spectrum, we have $$(D^2 - \lambda)^{-1}= B_\lambda 
(\Id+R_\lambda)^{-1}$$
with $(\Id+R_\lambda)^{-1}=\Id+A_\lambda$ and with $A_\lambda\in \mathcal{A}_G (M)$, since each $R_\lambda\in 
A_G^c (M)\subset \mathcal{A}^\infty_G (M)$ and 
$\mathcal{A}^\infty_G (M)$ is holomorphically closed. Thus 
$$(D^2 - \lambda)^{-1}= B_\lambda + B_\lambda \circ A_\lambda$$%C_\lambda,\quad C_\lambda:=B_\lambda \circ A_\lambda$$
%for $\lambda$ large.
 with $B_\lambda\in \Psi^{-2}_{G,c}(M,\Lambda)$ and uniform compact $G$-support
 and $A_\lambda\in \mathcal{A}^\infty_G (M)$.
 Thus, with $\Lambda=\gamma$ we have:
 $$
\frac{1}{2\pi i}\int_\gamma e^{-t\lambda} (D^2 - \lambda)^{-1} d\lambda\,=\,\frac{1}{2\pi i}\int_\gamma e^{-t\lambda} B_\lambda  d\lambda + \frac{1}{2\pi i}\int_\gamma e^{-t\lambda}  B_\lambda \circ A_\lambda  d\lambda$$
$$
$$
Proceeding again as in \cite{shubin-book} we know that the 
 first summand on the right hand side   is an element in $\Psi^{-\infty}_{G,c} (M)$. Since   we have a continuous
inclusion of the Fr\'echet algebra $\Psi^{-\infty}_{G,c}(M)\equiv A_G^c (M)$ into the Fr\'echet
algebra $\mathcal{A}^\infty_{G}( M)$ we conclude that 
$$
\frac{1}{2\pi i}\int_\gamma e^{-t\lambda} B_\lambda  d\lambda\;\;\in\;\;
\mathcal{A}^\infty_{G}( M).$$ It remains to analyze the second summand. To this end we 
shall now give further properties for the operator $A_\lambda\in \mathcal{A}^\infty_{G}( M)$ 
such that  $(\Id+R_\lambda)^{-1}=\Id+A_\lambda$.\\
We know that the inclusion of the Fr\'echet algebra $\mathcal{A}^\infty_G ( M)$ into $C^* ( M)^G$, the Roe algebra of the $G$-space $ M$,
is continuous. Moreover, it is clear that $\mathcal{A}^\infty_G (M)$ is symmetric, i.e. stable 
under taking adjoints. Thus  $\mathcal{A}^\infty_G (M)$ is a $\Psi^*$-subalgebra, see \cite[Definition 2.3]{LMN-olo}.
This implies, in particular, that the passage to the inverse is continuous, see  \cite[Proposition 2.4 (b)]{LMN-olo}.
Now, the $\lambda$-family $\Id+R_\lambda$ is a family in the unitalization of $\mathcal{A}^\infty_G (M)$
converging to the identity in the Fr\'echet-topology of $\mathcal{A}^\infty_G (M)$ as $\lambda\to \infty$
(in fact it even converges in the 
Fr\'echet topology of $\Psi^{-\infty}_{G,c}(M)$). By the stated continuity it follows that 
$(\Id+R_\lambda)^{-1}$ converges to the inverse of the identity, i.e. the identity,  in the Fr\'echet topology of 
$\mathcal{A}_G ( M)$. Since 
$(\Id+R_\lambda)^{-1}=\Id+A_\lambda$ with $A_\lambda\in \mathcal{A}^\infty_G ( M)$, we conclude that
$A_\lambda$ converges to $0$ in the Fr\'echet topology of $\mathcal{A}^\infty_G ( M)$ and it is, in particular,
bounded in $\lambda$. Recall that we need to analyse the term
$$ B_\lambda \circ A_\lambda$$
with $\{B_\lambda\}\in \Psi^{-2}_{G,c} (\partial M,\Lambda)$, $\Lambda=\gamma$,
and $A_\lambda$ as above. The family of operators 
$B_\lambda$ is  certainly converging to $0$ as a function of $\lambda$, as $|\lambda|\to \infty$,
with values in the Fr\'echet algebra $\Psi^0_{G,c}( M)$. (Notice that for a {\it uniform} behaviour 
at infinity in the pseudodifferential topology we need to trade some regularity
 with some $\lambda$-decay; see [Gilkey, 2nd edition, Lemma 1.7.1] and [Melrose, Chapter 7, from p. 284 to p. 286].)
   Moreover we know that 
$A_\lambda$ is converging to zero in $A^\infty_G ( M)$ as $|\lambda|\to \infty$.
We are effectively considering the composition of two maps:
the first one is the map 
$\gamma\ni\lambda\to \Psi^0_{G,c}( M)\times \mathcal{A}^\infty_G (M)$ assigning to $\lambda$ the
pair $(B_\lambda,A_\lambda)$; the second map is the composition
$ \Psi^0_{G,c}(M)\times \mathcal{A}^\infty_G (M)\xrightarrow{\circ} \mathcal{A}^\infty_G ( M)$.
The first map is certainly continuous;
the second map is continuous 
by Lemma \ref{lemma:composition}. We conclude from all of the above  that the following Proposition, that
we single out for later use, holds:
\begin{proposition}\label{bounded}
For the inverse $(D^2-\lambda)^{-1}$ we have  
$$(D^2-\lambda)^{-1} =B_\lambda + C_\lambda$$
 with $B_\lambda\in \Psi^{-2}_{G,c}(M)$
of uniform $G$-compact support 
and $C_\lambda%:=B_\lambda \circ A_\lambda
\in  \mathcal{A}^\infty_G (M)$ such that 
$$\gamma\ni\lambda\to C_\lambda\in  \mathcal{A}^\infty_G (M) $$
is continuous and bounded in the Fr\'echet topology of $\mathcal{A}^\infty_G (M) $. 
\end{proposition}

\noindent
Recall now that $C_\lambda=B_\lambda \circ A_\lambda$. The Proposition then immediately  implies  that
$$\frac{1}{2\pi i}\int_\gamma e^{-t\lambda}  B_\lambda \circ A_\lambda d\lambda\;\;\in\;\;\mathcal{A}^\infty_G (M)
$$
as required.
\end{proof}
The same proof applies to all the terms in the Connes-Moscovici projector $V(D)$. Consequently we have
given a new proof of  the following

\begin{proposition}\label{prop:CM-infty}
Assume that $G$ satisfies the RD condition.
Then the  idempotent $V(D)$  is an element in  $M_{2\times 2} (\mathcal{A}_G^{\infty} (M))$
(with the identity adjoined).
\end{proposition}

\section{Geometric preliminaries on manifolds with boundary}\label{section:geometric}

\subsection{Proper $G$-Manifolds with boundary}$\;$\\
Let $M_0$ be a  manifold with boundary, $G$ a Lie group acting properly
and cocompactly on $M_0$. 
There exists a collar neighbourhood $U$ of the boundary $\pa M_0$, $U\cong [0,2]\times \partial M_0$,
 which is $G$-invariant and such that the action of $G$ on $U$ is of product type, i.e. trivial in the normal
 direction. Notice that $\pa M_0$ inherits a proper $G$-manifold structure.
We endow $M_0$ with a $G$-invariant metric $g_0$ which is of product type near the boundary.
%\footnote{\blue{make sure this can be done}}.
We let
$(M_0,g_0)$ be the resulting riemannian  manifold with boundary;  in the collar
neighborhood $U\cong [0,2]\times \partial M_0$ the metric  $g_0$  can be written, through the above isomorphism,
 as $dt^2 + g_{\partial}$,
with $g_{\partial}$ a $G$-invariant riemannian metric on  $\partial M_0$. 
We consider the associated manifold with cylindrical ends
$M:= M_0\cup_{\partial M_0} \left(   (-\infty,0] \times \partial M_0 \right)$,
endowed with the extended metric $g$ and the extended $G$-action.
The coordinate along the cylinder will be denoted by $t$. We will also consider the $b$-version of $(M,g)$, obtained by
performing the change of variable $\log x=t$.
This is a $b$-riemannian
manifold with product $b$-metric $$\frac{dx^2}{x^2}+ g_{\partial}$$
 near the boundary. We shall freely pass from the $b$-picture to the
cylindrical-end picture, without employing two different notations.
(Our arguments will actually apply to the more general case of {\it exact}
$b$-metrics, or, equivalently, manifolds with asymptotic cylindrical ends; we shall not insist on this point.).

\medskip
\noindent
We refer the reader to Melrose' book for more on $b$-geometry, see \cite{Melrose}

\begin{example}
\label{basic-example-1}
With the slice theorem at hand, its is easy to construct examples: start with an inclusion $K\subset G$ of Lie groups with $K$ compact, and 
let $S$ be a compact $K$-manifold with boundary $\partial S$. Then $M:= G\times_KS$ is a manifold with boundary $\partial M=G\times_K\partial S$, equipped with a proper, cocompact action of $G$. Next, choose $K$-
invariant inner product on the Lie algebra $\mathfrak{g}$ of $G$, so that we have an orthogonal decomposition $\mathfrak{g}=\mathfrak{k}\oplus\mathfrak{p}$ where $\mathfrak{k}$ is the Lie algebra of $K$ and 
$\mathfrak{p}$ its orthogonal complement. This metric defines a connection on the principal $K$-bundle $G\to G\slash K$ and defines an isomorphism
\begin{equation}
\label{tb-induced}
TM\cong G\times_K(\underline{\mathfrak{p}}\oplus TS).
\end{equation}
(Here, $\underline{\mathfrak{p}}\to S$ denotes the trivial $K$-equivariant vector bundle with fiber $\mathfrak{p}$.) With this isomorphism we see that a choice of $K$-invariant metric $g_S$ on $S$ of the desired type 
(i.e., $b$-riemannian of product-type near the boundary, or more generally exact) defines a $G$-invariant of the same type on $M$. 

As an explicit example, consider $G=SL(2,\mathbb{R})$ and $K=SO(2)$ acting on the unit disk $\mathbb{D}^2$ in the complex plane by rotations around the origin. The resulting manifold 
$M:=SL(2,\mathbb{R})\times_{SO(2)}\mathbb{D}^2$ is a $4$-dimensional fiber bundle over hyperbolic $2$-space $SL(2,\mathbb{R})\slash SO(2)\cong\mathbb{H}^2$ with fiber $\mathbb{D}^2$. 
The boundary $\partial M$ of this manifold is isomorphic to $SL(2,\mathbb{R})$.

\end{example}

%\blue{$\bullet$ $\longrightarrow$} Say more about $M_0\times T$. Talk about the groupoid $\mathcal{G}=T\rtimes G$
%and the fact that we have a groupoid action on $M_0\times T$ with moment map $pr_2$, the projection onto
%the second factor.
\subsection{Dirac operators}\label{subsect:b-dirac}$\;$\\
%\textcolor{m}{$\bullet$ $\longrightarrow$}   Talk about $G$-equivariant families of Dirac operators. 
%Make fundamental assumption about the boundary family.\\
We assume the existence of  a $G$-equivariant bundle of Clifford modules $E_0$ on $M_0$, endowed with a hermitian
metric $h$, product-type near the boundary, for which the Clifford action is
unitary, and  equipped with a Clifford connection also of product type near the boundary. 
Associated to these structures there is a generalized $G$-invariant 
Dirac operator $D$ on $M_0$ with product structure near the
boundary acting on the sections of  $E_0$. We denote by $D_{\pa}$ the operator induced on the boundary.
We employ the same symbol, $D$, for the associated $b$-Dirac operator on $M$, acting on the extended
Clifford module $E$.
We also have $D_{\cyl}$ on $\RR\times \pa M_0\equiv \cyl(\pa M_0)$.
We shall make the following fundamental assumption
\begin{assumption}\label{assumption:invertibility}
There exists  $\alpha>0$ such that
\begin{equation}\label{invertibility}
{\rm spec}_{L^2} (D_{\partial})\cap [-\alpha,\alpha]=\emptyset
\end{equation}
\end{assumption}
%As explained in \cite{LPMEMOIRS} and in \cite[Appendix]{LPBSMF},
%in turn based on \cite{Lott-torsion}, this assumption implies that $\D_{\pa}$ is
%invertible in the $\B^\infty$-Mishchenko-Fomenko calculus; more on this later.
It should be noticed that because of the self-adjointness of $D_{\partial}$,
assumption \eqref{invertibility} implies the $L^2$-invertibility of $D_{\cyl}$.
This is based on the following elementary argument:
we conjugate the operator  $D_{\cyl}$,
$$D_{\cyl}=\left(\begin{array}{cc} 0&-\frac{\pa}{\pa t} + D_{\partial}\\
\frac{\pa}{\pa t}  + D_{\partial}&0 \end{array} \right),
$$
 by Fourier transform
in $t$, $\mathcal{F}_{t\to\lambda}$, obtaining \begin{equation}\label{cyl-ind}
\left(\begin{array}{cc} 0&-i\lambda + D_{\partial}\\
i\lambda + D_{\partial} &0 \end{array} \right).
\end{equation}
In the $b$-calculus-picture \eqref{cyl-ind}  is the indicial family $I(D_{\cyl},\lambda)$
of $D_{\cyl}$ and it is obtained through Mellin transform of the corresponding cylindrical
$b$ operator.
%$$\left(\begin{array}{cc} 0&-x\frac{\pa}{\pa x} + \widetilde{D}_{\partial}\\
%x\frac{\pa}{\pa x}  + \widetilde{D}_{\partial}&0 \end{array} \right),.
%$$
The self-adjointness of $ D_{\partial}$ implies that
\eqref{cyl-ind}, i.e. $I(D_{\cyl},\lambda)$, is $L^2$-invertible
for each $\lambda\in\RR\setminus\{0\}$; the invertibility of  $D_{\partial}$ then implies that
\eqref{cyl-ind} is $L^2$-invertible for each $\lambda\in\RR$. Conjugating back the inverse of \eqref{cyl-ind}
one obtains an operator which provides an $L^2$-inverse of $D_{\cyl}$.
\begin{example}
As an example where this condition is satisfied we can consider a $G$-proper manifold with boundary
with a $G$-invariant riemannian metric and a $G$-invariant spin structure with the property that the metric
on the boundary is of positive scalar curvature. We would then consider the spin-Dirac operator $D$;
because of the psc assumption on the boundary we do have that $D_\partial$ is $L^2$-invertible.
These manifolds arise as in \cite{guo-mathai-wang-psc} from the slice theorem and  a $K$-manifold with boundary $S$
endowed with a $K$-invariant metric which is of psc on $\partial S$. See example 
\ref{example-spin} below for more on this.\\
$K$-manifolds with boundary 
with a $K$-invariant metric of psc on the boundary arise, for example,  as follows.
Consider a compact $K$-manifold without boundary $N$ endowed with a $K$-invariant metric of positive scalar curvature.
For the existence of such manifolds see for example \cite{lawson-yau-psc},  \cite{hanke-symmetry}, \cite{wiemeler-TAMS}. We can 
now perform on this manifold $K$-equivariant surgeries and produce  along the process a $K$-manifold with boundary $W$.
Under suitable conditions (for example, equivariant surgeries only of codimension
at least equal to 3) this manifold with boundary $W$ will have a $K$-invariant metric of positive scalar curvature. 
We can now take the connected sum of $W$  with a closed $K$-manifold not admitting a K-invariant metric of psc. See \cite{hanke-symmetry}, \cite{wiemeler-TAMS}. The result will be a manifold $S$ with a K-invariant metric which is of psc
(only) on  $\partial S$.
\end{example}

\begin{example}\label{example-spin}
Continuing Example \ref{basic-example-1}, assume that $K\subset G$ is such that $G\slash K$ carries a $G$-invariant spin structure. This
is equivalent to giving a Spin-lifting
\[
\xymatrix{&&{\rm Spin}(\mathfrak{p})\ar[d]\\K\ar[rr]^{\rm Ad}\ar[rru] &&SO(\mathfrak{p})}
\]
of the adjoint action of $K$ on $\mathfrak{p}$. Together with a $K$-invariant spin structure on $S$, this defines a unique $G$-invariant spin structure on M as in \cite{HM}
whose space of spinors is given by
\[
\mathscr{S}(M)=C^\infty(G,\mathbb{S}_\mathfrak{p}\otimes\mathscr{S}(S))^K,
\]
where $\mathbb{S}_\mathfrak{p}$ is the spinor module for ${\rm Cliff}(\mathfrak{p})$ and $\mathscr{S}(S)$ denotes the space of spinors on $S$.
The Dirac operator is given by $D=D_{G\slash K}\otimes 1+1\otimes D_S$, where $D_S$ is the Dirac operator on $S$ and 
\[
D_{G\slash K}:=\sum_{i=1}^{\dim(\mathfrak{p})} X_i\otimes c(X_i),
\]
with $X_i,~i=1,\ldots,\dim(\mathfrak{p})$ is an orthonormal basis of $\mathfrak{p}$. The Dirac operator induced on the boundary has a similar shape, namely $D_\partial=D_{G\slash K}\otimes 1+1\otimes D_{\partial S}$.
We have 
$$D^2_\partial= D^2_{G\slash K}\otimes 1 + 1\otimes D^2_{\partial S}$$
so that this operator, and therefore $D_\partial$,
 is $L^2$invertible as long as one of the two summands in the above formula is strictly positive.

\noindent
When there is no full gap in the $L^2$-spectrum of $D_{\partial}$ near $0$, we can twist this construction with a (flat) $K$-equivariant vector bundle on $S$ to ensure $L^2$-invertibility of this boundary Dirac operator.

\medskip
\noindent
In the explicit example $M:=SL(2,\RR)\times_{SO(2)}\mathbb{D}^2$, we remark that the Dirac operator on the boundary $\partial\mathbb{D}^2=S^1$ is simply given by $D_{\partial\mathbb{D}^2}=i\partial\slash\partial \varphi$, 
where $\varphi$ is the angular coordinate on $S^1$. Identifying this circle with the maximal compact subgroup $SO(2)\subset SL(2,\RR)$ with generator $X_0\in\mathfrak{sl}(2,\RR)$, 
this results in the following boundary operator on the boundary $\partial M\cong SL(2,\RR)$:
\begin{equation}
\label{Dirac-SL(2)}
D_{SL(2,\RR)}:=\sum_{i=0}^2X_i\otimes c(X_i)=i\left(\begin{matrix} X_0&X_1+iX_2\\X_1-iX_2&-X_0\end{matrix}\right),
\end{equation}
acting as an unbounded operator on $L^2(G,\CC^2)$. The formula for this Dirac operator is exactly the same as the operator on the universal cover of $SL(2,\mathbb{R})$ studied in \cite{BNPW}. We can determine
its spectrum in the same way as done in {\em loc. cit.}, taking into account only the irreducible representations that enter the Plancherel decomposition
\[
L^2(G)\cong\int^{\oplus}_{\hat{G}} \mathcal{H}_\pi\otimes\mathcal{H}_\pi^*d\mu(\pi).
\]
It is well-known that the Plancherel measure for $SL(2,\mathbb{R})$ is supported at two families of irreducibles: the discrete series representations $\mathcal{D}^\pm_n,n=2,3,\ldots$ and the principal series
$\mathcal{P}^{\pm,i\nu},\nu\in\mathbb{R}$. For each element of these families, we can use the description in \cite{BNPW} of the spectrum of the Dirac operator \eqref{Dirac-SL(2)} at those representations: for 
the discrete series $\mathcal{D}^\pm_n$ it is given by
\[
\{-\frac{n}{2}\}\cup\{-\frac{1}{2}\pm\frac{1}{2}\sqrt{1+n(n-2)+2(n+2k)(n+2k-2)},k=1,2,\ldots\}.
\]
For the principal series $\mathcal{P}^{\pm}$ it is given by
\[
\{-\frac{1}{2}\pm\sqrt{\nu^2+\frac{1}{2}+2k(k-1)},k\in\ZZ\},\quad\mbox{and}\quad \{-\frac{1}{2}\pm\sqrt{\nu^2+\frac{1}{2}+2(k^2-\frac{1}{4})},k\in\ZZ\}.
\]
This shows that the boundary Dirac operator $D_\partial$ is indeed $L^2$-invertible. (In \cite{BNPW} it is in fact essential that the Dirac operator on the universal cover of $SL(2,\RR)$ {\em does not} have a spectral gap at
zero, but the irreducible representations that are responsible for this, do not descend to $SL(2,\RR)$.)  
\end{example}

%\noindent
%\textcolor{m}{ $\longrightarrow\;$ 
%{\bf Old strategy for the resolvant} :} define an algebra, in the closed case, by taking the algebra generated by whatever we need 
%for getting the $b$-parametrix to belong to (the b-version) of  it. See below.\\
%For example: all negative order psudodifferential operators with weighted exponential decay at infty (as in Taylor). We know
%they do not compose well, so we force the composition property by taking the algebra generated by them.
%Investigate extension of cyclic cocycles to such an algebra. \\(From Paolo: this seems like an extreme
%solution, to be taken if everything else fails). \textcolor{m}{ $\longleftarrow$$\longleftarrow$ }

%\begin{remark}\label{wu-versus-taylor}
%\textcolor{m}{We have proved that the Connes-Moscovici projector is, in particular, an element in
%$\mathcal{C}_G (N)$, the Harish-Chandra algebra (because $\mathcal{A}_G^{{\exp}} (N)
%\subset \mathcal{C}_G (N)$)\\
%On the other hand, we know by other means that the heat kernel is an element in $\mathcal{C}_G (M)$.
%Does this give a shorter proof that the Connes-Moscovici projector is in $\mathcal{C}_G (N)$,
%by using the argument of the proof of Propoistion 4.1 in Wu ? Actually this is explicitly stated in 
%Wu, page 181.}
%\end{remark}

\section{$C^*$ Atiyah-Patodi-Singer index classes for proper actions}\label{sect:b-c*-index}

\subsection{$C^*$-algebras and $C^*$-Hilbert modules}$\;$\\
Let $M_0$ be a $G$-proper manifold with boundary, with compact quotient, and let
$M$ be the associated manifold with cylindrical ends, or, equivalently, the associated $b$-manifold.
In the $b$-picture we consider $\mathcal{E}_b$, the $C^*_r G$-Hilbert module obtained
 by (double) completion of  $\dot{C}^\infty_c (M,E)$ endowed with the $C^\infty_c (G)$-valued
inner product 
$$(e,e^\prime)_{C^*_c (G)} (x):= (e,x\cdot e^\prime)_{L^2_b (M,E)}, \quad e, e^\prime
\in \dot{C}^\infty_c (M,E), \;x\in G\,.$$
Here the dot means vanishing of infinite order at the boundary; notice that on the right hand side we employ
$L^2_b (M,E)$.
One can prove that   $\mathbb{K}(\mathcal{E}_b)$ is isomorphic to the 
relative Roe algebra $C^* (M_0 \subset M,E)^G$, see \cite{PS-Stolz} for definitions and proofs,
 so that we have the following canonical
isomorphisms
\begin{equation}\label{iso-in-K}
K_* (\mathbb{K}(\mathcal{E}_b))\cong  K_* (C^* (M_0\subset M,E)^G)\cong K_* (C^* (M_0,E_0)^G)\cong K_* (C_r^*(G))
\end{equation}
with the second isomorphism also explained in \cite{PS-Stolz} and the third one well known due to the
co-compactness of the action (see for example \cite{Hochs-Wang-HC} in this context).\\

\subsection{ $b$-Pseudodifferential Operators}$\;$
%\blue{Here we talk about the short exact sequence of $b$-smoothing operators in the polynomial growing case.}

\noindent
We shall be interested in algebras of $b$-pseudodifferential operators and we invest a few lines
to recall the context in which they are defined, 
beginning with the well known case of a smooth compact riemannian manifold with boundary $(S_0,g_0)$
with associated $b$-manifold $(S,g)$. We also consider a hermitian complex vector bundle $E$ over $S$.
Recall that the small calculus of $b$-pseudodifferential 
operators on $S$, $ {}^b \Psi^{*}(S,E)$, is defined in terms
%$$ {}^b \Psi^{*}(S,E) + {}^b \Psi^{-\infty}(S,E) + \Psi^{-\infty}(S,E)$$
%with all of the three spaces defined in terms 
of Schwartz  kernels on the $b$-stretched product
$S \times_b  S$. We refer the reader to the basic reference  \cite{Melrose}. Quick introductions
are given in the Appendix of \cite{MPI} and also in the surveys
 \cite{Mazzeo-Piazza-II}, \cite{grieser-survey}, \cite{loya-survey}. 
% We shall be particularly intersted here 
% in  ${}^b \Psi^{-\infty}(S,E)$ and $\Psi^{-\infty}(S,E)$. The latter is defined by the smooth kernels on the manifold
% with corners $S\times S$ that vanish to infinite order at the left and right boundary of $S\times S$; the former 
% space is instead defined by smooth kernels on $S \times_b  S$ that vanish to infinite order at the left and right boundary
%of $S\times_b S$, denoted respectively ${\rm lb}$ and ${\rm rb}$  and are smooth up to the front face
%${\rm bf}$, the new boundary hypersurface created by the process of blow-up
% that is at the basis of the definition of $S \times_b  S$. 
It turns out that the algebra $ {}^b \Psi^{*}(S,E)$ is too small for many purposes, as it does not contain,
 for example, the inverse of an invertible $b$-differential operator.
 For this reason one needs to extend
 the small calculus and define the so-called $b$-calculus {\it with}  $\epsilon-${\it bounds}.
%The latter space on the right hand side is  the space of residual operators;  they are defined in terms of their kernels
%as 
%$\rho_{{\rm lb}} \rho_{{\rm rb}} H^\infty_b (S\times S, E\boxtimes E^*)$, $H^\infty=\cap H^m$ \footnote{here we are using  the notion
%of $b$-Sobolev space on a manifold with corners $X$; it is defined as usual
%but in terms of
%${\rm Diff}^m_b (X)$, the algebra
% of differential operators generated by $\mathcal{V}_b (X)$, the smooth vector fields on $X$ that are tangent to 
% the boundary of $X$.}. Notice that $\Psi^{-\infty,\epsilon}(S,E)$ has a natural Fr\'echet topology.
% We define 
% ${}^b \Psi^{-\infty,\epsilon}(S,E)$ 
% by doubling 
% as the space of distributional sections $A$ on $S\times_b S$ 
% with the property that $\phi A\in H^\infty_b (S\times_b S, \beta^* (E\boxtimes E^*))$ for any smooth function 
% with supported disjoint from the front face and 
% 
 In order to define the calculus with bounds we first recall some  basic facts. Consider quite generally
 a manifold with corners $X$ and the algebra of $b$-differential operators on it, ${\rm Diff}^*_b (X)$; this is the algebra
 of differential operators generated by $\mathcal{V}_b (X)$, the smooth vector fields on $X$ that are tangent to 
 the boundary of $X$. Let $\rho$ be a total boundary defining function for $\partial X$. \\
 For $\delta>0$ define \begin{equation}\label{conormal}\mathcal{A}^\delta (X):= \{u\in \rho^\delta L^\infty (X)\,|\, {\rm Diff}^*_b (X) (u)\subset 
 \rho^\delta L^\infty (X)\}\end{equation}
% and 
% $$\mathcal{A}^{\delta-} (X):= \bigcup_{\epsilon>0} \mathcal{A}^{\delta-\epsilon} (X)\,.
%$$
These spaces of conormal functions can also be defined using $b$-Sobolev spaces \cite[Ch. 4]{Melrose} from which it follows that they are in fact smooth in the interior. If $H$ is a boundary hypersurface of $X$ we can define
$\mathcal{A}^{\delta}_H (X)$ by requiring smoothness in the direction normal to $H$; this is done by doubling 
$X$ across $H$, obtaining $X_H$, and then defining $\mathcal{A}^{\delta}_H (X):= (\mathcal{A}^{\delta} (X_H))|_{X}$.
Similar definitions can be given for sections of a bundle $E$ over $X$, giving
$\mathcal{A}^\delta (X,E)$ and $\mathcal{A}^\delta_H (X,E)$.
 Notice that we have a natural Fr\'echet-space structure  on $\mathcal{A}^\delta (X,E)$ and on $\mathcal{A}^\delta_H (X,E)$.
 The calculus with $\epsilon$-bounds is defined as
  $$ {}^b \Psi^{*,\epsilon}(S,E) := {}^b \Psi^{*}(S,E) + {}^b \Psi^{-\infty,\epsilon}(S,E) + \Psi^{-\infty,\epsilon}(S,E)$$
The first summand on the right hand side are the $b$-operators in the small calculus of order $*$; then we have
$\Psi^{-\infty,\epsilon}(S,E)$ defined as $\mathcal{A}^\epsilon (S\times S, E\boxtimes E^*)$ and
$ {}^b \Psi^{-\infty,\epsilon}(S,E) $ defined as  $\mathcal{A}^\epsilon_{{\rm bf}} (S\times_b S, \beta^*(E\boxtimes E^*))$ with $\beta: S\times_b S\to S\times S$ the blow-down map \footnote{In the sequel, because of Assumption \ref{invertibility}, we shall have a certain 
flexibility on the choice of $\epsilon$, which is why we do not use the spaces  $\mathcal{A}^{\epsilon-} (S\times S, E\boxtimes E^*)$ and
$\mathcal{A}^{\epsilon-}_{{\rm bf}} (S\times_b S, \beta^*(E\boxtimes E^*))$ considered in \cite{Melrose}.}.
Composition formulae 
for the elements in the calculus with bounds are as in \cite[Theorem 4]{MPI}. In particular, 
${}^b \Psi^{-\infty,\epsilon}(S,E) + \Psi^{-\infty,\epsilon}(S,E)$ is a subalgebra
of the algebra of bounded operators on $L^2$ and $\Psi^{-\infty,\epsilon}(S,E)$
is an ideal in this algebra.

\noindent
We now consider a manifold with boundary $M_0$ endowed with a proper 
cocompact action of $G$. For simplicity we assume that the boundary is connected;
the general case is only slightly more complicated.
We assume that on $M_0$ we have a  $G$-invariant 
metric $g_0$ which is of product type near the boundary. 
We denote by $M$ the associated $b$-manifold; this is again a cocompact $G$-manifold 
with boundary but  now  endowed with a product $b$-metric $g$.\\
We sometimes denote the boundary of $M_0$ by $N$; for notational convenience we sometimes set $\pa M=N$.
% \textcolor{cyan}{Do we actually use the notation $N$ ?}\\

\medskip
We can define the $b$-stretched product $M\times_b M$ which inherits in a natural way an action of
$G\times G$ and a diagonal action of $G$, see \cite{LPMEMOIRS} \cite{LPETALE}. Proceeding as in these references, we can define
the algebra of $G$-equivariant $b$-pseudodifferential operators on $M$ with $G$-compact support, denoted ${}^b \Psi^*_{G,c}(M)$.
This is a $\ZZ$-graded algebra.  %In particular, ${}^b \Psi^{-\infty}_{G,c}(M)$ is an algebra.\\
% Of course we can also consider ${}^b \Psi^*_{G}(M)$, i.e. $b$-pseudodifferential operators 
% without any assumption on the support
% of the Schwartz kernel in $M\times_b M/G$; operators in these vector spaces will 
%{\bf not} compose. In particular, we can consider %the vector space  
%${}^b \Psi^{-\infty}_{G}(M)$.

Let us fix $\epsilon >0$. Then, as in the compact case, we can extend  the algebra ${}^b \Psi^*_{G,c}(M)$ and consider
 the $G$-equivariant $b$-calculus of $G$-compact support 
with $\epsilon$-bounds, denoted  ${}^b \Psi^{*,\epsilon}_{G,c}(M)$. Thus, by definition,
$${}^b \Psi^{*,\epsilon}_{G,c}(M)= {}^b \Psi^{*}_{G,c}(M) + {}^b \Psi^{-\infty,\epsilon}_{G,c}(M) + \Psi^{-\infty,\epsilon}_{G,c}(M)$$
%The second term on the right hand side corresponds to the 
%Schwartz kernels on $M\times_b M$, smooth in the interior,
%conormal of order $\epsilon$ on the left and right boundaries  of $M\times_b M$
%and smooth up to the front face. The third term on the right hand side 
%corresponds instead to  Schwartz kernels on $M\times M$, smooth in the interior and
%conormal of order $\epsilon$ on the boundary hypersurfaces of $M\times M$. See \cite{Melrose} and for a quick treatment
%the Appendix in \cite{MPI}. Elements in 
%$\Psi^{-\infty,\epsilon}_{G,c}(M)$ are called {\it residual}.\\
Composition formulae 
for the elements in the calculus with bounds are again as in \cite[Theorem 4]{MPI}. In particular, as before,
${}^b \Psi^{-\infty,\epsilon}_{G,c}(M)+
\Psi^{-\infty,\epsilon}_{G,c}(M)$ is an algebra and $
\Psi^{-\infty,\epsilon}_{G,c}(M)$ is an ideal in this algebra.
At this point, inspired also by the notation adopted in the closed case, we  introduce the notation \footnote{the reader will not confuse this notation with the one for 
conormal function adopted in \eqref{conormal}}
%$${}^b \mathcal{A}^c_{G}(M):={}^b \Psi^{-\infty}_{G,c}(M)+
%\Psi^{-\infty}_{G,c}(M)\,,\quad \mathcal{I}^c_{G}(M):=
%\Psi^{-\infty}_{G,c}(M)$$
%and also 
$${}^b \mathcal{A}^{c,\epsilon}_{G}(M):={}^b \Psi^{-\infty,\epsilon}_{G,c}(M)+
\Psi^{-\infty,\epsilon}_{G,c}(M)\,,\quad \mathcal{R}^{c,\epsilon}_{G}(M):=
\Psi^{-\infty,\epsilon}_{G,c}(M)$$
where $\mathcal{R}$ stands for {\it residual}.
We shall eventually simplify this notation and forget about the $\epsilon$.

\medskip
Assume now that $| \pi_0 (G)|<\infty$; we can apply to $M$ the slice theorem and obtain a diffeomorphism
$$ G\times_K S \xrightarrow{\alpha} M\,.$$
with $S$ a  compact manifold with boundary.
We observe that, as in the closed case, the correspondence $\widetilde{k}\rightarrow T_{\widetilde{k}}$
explained  in Part 1 and in Section \ref{subsect:3algebras} gives an identification 
\begin{equation}\label{identify-b} {}^b \Psi^{-\infty}_{G,c}(M)=\{T_{\widetilde{k}},\; \widetilde{k}\in(C^\infty_c (G)\hat{\otimes}\, {}^b\Psi^{-\infty}(S))^{K\times K} \}\,.
\end{equation}
%\red{$\bullet\longrightarrow$ As in the closed case, discuss this identification.}\\
%\red{Question: is ${}^b\Psi^{-\infty}(S))$ nuclear ? It would seem so given that it is a subspace
%of the smooth functions on the $b$-stretched product....}\\
%\red{$\bullet\longrightarrow$  Also in this case it would be useful to see $(C^\infty_c (G)\hat{\otimes}\, {}^b\Psi^{-\infty}(S))^{K\times K} $ as
%$K\times K$-equivariant maps $G\to {}^b\Psi^{-\infty}(S)$ of compact support. Here and above the Fr\'echet structure
%of  $ {}^b\Psi^{-\infty}(S)$, and also of ${}^b \Psi^{-\infty}_{G,c}(M)$ should be discussed. }\\
Similarly
\begin{equation}\label{identify-b-bis}
{}^b \Psi^{-\infty,\epsilon}_{G,c}(M)=
\{T_{\widetilde{k}},\; \widetilde{k}\in(C^\infty_c (G)\hat{\otimes}\, {}^b\Psi^{-\infty,\epsilon}(S))^{K\times K}\}
\end{equation}
\begin{equation}\label{identify-b-ter}
\Psi^{-\infty,\epsilon}_{G,c}(M)=
\{T_{\widetilde{k}},\; \widetilde{k}\in(C^\infty_c (G)\hat{\otimes}\, \Psi^{-\infty,\epsilon}(S))^{K\times K}\}
\end{equation}
One can prove, as in the closed case treated in Subsection \ref{subsect:3algebras}, that
${}^b \Psi^{-\infty,\epsilon}_{G,c}(M)\equiv (C^\infty_c (G)\hat{\otimes}\, {}^b\Psi^{-\infty,\epsilon}(S))^{K\times K} $ can be identified
 with
 \begin{equation}\label{b-c-with-functions}
\left\{\Phi:G\to {}^b \Psi^{-\infty,\epsilon}(S),~\mbox{smooth, compactly supported and }
~K\times K~\mbox{invariant}\right\}
\end{equation}
 A similar description can be given for
$\Psi^{-\infty,\epsilon}_{G,c}(M)$.
Notice that we are using throughout  the natural Fr\'echet structures of ${}^b \Psi^{-\infty,\epsilon}(S)$
and  $\Psi^{-\infty,\epsilon}(S)$.

%$K\times K$-equivariant maps $G\to {}^b\Psi^{-\infty}(S)$ of compact support. 

\smallskip
\noindent
Restriction to the front face of 
$M\times_b M$ defines  the indicial operator $I\,:\, {}^b \Psi^{-\infty}_{G,c} (M)\to {}^b \Psi^{-\infty}_{G,c,\RR^+} (\overline{N^+ \pa M})$ and, more generally, 
 \begin{equation}\label{indicial}I\,:\, {}^b \Psi^{-\infty,\epsilon}_{G,c} (M)\to {}^b \Psi^{-\infty,\epsilon}_{G,c,\RR^+} (\overline{N^+ \pa M})
 \end{equation}
 with $\overline{N^+ \pa M}$ denoting the compactified inward-pointing normal bundle to the boundary and the subscript
$\RR^+$ denoting equivariance with respect to the natural $\RR^+$-action.  See \cite[Section 4.15]{Melrose},
where $\overline{N^+ \pa M}$ is denoted $\widetilde{M}$.
Recall that there is a diffeomorphism:  $\overline{N^+ \pa M}\simeq \pa M\times [-1,1]$ and that
$\overline{N^+ \pa M}$  is endowed with a $b$-metric, so that metrically we can think of it as $\RR\times \pa M$.
 %The kernels in 
 %${}^b \Psi^{-\infty,\epsilon}_{G,c,\RR^+} (\overline{N^+ \pa M})$ 
%are localized in a neighbourhood of the intersection of the lifted diagonal with the front face.
%; if we see them as 
%kernels on $\RR\times \pa M$ they are $\RR\times G$-invariant and of $\RR\times G$-compact support.
%\textcolor{cyan}{Last sentence to be doubled-checked}
%%Put it differently,
%%they are of $G\times\RR^+$-compact support with respect to the natural $G\times \RR^+$-action 
%%on $\overline{N^+ \pa M}\times $\overline{N^+ \pa M}$.
Under the identification \eqref{identify-b} the indicial operator is simply induced by the indicial operator
 on the slice $S$, i.e. on the second factor of $(C^\infty_c (G)\hat{\otimes}\, {}^b\Psi^{-\infty}(S))^{K\times K}$: thus 
 if $\alpha\otimes A$ is a decomposable element in $ (C^\infty_c (G)\hat{\otimes}\, {}^b\Psi^{-\infty}(S))^{K\times K}$, then 
 $$I(\alpha\otimes A)= \alpha\otimes I(A)\in (C^\infty_c (G)\hat{\otimes}\, {}^b\Psi^{-\infty}_{\RR^+}(\overline{N^+ \partial S}))^{K\times K}$$
 and similarly for \eqref{indicial}.
% \footnote{\blue{Here we maybe need to understand the continuity
% properties of $\Id\otimes I$; we are defining this operator on the decomposable elements and we are saying that
% it extends to all of $(C^\infty_c (G)\hat{\otimes}\, {}^b\Psi^{-\infty}(S))^{K\times K}$ and that this extension 
% is compatible with \eqref{identify-b}, \eqref{identify-b-bis} and the indicial operator on the left hand side of 
% these identifications...This seems all pretty clear; the continuity is almost automatic because the indicial
% operator is continuos from  $\Psi^{-\infty}_b (S)$ to
%$ \Psi^{-\infty}_{b,\RR^+} (\overline{N^+ S})$
% with respect to their natural  Fr\'echet topologies. The rest should follow from a careful study of the identifications
% involved. }} 
By fixing a cut-off function $\chi$ on the collar neighborhood of the
boundary, equal to 1 on the boundary, we can define a section $s$ to the
indicial homomorphism $I$: $s:   {}^b \Psi^{-\infty,\epsilon}_{c,G,\RR^+} (\overline{N^+ \pa M})
\to  {}^b \Psi^{-\infty,\epsilon}_{G,c} (M)$. $s$ associates to a $\RR^+$-invariant operator
$G$ on  $\overline{N^+ \pa M}$ an operator on the $b$-manifold  $M$; the latter is obtained by pre-multiplying
and post-multiplying by the cut-off function $\chi$.\\
The Mellin transform in the $s$-variable for  an $\RR^+$-invariant kernel $\kappa (s,y,y')\in 
 {}^b \Psi^{-\infty,\epsilon}_{G,c,\RR^+} (\overline{N^+ \pa M})$ defines an isomorphism $\mathcal{M}$
between ${}^b \Psi^{-\infty,\epsilon}_{G,c,\RR^+} (\overline{N^+ \pa M})$ and holomorphic families 
of operators 
$$\{\RR\times i (-\epsilon,\epsilon)\ni\lambda \to \Psi^{-\infty}_{G,c} (\partial M)\}$$
rapidly decreasing as $|{\rm Re}\lambda |\to \infty$  as functions with values in the Fr\'echet algebras $\Psi^{-\infty}_{c,G} (\partial M)$. We denote by $I(R,\lambda)$ the Mellin transform of $R\in  
{}^b \Psi^{-\infty,\epsilon}_{c,G,\RR^+} (\overline{N^+ \pa M})$; if $P\in  {}^b \Psi^{-\infty,\epsilon}_{G,c} (M)$
then its indicial family is by definition the indicial family associated to its indicial operator $I(P)$.
The inverse $\mathcal{M}^{-1}$ is obtained by associating to a holomorphic 
family $\{S(\lambda), |{\rm Im} \lambda|<\epsilon \}$, rapidly decreasing in ${\rm Re}\lambda$,
the $\RR^+$-invariant Schwartz kernel $\kappa_S$ that in projective coordinates is given by 
\begin{equation}\label{inverse-mellin}
\kappa_G (s,y,y^\prime)=\int_{{\rm Im}\lambda=r} s^{i\lambda} G(\lambda)(y,y^\prime)d\lambda
\end{equation}
with $r\in (-\epsilon,\epsilon)$.
Notice once again that under the identifications \eqref{identify-b}, 
\eqref{identify-b-bis} the indicial family
and its inverse can be defined by reasoning solely on the slice $S$. 

\medskip
\noindent
For more on all of this we refer to \cite[Chapter 5]{Melrose}. 

%\footnote{\blue{we shall
%need to double check this statement....}}

\subsection{The $C^*$-index class}\label{subsect:C*-index} $\;$\\
Let $M_0$ and $M$ as above. We assume $M_0$ to be even dimensional. Let $D$ be a $\ZZ_2$-graded odd Dirac operator with boundary operator
satisfying Assumption \ref{assumption:invertibility}.
We shall now prove the existence of a $C^*$-index class associated to $D^+$.
One begins
by finding a symbolic parametrix $Q_\sigma$ to $D^+$, with remainders $R^\pm_\sigma$:
\begin{equation}\label{para1}
Q_\sigma D^+=\Id - R^+_\sigma\,,\quad D^+ Q_\sigma=\Id - R_\sigma^-\,.
\end{equation}
We can choose
$Q_\sigma\in {}^b \Psi^{-1}_{G,c}(M)$, i.e. of $G$-compact support
and then get
 $R^\pm_\sigma\in {}^b \Psi^{-\infty}_{G,c}(M)$.

Consider now $I(D^+)\equiv D^+_{{\rm cyl}}\equiv \partial/\partial t + D_{\partial}$; by assumption we know that
this operator is $L^2$-invertible on the $b$-cylinder ${\rm cyl}(\partial M)\equiv \overline{N_+ \pa M}$.
Consider the 
smooth 
 kernel $K(Q^\prime)$ on $M\times_G M$ which is zero outside a neighbourhood $U_{{\rm bf}}$ of the front face 
 and such that in $U_{{\rm bf}}$ 
 with (projective) coordinates $(x,s,y,y')$ is equal to 

$$K(Q^\prime)(s,x,y,y^\prime):= \chi (x) \int_{-\infty}^\infty s^{i\lambda} K((D_\pa +i\lambda)^{-1}  \circ I( R_\sigma^-,\lambda))(y,y^\prime) d\lambda\,.$$
Because of the presence of $(D_\pa +i\lambda)^{-1} $ this is not compactly supported. Still,
by proceeding {\it exactly} as in \cite[Lemma 4 and 5]{LPETALE} we can establish the following fundamental result:

\begin{theorem}\label{b-index-class-C*}
The kernel $K(Q^\prime)$ defines  a bounded operator on the $C^*_r G$-module $\mathcal{E}_b$:
$Q^\prime\in \mathbb{B}(\mathcal{E}_b)$.  If $Q:= Q_\sigma-Q^\prime$, then $Q$ provides an inverse of $D^+$ modulo elements in $\mathbb{K}(\mathcal{E}_b)$. Thus, for a Dirac operator $D$ satisfying the assumption \ref{assumption:invertibility} there is a well defined index class $\Ind_{C^*} (D)\in K_0 ( \mathbb{K}(\mathcal{E}_b))\equiv K_0 (C^* (M_0\subset M,E)^G)$. 
\end{theorem}

We have given a $b$-calculus approach to this index class because it is the one that will prove to be useful 
in establishing  the existence of a {\it smooth} index class, see below, and for proving a {\it higher} APS index formula.
However, there are other approaches. Indeed, in the recent
preprint  \cite{HWW2}  a coarse approach to this index class is explained. It should not be difficult
to show that the two index classes, the one defined here and the one defined in  \cite{HWW2}  are in fact the same.
(This should proceed as in the case of Galois coverings, where the analogous result is established in
\cite[Proposition 2.4]{PS-Stolz}.)\\
As in the case of Galois coverings, and always under the invertibility assumption \ref{assumption:invertibility},
it should also be possible to establish the existence of  a APS $C^*$-index class through
a boundary value problem defined through the spectral projection $\chi_{(0,\infty)} (D_\partial)$
and show that it is equal to the index class as defined above. See \cite{LP03} for the case of Galois coverings.

%For more on APS index classes 
%we refer the reader to example, in \cite[Section 10.4]{mp} there is an elementary functional analytic approach to the above $C^*$-index class
%in the context of foliations with cylindrical ends. For Galois coverings, there is a direct coarse
%definition of an index class $\Ind^{{\rm coarse}} (D)\in K_0 ( C^* (M_0\subset M,E)^G )$
% (always under the invertibility assumption \ref{assumption:invertibility}); this is based on 
%work of Roe \cite{Roe-partial} and Hanke-Pape-Schick \cite{Hanke-Pape-Schick}. The compatibility with the
%index class as explained here is discussed in \cite[Proposition 2.4]{PS-Stolz}.
%The existence of the index class in Theorem 
%recently discussed by  Hochs-Wang-Wang \cite{HWW1} \cite{HWW2} 
%\textcolor{cyan}{we need to double check that their index class is truly via bvp \`a la APS....}
%; as in the case of Galois coverings 
%we expect this index class 
%to be compatible with the one defined here.

\section{Smooth Atiyah-Patodi-Singer index classes for proper actions}\label{sect:smooth-index-class}

Let us assume that $G$ satisfies the RD condition with respect to a length function $L$
and let us consider $H^\infty_L (G)$, a subalgebra of $ C^*_r (G)$.
In this section we prove the existence of a {\it smooth} APS index class.  More precisely, this means the following:
 we construct 
a dense holomorphically closed subalgebra $\mathcal{A}^{\infty,\epsilon}_G (M,E) $ of $C^*(M_0\subset M,E)^G$ 
and prove that the index class appearing in the statement of Theorem \ref{b-index-class-C*} is in fact an element
in $K_0 (\mathcal{A}^{\infty,\epsilon}_G (M,E))$.

\subsection{Enlarged calculus with bounds}$\;$\\
We are  now going to enlarge the calculus with bounds by adding suitable integral operators.
In this subsection we shall not assume that $G$ satisfies the RD condition.
In this particular subsection we will work with an arbitrary admissible Fr\'echet subalgebra $\mathcal{A} (G)\subset C^*_r G$.\\
We consider the Fr\'echet algebras 
\begin{equation} 
%{}^b \widetilde{A}^\epsilon (M):= 
(\mathcal{A}(G)\hat{\otimes} \,{}^b \Psi^{-\infty,\epsilon}(S))^{K\times K}
\,,\quad
 %\widetilde{A}^\epsilon (M):= 
 (\mathcal{A}(G)\hat{\otimes}  \,\Psi^{-\infty,\epsilon}(S))^{K\times K}
 \end{equation}
 and we set
 \begin{equation}\label{definitionof-b-alg-extended}
{}^b \mathcal{A}^\epsilon _G (M):= 
(\mathcal{A}(G)\hat{\otimes} \,{}^b \Psi^{-\infty,\epsilon}(S))^{K\times K}
+(\mathcal{A}(G)\hat{\otimes}  \,\Psi^{-\infty,\epsilon}(S))^{K\times K}
%{}^b \widetilde{A}^\epsilon (M)+ \widetilde{A}^\epsilon (M) 
 \end{equation}
 and
 \begin{equation}\label{definitionof-residual-alg-extended}
 \mathcal{R}^\epsilon _G (M):= 
(\mathcal{A}(G)\hat{\otimes}  \,\Psi^{-\infty,\epsilon}(S))^{K\times K}
%{}^b \widetilde{A}^\epsilon (M)+ \widetilde{A}^\epsilon (M) 
 \end{equation}

%where, because of Remark \ref{remark-bd}, we are not requiring boundedness 
\noindent
Notice that the  corresponding operators
are automatically bounded on $L^2_b (M)$.

\begin{definition}
We define ${}^b \mathcal{S}^\epsilon _G (M)$ as the subalgebra of the bounded operators on  $L^2_b (M)$
given by
$${}^b \mathcal{S}^\epsilon _G (M) :=\{T_{\widetilde{k}}, \widetilde{k}\in {}^b \mathcal{
A}^\epsilon _G (M)
\}\,.$$
We define 
 $\mathcal{S}^\epsilon _G (M)$ as the subalgebra of the bounded operators on  $L^2_b (M)$
given by
$$ \mathcal{S}^\epsilon _G (M) :=\{T_{\widetilde{k}}, \widetilde{k}\in \mathcal{R}^\epsilon _G (M)\}\,.$$
$ \mathcal{S}^\epsilon _G (M) $ is an ideal in ${}^b \mathcal{S}^\epsilon _G (M)$.
\end{definition}

\medskip\noindent
As in previous sections we shall often treat a space of kernels and the corresponding  space of operators as the same object,
thus identifying ${}^b \mathcal{S}^\epsilon _G (M) $ with ${}^b \mathcal{A}^\epsilon _G (M)$
and $\mathcal{S}^\epsilon _G (M) $ with $\mathcal{R}^\epsilon _G (M)$.

\medskip
\noindent
Notice that from equations  \eqref{identify-b}, \eqref{identify-b-bis}  and \eqref{identify-b-ter} we have that 
$${}^b \Psi^{-\infty,\epsilon}_{G,c}(M)+
\Psi^{-\infty,\epsilon}_{G,c}(M)=: {}^b \mathcal{A}^{c,\epsilon} _G (M)\subset  {}^b \mathcal{A}^\epsilon _G (M)\,,\qquad
\Psi^{-\infty,\epsilon}_{G,c}(M)=:  \mathcal{R}^{c,\epsilon} _G (M)\subset  \mathcal{R}^\epsilon _G (M)
$$
We can similarly define
\begin{equation} 
{}^b \mathcal{A}^{\epsilon}_{G,\RR^+} (\overline{N_+ \partial M}):= 
(\mathcal{A}(G)\hat{\otimes} \,{}^b \Psi^{-\infty,\epsilon}_{\RR^+}(\overline{N_+ \partial S}))^{K\times K}
\end{equation}
and observe that
the indicial operator on the slice $S$ induces a surjective  algebra homomorphism
$${}^b \mathcal{A}^\epsilon _G (M)\xrightarrow{I} {}^b \mathcal{A}^\epsilon _{G,\RR^+} (\overline{N^+ \pa M})$$
where $I$ is equal to the indicial operator on $(\mathcal{A}(G)\hat{\otimes} \,{}^b \Psi^{-\infty,\epsilon}(S))^{K\times K}$
and it is defined as zero on\\
$(\mathcal{A}(G)\hat{\otimes}  \,\Psi^{-\infty,\epsilon}(S))^{K\times K}$
(recall that 
${}^b \mathcal{A}^\epsilon _G (M):= 
(\mathcal{A}(G)\hat{\otimes} \,{}^b \Psi^{-\infty,\epsilon}(S))^{K\times K}
+(\mathcal{A}(G)\hat{\otimes}  \,\Psi^{-\infty,\epsilon}(S))^{K\times K}$)

\medskip
\noindent
Define
\begin{equation}\label{true-ideal}
 \mathcal{A}^\epsilon _G (M):=\Ker \left( I: {}^b \mathcal{A}^\epsilon _G (M)\rightarrow {}^b \mathcal{A}^\epsilon _{G,\RR^+} (\overline{N^+ \pa M}) \right)\,.
 \end{equation}

\medskip
The short exact sequence of algebras
\begin{equation}\label{short-exact-b}
0\to  \mathcal{A}^\epsilon _G (M)\to {}^b \mathcal{A}^\epsilon _G (M)\xrightarrow{I} {}^b \mathcal{A}^\epsilon  _{G,\RR^+}  (\overline{N^+ \pa M})\to 0
\end{equation}
will play a crucial role in this paper.
The Mellin transform on  $\RR^+$-invariant operators on the compactified normal bundle to the boundary of the slice,  
$\overline{N^+\partial S}$, induces an isomorphism 
$\mathcal{M}$ 
between $ {}^b \mathcal{A}^\epsilon  _{G,\RR^+}  (\overline{N_+ \pa M})$ and holomorphic families 
$$\{\RR\times i (-\epsilon,\epsilon)\ni\lambda \to \mathcal{A}_G (\partial M)\}$$
with values in the Fr\'echet algebra $\mathcal{A}_G (\partial M)$, rapidly decreasing in ${\rm Re}\lambda$.
%\footnote{\blue{The notion of holomorphy wih values in a Fr\'echet algebra must be clarified.}}  
The inverse 
$\mathcal{M}^{-1}$ is again given by 
associating to a family $\{S(\lambda)\}$ the $\RR^+$-invariant Schwartz kernel $\kappa_S$ that in projective coordinates is given by 
$$\kappa_S (s,y,y^\prime)=\int_{{\rm Im}\lambda=r} s^{i\lambda} S(\lambda)(y,y^\prime)d\lambda$$
with $r\in (-\epsilon,\epsilon)$.
%\noindent
% \blue{Do we need to say more ? $\rightarrow$} What I wrote works in case 
% $\mathcal{A}(G)$ is made of smooth or continuous functions.
%\footnote{\textcolor{cyan}{FROM PAOLO: the above makes sense but will not give a true smoothing operator
%unless $\mathcal{A}(G)\subset C^\infty (G)$. Notice that
% $$\int_{{\rm Im}\lambda=r} s^{i\lambda} S(\lambda)d\lambda$$
%also makes sense as an  integral of a Fr\'echet-space-valued function. See for example
%{\it Integrals with values in Banach spaces and locally convex spaces}
%by Piotr Mikusinski  or 
%{\it Functions on Manifolds: Algebraic and Topological Aspects}
%by Vladimir Vasilievich Sharko.\\
%I do not think we really need this....\\
%If G is semisimple we can take the HC algebra, which is contained in $H^\infty (G)$
%and is made of smooth functions.}}
Notice that we can  enlarge the whole calculus (i.e. we can consider 
operators of any order)
  by considering
\begin{equation}\label{full-calculus-extended}
{}^b \Psi^{*,\epsilon}_{G,\mathcal{A}}(M): =   {}^b \Psi^{*}_{G,c}(M) +   {}^b \mathcal{A}^\epsilon_G (M)\,.
\end{equation}

%\textcolor{cyan}{But since we basically always use $H^\infty (G)$ it is maybe better to use the $\infty$
%superscript everywhere; this way it is clear when the kernels are compactly supported and when 
%they are not. In any case, we might have to reconsider the whole notation in the b-case...}

\subsection{Rapid Decay condition and a compact notation}\label{subsect:notation}
We assume from now on 
that $G$ satisfies the RD condition with respect to a length function $L$ 
and we choose as an admissible Fr\'echet subalgebra $H^\infty_L (G)\subset
C^*_r(G)$. %We shall often omit the length function $L$ from the notation.
We denote by 
$$\mathcal{A}^{\infty,\epsilon} _G (M)\,,\quad  {}^b \mathcal{A}^{\infty,\epsilon} _G (M)\,,\quad
{}^b \mathcal{A}^{\infty\epsilon}  _{G,\RR^+}  (\overline{N^+ \pa M})
$$
the three algebras introduced in the previous subsection corresponding to 
the particular choice $\mathcal{A}(G)=H^\infty_L (G)$.
We have a short exact sequence of algebras 
$$0\to \mathcal{A}^{\infty,\epsilon} _G (M)\rightarrow {}^b \mathcal{A}^{\infty,\epsilon} _G (M)\xrightarrow{I}
{}^b \mathcal{A}^{\infty,\epsilon}  _{G,\RR^+}  (\overline{N^+ \pa M})\to 0\,.
$$
We shall often omit $\epsilon$ in the algebras appearing in this
short exact sequence  and rewrite it 
as \begin{equation}\label{short-exact-b-bis}
0\to  \mathcal{A}^\infty_G (M)\to {}^b \mathcal{A}^\infty_G (M)\xrightarrow{I} {}^b \mathcal{A}^\infty_{G,\RR^+}  (\overline{N^+ \pa M})\to 0
\end{equation}
Furthermore, we shall also sometimes  drastically simplify the notation and set:
\begin{equation}\label{short-exact-b-newnotation}
\mathcal{A}^\infty:=  \mathcal{A}^\infty_G (M)\,,\quad  {}^b \mathcal{A}^\infty:=  {}^b \mathcal{A}^\infty_G (M)\,,\quad 
 {}^b \mathcal{A}_{\RR^+}^\infty :=  {}^b \mathcal{A}^\infty_{G,\RR^+}  (\overline{N^+ \pa M})
\end{equation}
so that \eqref{short-exact-b-bis} can be rewritten as 
\begin{equation}\label{short-exact-b-tris}
0\to  \mathcal{A}^\infty \to  {}^b \mathcal{A}^\infty \xrightarrow{I}  {}^b \mathcal{A}^\infty_{\RR^+} \to 0
\end{equation}

\subsection{Improved parametrix construction and smooth APS index classes}$\;$\\
%\footnote{\textcolor{m}{It might be cleaner to use the cusp-calculus instead of the b-calculus; the inverse 
%of an elliptic differential b-operator is not in the  b-calculus but, rather, in an enlarged b-calculus
%(named {\it the b-calculus with bounds}); the inverse 
%of an elliptic differential cusp-operator, on the other hand,  is contained in the cusp-calculus itself and this
%simplifies the analysis and the presentation.}}
Let $M_0$ and $M$ as above. We assume $M_0$ to be even dimensional. Let $D$ be a $\ZZ_2$-graded odd 
$G$-equivariant Dirac operator with boundary operator $D_\partial$
satisfying Assumption \ref{assumption:invertibility}.
Let $\epsilon>0$ be strictly smaller than $\alpha/2$, with $[-\alpha,\alpha]$ the width of the spectral gap for the boundary operator, see  \ref{assumption:invertibility}.

%We fix $\epsilon$ once and for all.

%\bigskip
%We shall now construct the (true) parametrix of $D^+$ \footnote{\textcolor{m}{The following will be expanded...}}.
%One begins
%by finding a symbolic parametrix $Q_\sigma$ to $D^+$, with remainders $R^\pm_\sigma$. We can choose
%$Q_\sigma\in {}^b \Psi^{-1}_{G,c}(M)$, i.e. of $G$-compact support; consequently $R^\pm_\sigma\in {}^b \Psi^{-\infty}_{G,c}(M)$.
%Next, by fixing a cut-off function $\chi$ on the collar neighborhood of the
%boundary, equal to 1 on the boundary, we can define a section $s$ to the
%indicial homomorphism $I$. $s$ associates to a $\RR^+$-invariant operator
%$G$ on  $\overline{N^+ \pa M}$ an operator on the $b$-manifold  $M$; the latter is essentially obtained by pre-multiplying
%and post-multiplying by the cut-off function $\chi$. Consider now $I(D^+)$; by assumption we know that
%this operator is $L^2$-invertible on the $b$-cylinder $\overline{N^+ \pa M}$.

\begin{proposition}\label{true-para}
The element  $ I(D^+)^{-1} I( R^-_\sigma)$
is  an element in $ {}^b\mathcal{A}^\infty_{\RR^+}\equiv {}^b \mathcal{A}^{\infty,\epsilon}_{G,\RR^+} (\overline{N^+ \pa M})$.\\ Consequently,
$s (( I(D^+)^{-1} I( R^-_\sigma))$ is an element  in $ {}^b\mathcal{A}^\infty\equiv {}^b \mathcal{A}^{\infty,\epsilon}_G (M)$.
\end{proposition}

\begin{proof} 
By our discussion on the  inverse Mellin
transform  we are lead to consider the $\lambda$-family of operators 
$$ (D_\pa +i\lambda)^{-1}  \circ I( R_\sigma^-,\lambda)$$
and the $\RR^+$-invariant kernel on $\overline{N^+ \pa M}$
whose values at $(s,y,y^\prime)$ is equal to $$\int_{-\infty}^\infty s^{i\lambda} (D_\pa +i\lambda)^{-1}  \circ I( R_\sigma^-,\lambda)(y,y^\prime) d\lambda\,.$$
We need to prove that this kernel is in ${}^b \mathcal{A}^{\infty,\epsilon}_{G,\RR^+} (\overline{N^+ \pa M})$
or, equivalently, that the family
$$\RR\times i (-\epsilon,\epsilon)\ni\lambda \to (D_\pa +i\lambda)^{-1}  \circ I( R_\sigma^-,\lambda)$$ 
is holomorphic and rapidly decreasing as $|{\rm Re} \lambda |\to \infty$ as a function with values in the Fr\'echet algebra
$\mathcal{A}^\infty_G (\partial M)$.
To this end, we first observe that for $\lambda$
ranging in a finite rectangle  we can simply make use of the main result
of Subsections \ref{subsect:resolvant} and \ref{subsect:heat}.
We thus only need to investigate  the large $\lambda$ behaviour of the family of operators
$(D_\pa +i\lambda)^{-1}  \circ I( R_\sigma^-,\lambda)$. We first consider the case $\lambda\in \RR$.\\
We know that $I( R_\sigma^-,\lambda)$ is rapidly decreasing in $\lambda$, with values in 
$\Psi^{-\infty}_{G,c}(\partial M)$; in fact the family $I( R_\sigma^-,\lambda)$, $\lambda\in \RR$, is a smoothing operator in the
pseudodifferential calculus with parameters $\Psi^{-\infty}_{G,c} (\partial M,\Lambda)$, $\Lambda=\{z\in\CC\,:\, {\rm Im} z =0\}$. Notice that $\Lambda$ is an allowable conic set in the complex plane.
Thanks to Proposition \ref{bounded} and the analogue of Proposition \ref{prop:inverse-0}
for first order operators, we can write  $(D_\pa +i\lambda)^{-1}=B_\lambda+  C_\lambda$
with $\{B_\lambda\}\in  \Psi^{-1}_{G,c}(\partial M,\Lambda)$ and of uniform $G$-compact support, and $C_\lambda\in \mathcal{A}^\infty_G (\partial M)$ and {\it bounded} as a function of $\lambda\in\RR$ with values in $\mathcal{A}^\infty_G (\partial M)$. (It is here that crucial use is made of the results established in Section \ref{sect:more-on-heat}.)\\
Let us analyze 
$$(D_\pa +i\lambda)^{-1} I( R_\sigma^-,\lambda)= B_\lambda \circ I( R_\sigma^-,\lambda)+C_\lambda \circ I( R_\sigma^-,\lambda)$$
The first summand  is the composition of an element in $\Psi^{-1}_{G,c} (\partial M,\Lambda)$
and an element in  $\Psi^{-\infty}_{G,c} (\partial M,\Lambda)$ and so in particular it is  a rapidly decreasing functions with values in the Fr\'echet algebra 
$\Psi^{-\infty}_{G,c}(\partial M)$; since, as already remarked,  we have a continuous
inclusion of the Fr\'echet algebra $\Psi^{-\infty}_{G,c}(\partial M)$ into the Fr\'echet
algebra $\mathcal{A}^\infty_{G}(\partial M)$ we conclude that the first summand is a function 
of $\lambda$ which is rapidly decreasing as a function with values in the Fr\'echet algebra
$\mathcal{A}^\infty_{G}(\partial M)$.
The second summand is certainly the composition of a bounded function of $\lambda$ with values in $\mathcal{A}_{G}(\partial M)$
and of a rapidly decreasing function with values in $\mathcal{A}_{G} (\partial M)$ and so, it is also a
rapidly decreasing function with values in $\mathcal{A}_{G} (\partial M)$ as required.\\
We need to extend these results to each horizontal  lines ${\rm Im} z=\pm\delta$, $|\delta|<\epsilon$; however 
for $z\in \{{\rm Im} z=\pm \delta\}$ the indicial family is equal to  
$(D_{\partial M}\mp \delta)+i\lambda$ with $\lambda$ real and we are back to the previous analysis on
$\Lambda=\RR$
but for the operator $(D_{\partial M}\mp \delta)$.\\ 
Summarising, we have proved that 
the family $(D_\pa +i\lambda)^{-1}  \circ I( R_\sigma^-,\lambda)$ is rapidly decreasing as a 
function $\RR\times i (-\epsilon,\epsilon) \ni\lambda \to A^\infty_G (\partial M)$.
We can thus  apply the inverse Mellin transform and finally 
obtain that $ I(D^+)^{-1} I( \R^-_\sigma)\in {}^b \mathcal{A}^{\infty,\epsilon}_{G,\RR^+} (\overline{N^+ \pa M})$ which is what we
wanted to prove. \end{proof}

\medskip
\noindent
{\bf From now on,  unless confusion should arise, we shall omit the $\epsilon$ from the notation.}

\medskip
\medskip
\noindent
The (true) parametrix
of $D^+$ is defined as  $Q^b=Q_\sigma- Q^\prime$
with $Q^\prime$ equal to
$s (( I(D^+)^{-1} I( \R^-_\sigma))$ (recall that $s$ is simply defined as pre-multiplication and post-moltiplication 
by a cut-off function equal to $1$ on the boundary).
Then, with this definition, one can check as in \cite{Melrose}  that $D^+ Q^b=I-{}^b S_-$ and
 $Q^b D^+=I-{}^b S_+$
with ${}^b S_\pm$   in the ideal  $\mathcal{A}^\infty_G (M)$. 
%\footnote{\blue{Make sure we are not using some subtle composition formula.
%In any case, the first equation is clear, maybe for the second we need to proceed as in formula
%(5.177) in Melrose book \cite{Melrose} and this means that we do need to discuss composition formulae a little bit.
%Maybe a trick with adjoints would also work....}}

\medskip

If we consider, as in the closed case,  the $2\times 2$ Connes-Skandalis matrix
\begin{equation}\label{CS-projector-bis}
P^b_{Q}:= \left(\begin{array}{cc} {}^b S_{+}^2 & {}^b S_{+}  (I+{}^b S_{+}) Q^b \\ {}^b S_{-} D^+ &
I-{}^b S_{-}^2 \end{array} \right).
\end{equation}
we obtain  a class 
\begin{equation}\label{CS-class-bis}
%\operatorname{Ind} (D):= 
[P^b_{Q}] - [e_1]\in K_0 (\mathcal{A}^\infty_G (M))\;\;\;\text{with}\;\;\;e_1:=\left( \begin{array}{cc} 0 & 0 \\ 0&1
\end{array} \right)
\end{equation}
Using the fact that $H^\infty_L (G)$ is holomorphically closed in $C^*_r G$ and that the residual
operators on the manifold with boundary $S$, $\Psi^{-\infty} (S,E|_{E})$, are holomorphically closed 
in the compact operators of $L^2_b (S,E|_S)$ one can prove, similarly to the closed
case, that $\mathcal{A}^\infty_G (M)$ is dense and holomorphically closed in $C^* (M_0\subset M,E)^G$.
%\footnote{\blue{We have been a bit brief in Part 1 on this property; we should maybe try to be more precise.}}
We conclude that we have the following natural isomorphism:
\begin{equation}\label{iso-smooth}
K_0 (\mathcal{A}^\infty_G (M))\simeq K_0 (C^* (M_0\subset M,E)^G) %\simeq K_0 (C^* (M_0,E_0)^G)\simeq K_0 (C^*_r (G)).
\end{equation}
so that, summarizing,
$$K_0 (\mathcal{A}^\infty_G (M))\simeq K_0 (C^* (M_0\subset M,E)^G) \simeq K_0 (C^* (M_0,E_0)^G)\simeq K_0 (C^*_r (G)).
$$
%
%where, as already remarked, the second isomorphism is a general result, see \cite[Lemma 1.8]{PS-Stolz} and references therein,
%and the third is proved as in the closed case using the co-compactness of the action of $G$ on $M_0$.
%See  \cite{Hochs-Wang-HC}.

\noindent
The class
\begin{equation}\label{def-b-class}
\operatorname{Ind}_\infty (D):= [P^b_Q] - [e_1]\in K_0 (\mathcal{A}^\infty_G (M))\equiv  K_0 (C^* (M_0\subset M,E)^G)
\end{equation}
is a {\it smooth} representative of the $C^*$-index class constructed in Subsection \ref{subsect:C*-index}.
\begin{remark}
Notice that in contrast with the closed case, in the 
 $b$-case we do not have an index class in the K-theory of an algebra of $G$-compactly supported operators;
 the algebra   $\mathcal{A}^\infty_G (M)$ is in a certain sense the smallest algebra in which we can find 
the index class.
\end{remark}
%
%\subsection{The APS index class defined by the Moscovici-Wu projector}$\;$\\
%If we use instead the Moscovici-Wu projector then we would like to show  that we obtain an equivalent
%description of the index class, that is,
%\begin{equation}\label{MW-class-bis}
%\operatorname{Ind}_\infty (D)= [P^b (D)] - [e_1]\in K_0 (\mathcal{I}_G (M))\;\;\;\text{with}\;\;\;e_1:=\left( \begin{array}{cc} 0 & 0 \\ 0&1
%\end{array} \right)
%\end{equation}
%Here $P^b (D)$ is the matrix as in Moscovici-Wu but with parametrix and remainders improved as above. 
%\blue{A priori it is not clear that the constituents
%of the MW projector, even before we improve the parametrix, are elements in the $b$-calculus. (Notice that
%for the relative class below we shall need the plain definition of the MW projector, thus without using the
%improved parametrix.) The constituents
%of the MW projector are localized around the diagonal because of standard finite propagation speed
%techniques, but we do not know, a priori, that they lift to 
%an element in ${}^b \Psi^{-\infty,\epsilon}_{c,G}$. 
%Here a technique due to Vasy, based on results of Helffer-Sjostrand could be
%useful. This technique employs the resolvant, as the Connes-Moscovici projector in the next
%subsection, and an almost analytic extension of the real functions employed in MW.\\
%There is work to be done here, so unless necessary I would concentrate on Connes-Moscovici.}

\subsection{The smooth APS index class defined by the Connes-Moscovici projector}$\;$\\
In this section we keep assuming that $G$ satisfies
the RD condition. We wish to  analyze the Connes-Moscovici projector.  Recall that this is
the Connes-Skandalis projector but for the parametrix
\begin{equation}
 Q
:= \frac{I-\exp(-\frac{1}{2} D^- D^+)}{D^- D^+} D^+
\end{equation}
with
$I-Q D^+ = \exp(-\frac{1}{2} D^- D^+)$, $I-D^+ Q =  \exp(-\frac{1}{2} D^+ D^-)$.\\
This particular choice of parametrix produces
 the
idempotent
\begin{equation}
\label{cm-idempotent}
V (D)=\left( \begin{array}{cc} e^{-D^- D^+} & e^{-\frac{1}{2}D^- D^+}
\left( \frac{I- e^{-D^- D^+}}{D^- D^+} \right) D^-\\
e^{-\frac{1}{2}D^+ D^-}D^+& I- e^{-D^+ D^-}
\end{array} \right)
\end{equation}
We have two goals  in this subsection:

\begin{itemize}
\item we want to show that on a $b$-manifold the entries of this matrix belongs to the unitalization
of ${}^b \mathcal{A}^\infty_G (M)$;
\item we want to show that if we improve the parametrix $Q$ to $Q^b:= Q-Q^\prime$, with
$$Q^\prime:= s (I(D^+)^{-1} I( \exp(-\frac{1}{2} (D^- D^+)))\equiv 
s ((D^+_{{\rm cyl}})^{-1} \exp (-\frac{1}{2} D^-_{{\rm cyl}}D^+_{{\rm cyl}}))
$$ then the resulting Connes-Moscovici
projector $V^b(D)$ is an element in the unitalization of $\mathcal{A}^\infty_G (M)$
\end{itemize}

\medskip
\noindent
{\bf Important remark:} a manifold with cylindrical ends is certainly a complete manifolds and so
the heat kernel will have the known decaying properties at infinity. However, our goal in this section is to prove that
not only the heat kernel has the right decaying properties at infinity, but that it is also a $b$-pseudodifferential
operator. This will be crucial in the proof. This is precisely the reason why we provide a detailed
analysis of the Cauchy integral below.

\medskip
\noindent
{\bf Notation:} notice that we have thus denoted by  $V^b(D)$ the Connes-Moscovici projector associated to 
the improved parametrix $Q^b= Q-Q^\prime$. 

\medskip
\noindent
This brings us first of all to the analysis of the heat kernel $\exp (-t D^2)$ on a $b$-manifold.
In a second stage we shall also want to analyze the operator 
\begin{equation}\label{right-upper-corner}
e^{-\frac{1}{2}D^- D^+}
\left( \frac{I- e^{-D^- D^+}}{D^- D^+} \right) D^-
\end{equation}

%If $M$ has an equivariant Spin$_c$ structure and $G/K$ is even dimensional then we could 
%proceed as in Hochs-Wang and express the heat kernel on $M$ in terms of a function on $G$ and 
%the heat kernel for a Dirac Laplacian induced on the slice. The slice is a $b$-manifold and the heat kernel there has been
%analysed in great detail by Melrose in his book on APS. In particular, the heat kernel of the Dirac Laplacian
%induced on the slice is a smoothing $b$-operator in the small $b$-calculus.
%This  means that under the above hypothesis we should be able to prove that $\exp (-t D^2)\in {}^b \mathcal{A}^\epsilon_G (M)$
%with appropriate choices of admissible algebras $\mathcal{A}(G)$:
%for example, if  $G$ is semisimple and $\mathcal{A} (G)$ is the HC algebra; or if $G$ satisfies $RD$ and  $\mathcal{A} (G)= H^\infty_L (G)$. See Proposition \ref{heat-hochs-wang}.
%\blue{{\it However, this would still leave us with the problem of analysing \eqref{right-upper-corner}.}}
   
   \noindent
   We shall 
analyse both $\exp (-t D^2)$ and \eqref{right-upper-corner} via a Cauchy integral and a detailed analysis of the
resolvant, as we did in the closed case.

\medskip
\noindent
{\bf Analysis of the heat operator.}\\
Let us begin with $\exp (-t D^2)$ and let us write it as 
$$\exp (-t D^2) =\int_\gamma e^{-t\mu} (D^2-\mu)^{-1} d\mu$$
with $\gamma$ a suitable path in $\CC\setminus [0,\infty)$.%, for example the one we choose for a manifold without boundary. 
 We shall follow the analysis of the complex powers of an elliptic $b$-operator given in \cite{piazza-jfa} and the study
of the large time bahaviour of the heat kernel in \cite{Melrose}. These are based
on the analysis of $(D^2-\mu)$ as a parameter-dependent $b$-differential operator. The analysis  is  similar
to the one for the term $D_{\partial M} + i \mu$ given in the proof of Proposition \ref{true-para}
once we use the $b$-calculus instead of the usual calculus on closed manifolds such as $\partial M$. 
%First, let us fix $\gamma$: we choose $\gamma$ equal to union of the semi-lines $\ell_\pm:= \{{\rm Im}z=\pm\delta, {\rm Re z}\geq 0\}$
%and the half-circe $\{\delta e^{i\theta}, \theta\in [-\pi/2,\pi/2]\}$ with  $\delta<\epsilon$ and $\epsilon$ such that ${\rm spec}_{L^2} (D_\partial)\cap (-\epsilon,\epsilon)=\emptyset$. \blue{Probably should take $\gamma$ as usual; see [Melrose]}

\begin{proposition}\label{b-resolvant}
There exists a $\delta > 0$ such that for each $\mu\in\gamma$ we have 
$$(D^2-\mu)^{-1}\in  {}^b \Psi^{-2}_{G,c} (M)+ {}^b \mathcal{A}^{\infty,\delta}_{G} ( M)$$
Here we have gone back to the precise notation, including $\delta$, for the extended calculus with bounds.
\end{proposition}

\begin{proof}
First, proceeding as in Shubin's classical monograph \cite{shubin-book}, see \cite{piazza-jfa}
for the details, we find a symbolic parametrix in the small $b$-calculus
with parameters, $B^\sigma_\mu$. We denote by  $R^\sigma_\mu$ and $S^\sigma_\mu$ the remainders; thus with obvious notation
$B^\sigma_\mu\in {}^b \Psi^{-2}_{G,c} (M,\Lambda)$, with $\Lambda$ a suitable sector in the complex plane,
and $R^\sigma_\mu, S^\sigma_\mu \in {}^b \Psi^{-\infty}_{G,c} (M,\Lambda)$. In particular $R^\sigma_\mu, S^\sigma_\mu $ are rapidly decreasing in $\mu$
as ${\rm Re}\mu\to \pm \infty$
as functions  with values in ${}^b \Psi^{-\infty}_{G,c} (M)$.
We then consider $I(D^2-\mu)^{-1} I(R^\sigma_\mu)$. Notice that $I(D^2-\mu)^{-1}$ does exists; we can
argue as at the end of Subsection \ref{subsect:b-dirac} and see easily that the indicial
family $I(D^2-\mu,\lambda)= D^2_{\partial} + \lambda^2 -\mu$ is indeed invertible
for $\lambda$ real and $\mu  \in\gamma$.  %\blue{Need $\lambda$ in the strip} 
Put it differently
$${\rm spec}_b (D^2-\mu)\cap \{{\rm Im} z =0 \}=\emptyset$$
for each $\mu\in\gamma$. Thus
$$(\cap_{\mu} {\rm spec}_b (D^2-\mu))\cap \{{\rm Im} z =0 \}=\emptyset$$
and so there is a $\delta >0$ such that 
$$(\cap_{\mu} {\rm spec}_b (D^2-\mu))\cap \{|{\rm Im} z |<\delta\}=\emptyset$$
For each fixed $\mu\in\gamma$ we know from Proposition \ref{true-para}
that $I(D^2-\mu)^{-1} I(R^\sigma_\mu)\in {}^b \mathcal{A}^{\infty,\delta}_{G,\RR^+} (\overline{N^+ \pa M})$. Thus if $B_\mu= B^\sigma_\mu- s (I(D^2-\mu)^{-1} I(R^\sigma_\mu))$ then
$$(D^2-\mu)\circ B_\mu= \Id + R_\mu$$
with $B_\mu\in{}^b \Psi^{-2}_{G,c} (M)+ {}^b \mathcal{A}^{\infty,\delta}_{G} ( M)$ and  $R_\mu \in \mathcal{A}^{\infty,\delta}_G (M)$ and rapidly descreasing in ${\rm Re} \mu$. Thus $(D^2-\mu)^{-1}$ is equal to $B_\mu (\Id + R_\mu)^{-1}$
and since $\mathcal{A}^{\infty,\delta}_G (M)$ is holomorphically closed we see that 
 $$(D^2-\mu)^{-1}= B_\mu + B_\mu  \circ L_\mu$$
 with $L_\mu\in \mathcal{A}^{\infty,\delta}_G (M)$ and with $B_\mu\in  {}^b \Psi^{-2}_{G,c} (M)+ {}^b \mathcal{A}^{\infty,\delta}_{G} ( M)$
 so that, finally,
  $$(D^2-\mu)^{-1}\in  {}^b \Psi^{-2}_{G,c} (M)+ {}^b \mathcal{A}^{\infty,\delta}_{G} ( M)$$
  which is what we wanted to show.
\end{proof}

%\red{Need to understand the behaviour in $\mu$ for $\mu$ large, so as to conclude that
%$$\int_\gamma e^{-t\mu} (D^2-\mu)^{-1} d\mu\;\in\;  {}^b \mathcal{A}_{G} ( M)\,.$$
%Clearly it would suffice to show that the seminorms of $(D^2-\mu)^{-1}$, as an element in $  {}^b \Psi^{-2}_{G,c} (M)+ {}^b \mathcal{A}_{G} ( M)$,
%are bounded in $\mu$.\\
%Below is an attempt.}

We now consider 
$$\int_\gamma e^{-t\mu} (D^2-\mu)^{-1} d\mu$$
and our goal is to show that it belongs to ${}^b \mathcal{A}^{\infty,\delta}_{G} ( M)$. 
From now on we have fixed $\delta$ as in Proposition \ref{b-resolvant}.% and we omit to report it in the notation.

\begin{proposition}\label{prop:b-heat}
The heat operator $\exp (-t D^2)$ is defined by a kernel  in ${}^b \mathcal{A}^{\infty}\equiv {}^b \mathcal{A}^{\infty,\delta,}_{G} ( M)$.
\end{proposition}

\begin{proof}
We can write 
$$\int_\gamma e^{-t\mu} (D^2-\mu)^{-1} d\mu= \int_\gamma e^{-t\mu} B_\mu   d\mu\;+\; \int_\gamma e^{-t\mu} B_\mu  \circ L_\mu  d\mu$$
%The operator on the left hand side certainly make sense as  a bounded operator.
We need to make sense of the individual summands on the right hand side and show that they both belong to ${}^b \mathcal{A}^{\infty,\delta,}_{G} ( M).$
Let us analyze the first summand on the right hand side. This can be written as
\begin{equation}\label{int-of-Blambda} \int_\gamma e^{-t\mu}  B^\sigma_\mu d\mu - \int_\gamma e^{-t\mu}  s (I(D^2-\mu)^{-1} I(R^\sigma_\mu))d\mu
\end{equation}
Proceeding as in \cite[Prop. 3.36]{piazza-jfa} it is easy to show that the purely symbolic term
$$\int_\gamma e^{-t\mu}  B^\sigma_\mu d\mu $$
is an element in ${}^b \Psi^{-\infty}_{G,c}(M)\subset  {}^b \mathcal{A}^{\infty,\delta,}_{G} ( M)$.
It remains to show that the second summand in \eqref{int-of-Blambda}, that is
\begin{equation}\label{second-summand}
- \int_\gamma e^{-t\mu}  s (I(D^2-\mu)^{-1} I(R^\sigma_\mu))d\mu\,,
\end{equation}
is well defined and an element in  ${}^b \mathcal{A}^{\infty,\delta,}_{G} ( M)$. 
To this end we need to analyse the term $s(I(D^2-\mu)^{-1} I(R^\sigma_\mu)).$ We need 
to show that this is at least bounded as a function of $\mu$ with values 
${}^b \mathcal{A}^{\infty,\delta}_G (M)$.
As usual, the  large $\mu$
behaviour is what we are concerned with. 
%and we shall in fact prove that is rapidly decreasing in $\mu$.
Now
$$I(D^2-\mu)^{-1} I(R^\sigma_\mu)=\int_\RR s^{i\lambda}
(D^2_\partial + \lambda^2 -\mu)^{-1} I (R^\sigma_\mu,\lambda)d\lambda$$
with $I (R^\sigma_\mu,\lambda)$  rapidly decreasing jointly in $\lambda$  and in  ${\rm Re} \mu$. 
This integral is analysed as  in the proof of Proposition \ref{true-para}
(see also
\cite[Lemma 7.35]{Melrose} for this).
Thus we can certainly write
$$(D^2_\partial + \lambda^2 -\mu)^{-1}=\Sigma(\mu,\lambda)+ \Sigma(\mu,\lambda)\circ \Omega (\mu,\lambda)$$
with $\Sigma$ bounded as a function with values in $\Psi^0_{G,c}(\partial M)$ and 
$\Omega (\mu,\lambda)$ bounded as a function with values in $\mathcal{A}^\infty_G (\partial M)$.
Proceeding as in the proof of Proposition \ref{true-para}, using implicitly Lemma \ref{lemma:composition}, we conclude 
 that 
$$(D^2_\partial  -\mu)+ \lambda^2)^{-1}  I (R^\sigma_\mu,\lambda)$$
defines  a function 
$$(\RR\times i (-\delta,\delta))\times \gamma \to \mathcal{A}^\infty_G (\partial M)$$
which is rapidly decreasing  as $| {\rm Re}\lambda|$ and $|{\rm Re} \mu|$ go to infinity. 
%\footnote{\blue{Notice that, again, we are stating implicitly that the composition $ \Sigma(\mu,\lambda)\circ \Omega (\mu,\lambda)$ 
%is bounded. Here we could get away with a trick employing adjoints and so on;} \red{however for
%the term $ \int_\gamma e^{-t\mu} B_\mu  \circ L_\mu  d\mu$ treated ahead, as in the closed case, we cannot avoid the continuity of the composition. }} 
Consequently,
$I(D^2-\mu)^{-1} I(R^\sigma_\mu)$ is rapidly decreasing as a ${}^b \mathcal{A}^{\delta,\infty}_{G} ( \overline{N^+ \partial M})$-valued map and so
\eqref{second-summand} is an element in ${}^b \mathcal{A}^{\infty,\delta,}_{G} ( M)$.
We finally come to the integral
$$ \int_\gamma e^{-t\mu} B_\mu  \circ L_\mu  d\mu$$
that we rewrite as
$$ \int_\gamma e^{-t\mu}B^\sigma_\mu\circ L_\mu d\mu-  \int_\gamma e^{-t\mu} 
s (I(D^2-\mu)^{-1} I(R^\sigma_\mu))\circ L_\mu d\mu\,.$$
The first term is analysed as in the closed case. Here the analogue of Lemma  \ref{lemma:composition}
but for the composition of  0-order G-compactly supported b-operators with elements  in in $\mathcal{A}^{\infty,\delta}_G (M)$
is used: the proof is exactly the same since the slice-decomposition allows us to use the result
of Lemma  \ref{lemma:composition} on the $G$-component whereas the usual $b$-calculus is used in the $S$ component.
Thus, the first summand  is the integral of a rapidly decreasing function with values 
in $\mathcal{A}^{\infty,\delta}_G (M)\subset {}^b \mathcal{A}^{\infty,\delta,}_{G} ( M)$ and so it is certainly an element in  ${}^b \mathcal{A}^{\infty,\delta,}_{G} ( M)$ ;
the integrand in the second summand is the composition of a rapidly decreasing function with values in 
${}^b \mathcal{A}^{\infty,\delta,}_{G} ( M)$ with a rapidly decreasing function with values in $ \mathcal{A}^{\infty,\delta}_{G} ( M)$
and so, composition being continuous in a Fr\'echet algebra, it is certainly a rapidly decreasing function
with values in ${}^b \mathcal{A}^{\infty,\delta,}_{G} ( M)$ and so the integral is well defined as an element
in ${}^b \mathcal{A}^{\infty,\delta,}_{G} ( M)$.
\end{proof}

\medskip
\noindent
{\bf Analysis of the Connes-Moscovici projector(s).}\\
We finally discuss the Connes-Moscovici matrix
\begin{equation}
\label{cm-idempotent-bis}
V(D)=\left( \begin{array}{cc} e^{-D^- D^+} & e^{-\frac{1}{2}D^- D^+}
\left( \frac{I- e^{-D^- D^+}}{D^- D^+} \right) D^-\\
e^{-\frac{1}{2}D^+ D^-}D^+& I- e^{-D^+ D^-}
\end{array} \right)
\end{equation}
\begin{proposition}\label{CM-for-relative}
The matrix $V(D)$  has its entries in (the unitalization of)
 ${}^b \mathcal{A}^{\infty}\equiv {}^b \mathcal{A}^{\infty,\delta}_G (M)$.
\end{proposition}

\begin{proof}
The only  entry  in  \eqref{cm-idempotent-bis} that we have not discussed
is the one in the right upper corner. This entry has the form $f(D)D^-$, where
$f(z):=e^{-z^2/2}((1-e^{-z^2})/z^2)$. Notice that $f$ is holomorphic. It suffices to show that $f(D)\in {}^b \mathcal{A}^{\infty}$.
Let $g(z)= e^{-z/2}((1-e^{-z})/z)$; then 
$$f(D)=g(D^2)=\frac{1}{2\pi i}\int_\gamma g(\mu) (D^2-\mu)^{-1} d\mu$$
where the integral certainly makes sense as a bounded operator on $L^2_b$.
To improve this result and show that the integral is actually an element in ${}^b \mathcal{A}^{\infty}$
we can repeat the arguments that have been given for the 
heat operator, using our knowledge of the properties of the resolvant $(D^2-\mu)^{-1}$. We omit the details since they would be a repetition of those already given for $\exp(-t D^2)$.
\end{proof}

Next we consider the $b$-Connes-Moscovici projector produced by the improved parametrix $Q^b:=Q- Q^\prime$ with
\begin{equation}%\label{gprime}
Q^\prime := - \chi ((D_{\cyl}^+)^{-1}  \exp(-\frac{1}{2} D^+_{\cyl} D^-_{\cyl}) \chi \,.
\end{equation}
and $\chi$ a cut-off function equal to $1$ on the boundary of our $b$-manifold.
This gives us
\begin{align}%\label{g-prime-d-plus}
Q^\prime D^+ &= -\chi  \exp(-\frac{1}{2} D^-_{\cyl} D^+_{\cyl}) \chi  + \chi (D^+_{\cyl})^{-1}   \exp(-\frac{1}{2} D^+_{\cyl} D^-_{\cyl})
\mathop{cl}(d\chi)\\  D^+ Q^\prime &= -\chi  \exp(-\frac{1}{2} D^+_{\cyl} D^-_{\cyl})^{-1} )\chi - \mathop{cl}(d\chi)
(D^+_{\cyl})^{-1}  \exp(-\frac{1}{2} D^+_{\cyl} D^-_{\cyl})\chi\,.
\end{align}
This means that the remainder of the improved parametrix 
${}^b S_+:= I-Q^b D^+=I- (Q-Q^\prime)\, D^+=\exp(-\frac{1}{2} D^+ D^- ) + Q^\prime D^+ $%, a  residual operator, 
is
explicitly given by
\begin{equation}\label{s+w}
{}^b S_+ = \exp(-\frac{1}{2} D^+ D^- )
     -\chi \exp(-\frac{1}{2} D^+_{\cyl} D^-_{\cyl})\chi
   + \chi (D^+_{\cyl})^{-1}   \exp(-\frac{1}{2} D^+_{\cyl} D^-_{\cyl})
\mathop{cl}(d\chi).
   \end{equation}
    A similar expression can be found for ${}^b S_-$.
Substituting ${}^b S_\pm$ and $Q^b$ at the place of
$S_\pm$ and $Q$ into the expression of the Connes-Skandalis projection
%\begin{equation*} \left(\begin{array}{cc} \mathcal{S}_{+}^2 & \mathcal{S}_{+}  (I+\mathcal{S}_{+}) \Q\\ \mathcal{S}_{-}\D^+ &
%I-\mathcal{S}_{-}^2 \end{array} \right).
%\end{equation*}
we obtain $V^b (D)$.

\begin{proposition}
The matrix  $V^b (D)$ has its entries in (the unitalization of)
  $\mathcal{A}^{\infty} \equiv \mathcal{A}^{\infty,\delta}_G (M)$.
\end{proposition}

\begin{proof}
${}^b S_\pm$ is an element in $ \mathcal{A}^\infty$, since it is in ${}^b \mathcal{A}^{\infty}\equiv {}^b \mathcal{A}^{\infty,\delta}_G (M)$
and has vanishing indicial operator. This takes care of all the entries but the one in the right upper corner,
namely ${}^b S_+ (1+{}^b S_+ )Q^b $. This can be rewritten as 
 $${}^b S_+ (1+{}^b S_+ )Q -  {}^b S_+ (1+{}^b S_+ )Q^\prime $$
 From the expression of ${}^b S_+$ and what has been already proved above
 we know that ${}^b S_+ (1+{}^b S_+ )Q\in {}^b \mathcal{A}^\infty$. It remains to see that
 ${}^b S_+ (1+{}^b S_+ )Q^\prime $ is in ${}^b \mathcal{A}^\infty$. However, this is clear
 because ${}^b S_+ (1+{}^b S_+ )$ is in $ \mathcal{A}^\infty\subset {}^b \mathcal{A}^\infty$
 and $Q^\prime\equiv - \chi ((D_{\cyl}^+)^{-1}  \exp(-\frac{1}{2} D^+_{\cyl} D^-_{\cyl}) \chi$
 is also clearly in ${}^b \mathcal{A}^\infty$ given that $(D_{\cyl}^+)^{-1}  \exp(-\frac{1}{2} D^+_{\cyl} D^-_{\cyl})$
 is obtained by inverse Mellin trasform  of the family
 $$(D_{\partial}+i\lambda)^{-1} \exp (-\frac{1}{2} (\lambda^2 + D^2_{\partial}))$$
which is holomorphic and rapidly decreasing as $|{\rm Re} \lambda|\to\infty$ in the region $\{z\in\CC\;|\;
|{\rm Im} z|<\delta\}$. Summing up: ${}^b S_+ (1+{}^b S_+ )Q^b $ is 
an element in ${}^b \mathcal{A}^\infty$ and has vanishing indicial operator; it is therefore 
 in $ \mathcal{A}^\infty$ as stated.
\end{proof}

\begin{corollary}
If $M_0$ is an even-dimensional  G-proper manifold with boundary,
with associated manifold with cylindrical ends $M$, and if 
 $D$ is a $\ZZ_2$-graded odd Dirac operator on $M_0$ with boundary operator
satisfying Assumption \ref{assumption:invertibility}, then  
there is a $\delta>0$ and a well defined Connes-Moscovici class
\begin{equation}\label{index-b-class-CM}
 [V^b (D)]-[e_1] \in K_0 (\mathcal{A}^{\infty,\delta}_G (M))\;\;\;\text{with}\;\;\;e_1:=\left( \begin{array}{cc} 0 & 0 \\ 0&1
\end{array} \right)
\end{equation}
Moreover,
\begin{equation}\label{index-b-class}
[V^b (D)]-[e_1]=\operatorname{Ind}_\infty (D) \;\;\text{in}\;\; K_0 (\mathcal{A}^{\infty,\delta}_G (M))=K_0 (C^*(M_0\subset M)^G)\end{equation}
with $\operatorname{Ind}_\infty (D)$, defined in \eqref{def-b-class}, a smooth representative 
of the $C^*$-index class $\Ind_{C^*}(D)$.
\end{corollary}

\medskip
\noindent
{\bf From now on we fix a $\delta$ as in the previous propositions but omit it from the notation.}

\medskip

\subsection{The relative index class and excision}$\;$\\
Let $0\to J\to A\xrightarrow{\pi} B\to 0$ a short exact sequence of Fr\'echet
  algebras. We recall that $K_0 (J):=
K_0 (J^+,J)\cong \Ker (K_0 (J^+)\to \ZZ)$
and that $K_0 (A^+,B^+)= K_0 (A,B)$ with $( )^+$ denoting unitalization.
See \cite{Bla,hr-book,LMP2}.
Recall that a relative $K_0$-element
for $ A\xrightarrow{\pi} B$ with unital algebras $ A, B$
is represented by a  triple $(P,Q,p_t)$ with $P$ and $Q$ idempotents in $M_{n\times n} (A)$
and $p_t\in M_{n\times n} (B)$ a
path of idempotents connecting $\pi (P)$ to $\pi (Q)$.
The excision  isomorphism
\begin{equation}\label{excision-general}
\alpha_{{\rm ex}}: K_0 (J)\longrightarrow K_0 (A,B)
\end{equation}
is given by
%\begin{equation*}
$\alpha_{{\rm ex}}([(P,Q)])=[(P,Q,{\bf c})]$
%\end{equation*}
with ${\bf c}$ denoting the constant path.

Consider the Connes-Moscovici projections $V (D)$ and $V(D_{\cyl})$
associated to $D$ and $D_{\cyl}$.
With $e_1:=\begin{pmatrix} 0&0\\0&1 \end{pmatrix}$ consider the  triple
%$(W_{\D}, \begin{pmatrix} 0&0\\0&1 \end{pmatrix}, q_t)$,
%$t\in [0,+\infty]$,
%with
\begin{equation}\label{pre-wassermann-triple}
(V(D), e_1, q_t)
\,, \;\;t\in [1,+\infty]\,,\;\;\text{ with }
q_t:= \begin{cases} V(t D_{\cyl})
\;\;\quad\text{if}
\;\;\;t\in [1,+\infty)\\
e_1 \;\;\;\;\;\;\;\;\;\;\;\;\;\,\text{ if }
\;\;t=\infty
 \end{cases}
\end{equation}
%Similarly, we can
%  consider the triple
%\begin{equation}\label{wassermann-triple}
%(V_{\D}^*, e_1, V_{(t\D_{\cyl})}^*)
%\,, \;\;t\in [1,+\infty].
%\end{equation}

\begin{proposition}\label{prop:relative-indeces}
Under the invertibility assumption \eqref{invertibility}, the Connes-Moscovici idempotents $V(D)$
and  $V(D_{\cyl})$ define  through formula
  \eqref{pre-wassermann-triple},
a relative class in $K_0 ({}^b \mathcal{A}^\infty_G (M),{}^b \mathcal{A}^\infty _{G,\RR^+}  (\overline{N_+ \pa M}))$,
 the relative group 
associated to  the short exact sequence
\begin{equation*}
0\to  \mathcal{A}^\infty_G (M)\to {}^b \mathcal{A}^\infty_G (M)\xrightarrow{I} {}^b \mathcal{A}^\infty_{G,\RR^+}  (\overline{N_+ \pa M})\to 0
\end{equation*}
With a small abuse of notation we  denote this class by $[V_{\D}, e_1, V_{(t\D_{\cyl})}]$.
%Similarly  the adjoint idempotents define
%  a relative class  $[V_{\D}^*, e_1, V_{(t\D_{\cyl})}^*]\in K_0 (\mathfrak{A},\mathfrak{G})$.  These two classes are equal.
\end{proposition}
%\smallskip
%\noindent
%We shall use the equality $[V_{\D}, e_1, V_{(t\D_{\cyl})}]=[W_{\D}, e_1, W_{(t\D_{\cyl})}]$
%in the proof of our main Theorem, in the short-time heat-kernel computation involved in its proof.
\begin{proof}
The fact that the two idempotents have entries in the right algebras follows immediately
from Proposition \ref{CM-for-relative}.
The triple $(V(D), e_1, V(t\D_{\cyl}))$ defines a relative class because of the invertibility of $D_{\cyl}$
and well-known properties of the heat kernel.% \blue{Need more ?}
 \end{proof}

\begin{definition}\label{def:rel-index-smooth}
We define the relative (smooth) index class as
$$\Ind_\infty (D,D_\partial):= [V_{\D}, e_1, V_{(t\D_{\cyl})}]\;\;\in\;\;K_0 ({}^b \mathcal{A}^\infty_G (M),{}^b \mathcal{A}^\infty _{G,\RR^+}  (\overline{N_+ \pa M}))\,.$$
\end{definition}

The following result will play a crucial role:

\begin{theorem}\label{theo:excision}
Let $$\alpha_{{\rm ex}}: K_0 ( \mathcal{A}^\infty_G (M))\to K_0 ({}^b \mathcal{A}^\infty_G (M),{}^b \mathcal{A}^\infty _{G,\RR^+}  (\overline{N_+ \pa M}))$$ be the
excision isomorphism for the short exact sequence $0\to  \mathcal{A}^\infty_G (M)\to {}^b \mathcal{A}^\infty_G (M)\xrightarrow{I} {}^b \mathcal{A}^\infty_{G,\RR^+}  (\overline{N_+ \pa M})\to 0$.\\
 Then
\begin{equation}\label{excision-for-cs}
\alpha_{{\rm ex}}(\operatorname{Ind}_\infty (D))=
\operatorname{Ind}_\infty (D,D_{\pa})\;\;\text{in}\;\; K_0 ({}^b \mathcal{A}^\infty_G (M),{}^b \mathcal{A}^\infty_{G,\RR^+}  (\overline{N_+ \pa M}))
\end{equation}
\end{theorem}

\begin{proof}
We need to show that
$$\alpha_{{\rm ex}}([V^b (D)]-[e_1])=[V(D), e_1, V(t\D_{\cyl})]\,.$$
As explained in \cite{GMPi}, we can adapt the proof of the analogous equality
 given in \cite{mp} for the graph projection.
As in \cite{GMPi} we omit the elementary details that allow to pass from the graph projection to the Connes-Moscovici
projection.
 \end{proof}

\section{Relative cyclic cocycles}\label{sect:relative-cocycles}
In this section we shall generalize the construction in \S \ref{section:van-est} of cyclic cocycles on algebras of $G$-invariant
smoothing operators on closed manifolds to obtain {\em relative} cocycles associated to smooth group cocycles. 

First, recall that the relative cyclic complex associated to a short exact sequence $0\to J\to A\xrightarrow{\pi}B\to 0$ of algebras
is given by 
\[
CC^k(A,B):=CC^k(A)\oplus CC^{k+1}(B),
\]
equipped with the differential
\[
\left(\begin{matrix} b+B&-\pi^*\\ 0&-(b+B)\end{matrix}\right),
\]
where $b,B$ are the usual Hochschild and cyclic differential and $\pi^*$ denotes the pull-back of functionals through the surjective morphism 
$\pi:A\to B$.
In our case, the relevant extension is given, first of all, by 
\begin{equation*}
0\to  \mathcal{A}^c_G (M)\to {}^b \mathcal{A}_{G}^c (M)\xrightarrow{I} {}^b \mathcal{A}_{G,\RR^+}^c  (\overline{N^+ \pa M})\to 0\,,\quad   \mathcal{A}^c_G (M):= \Ker I,
\end{equation*}
where this is now, for simplicity, the short exact sequence for the small $b$-calculus. 

%we omit the $\delta$ of the calculus with bounds in the notation.

%\textcolor{cyan}{Once again, here and below we should really consider $ {}^b \mathcal{A}_{G}^{c,\epsilon} (M)$
%etc etc....}\\
With the global slice $S\subset M$, we can view $A\in {}^b \mathcal{A}^c_G (M)$ as a map $\Phi_A:G\to {}^b\Psi^{-\infty}(S)$ 
by setting $\Phi_A(g,s_1,s_2):=A(s_1,gs_2)$. Likewise, an element $B\in {}^b \mathcal{A}^c_{G,\RR^+}  (\overline{N^+ \pa M})$
gives rise, by  Mellin transform, to a map $\Phi_B:G\times \mathbb{R}\to {}\Psi^{-\infty}(\partial S)$, which is 
compactly supported on $G$, and rapidly decreasing on $\RR$. In the following, we shall denote by $\Phi_A\mapsto \hat{\Phi}_A$  the morphism $I:{}^b \mathcal{A}_{G}^c (M)\rightarrow {}^b \mathcal{A}_{G,\RR^+}^c  (\overline{N^+ \pa M})$ followed by the Mellin transform.

Before we construct cyclic cocycles, we shall construct the correct analogue of the $b$-trace of Melrose in this setting where we have a proper group action. In our geometric setting, a choice of cut-off function $\chi$
for the action of $G$ on $M_{0}$ %satisfying \eqref{prop-cut-off}, 
restricts to give a cut-off function $\chi_{\partial M_{0}}:=\chi |_{\partial M_{0}}$ for 
the $G$-action on $\partial M_{0}$. We shall also write briefly $\chi_{\partial}$.
We consider as usual the associated $b$-manifold $M$, endowed with a product-type $b$-metric $g$, so that,
metrically, $M$ is a manifold with cylindrical ends, and we shall, by a small abuse of notation, write $\chi$ for the extension of $\chi$ on $M_0$ which is constant
in the cylindrical coordinate.

Using the $b$-integral of Melrose, see \cite{Melrose}, we now define
\begin{definition}
For $A\in {}^b \mathcal{A}_{G}^c(M)$ its $~^{b}\,G$-trace is defined as
\[
{}^b {\rm Tr}_{\chi}(A):=\bint
 K_A(x,x)\chi(x)d{\rm vol}(x).
\]
\end{definition}
Remark that the cut-off function on $M_0$
has compact support, so the $b$-regularized integral is indeed well defined.
The argument in \cite[Prop. 2.3]{PPT} shows that ${}^b {\rm Tr}_{\chi}$ is independent of the choice of 
cut-off function $\chi$. Next, using the trick with the family $\{\chi_{\epsilon}\}_{\epsilon>0}$ of cut-off functions converging to 
the characteristic function on $S$, we can rewrite
\begin{equation}
\label{b-trace}
{}^b {\rm Tr}_{\chi}(A)={}^b{\rm Tr}_{S}\left(\Phi_A(e)\right).
\end{equation}
As in the usual $b$-calculus, ${}^b {\rm Tr}_{\chi}$ is not a trace, but we have a precise formula for its defect on commutators:
\begin{lemma}
\label{btrace}
For $A_{1},A_{2}\in  {}^b \mathcal{A}_{G}^c (M)$, we have
\[
{}^b {\rm Tr}_{\chi}\left([\Phi_{A_{1}},\Phi_{A_{2}}]\right)=\frac{i}{2\pi}\int_{\mathbb{R}}\int_G{\rm Tr}_{\partial S}\left(\frac{\partial I(\Phi_{A_1},h^{-1},\lambda)}{\partial \lambda}\circ
I(\Phi_{A_2},h,\lambda)\right)dh d\lambda.
\]
\end{lemma}
\begin{proof}
With the expression \eqref{btrace} for the trace, this follows straightforwardly from the usual trace-defect for the $b$-trace
on the compact manifold $S$:
\begin{align*}
{}^b {\rm Tr}_{\chi}\left([\Phi_{A_{1}},\Phi_{A_{2}}]\right)=&
\int_G{}^b{\rm Tr}_S\left(\Phi_{A_1}(h^{-1})\circ \Phi_{A_2}(h)-\Phi_{A_2}(h^{-1})\circ \Phi_{A_1}(h)\right)dh\\
&=\int_G{}^b{\rm Tr}_S\left(\left[\Phi_{A_1}(h^{-1}),\Phi_{A_2}(h)\right]\right)dh\\
&=\frac{i}{2\pi}\int_{\mathbb{R}}\int_G{\rm Tr}_{\partial S}\left(\frac{\partial I(\Phi_{A_1},h^{-1},\lambda)}{\partial \lambda}\circ
I(\Phi_{A_2},h,\lambda)\right)dh d\lambda.
\end{align*}
Here we have changed integration variables $h\mapsto h^{-1}$ to go to the second line, and used the fact that our group $G$ is 
unimodular.
\end{proof}
\begin{definition}
Let $M$ be a proper $G$-manifold with boundary, and $c\in Z^k_{\rm diff}(G)$ be a smooth group cocycle.  For $A_0,\ldots, A_k\in 
{}^b \mathcal{A}_{G}^c (M)$, define
\[
\tau^r_\varphi (A_0,\ldots,A_k):=\int_{G^{\times k}}{}^b{\rm Tr}_S
\left(\Phi_{A_0}((g_1\cdots g_k)^{-1})\circ \Phi_{A_1}(g_1)\circ \ldots\circ \Phi_{A_k}(g_k)\right)
\varphi (e,g_1,g_1g_2,\ldots,g_1\cdots g_k)dg_1\cdots dg_k
\]
\end{definition}

\noindent
where $r$ stands for {\it regularized}. (The subscript $b$ is already used for the differential in cyclic cohomology.)
Now, let $Y$ be a closed manifold equipped with a proper, cocompact action of $G$. We denote by ${\rm cyl}(Y)=Y\times\RR$
the cylinder over $Y$, equipped with the action of $G\times\RR$. We can compactify  ${\rm cyl}(Y)$ to 
a $b$-cylinder as in \cite{Melrose} (where it is denoted $\widetilde{Y}$) and turn the $\RR$ action into a $\RR^+$-action.
Viceversa, we can look at $\RR^+$-invariant kernels on the $b$-cylinder $\widetilde{Y}$
as translation invariant kernels on ${\rm cyl}(Y)=Y\times\RR$.

\medskip
\noindent
{\bf With a small abuse of notation, and with the above remarks in mind,
  we denote by ${}^b \mathcal{A}_{G,\RR}^c  (\cyl (Y))$ the $\RR^+$-invariant
$b$-calculus with bounds on the $b$-cylinder associated to $Y$.}

\medskip
\noindent
Using the Mellin  transforms we obtain an
isomorphism 
\begin{equation}\label{identification}
{}^b \mathcal{A}_{G,\RR}^c  (\cyl (Y))\cong \mathscr{S}(\RR,\mathcal{A}^c_{G}(Y)),\quad A\mapsto \hat{A}.
\end{equation}
%We write $\cyl (Y)$ for $[-1,1]\times Y$.
%\textcolor{cyan}{This needs to be clarified a bit in that on the left we should take $[-1,1]\times Y$
%and also this is not really an isomorphism unless you take the small calculus...}
\begin{definition}\label{sigma-cocycle}
Let $Y$ be a closed manifold equipped with a proper, cocompact action of $G$, and let $\varphi\in C^k_{\rm diff}(G)$ be a smooth group 
cocchain.  The {\em eta cochain} on ${}^b \mathcal{A}_{G,\RR}^c  ({\rm cyl}(Y))$ associated to $c$ is defined as
\begin{align*}
\sigma_\varphi (B_0,\ldots,B_{k+1})
&:=\frac{(-1)^{k+1}}{2\pi}\int_{G^{k+1}}\int_\RR{\rm Tr}\left(\hat{B}_0((g_1\cdots g_{k+1})^{-1},\lambda)\circ
\hat{B}_1(g_1,\lambda)\circ\cdots \circ\hat{B}_k(g_k,\lambda)\circ\frac{\partial\hat{B}_{k+1}(g_{k+1},\lambda)}{\partial \lambda}\right)
\\&\hspace{7cm}\varphi (e,g_1,g_1g_2,\ldots,g_1\cdots g_{k})
d\lambda dg_1\cdots dg_{k+1}\\
\end{align*}
\end{definition}
\begin{proposition}
Let $M$ be a proper $G$-manifold with boundary, and $c\in C^k_{{\rm diff},\lambda}(G)$ a cyclic group cochain.
We have the identities:
\[
\left(\begin{matrix} (b+B)&-I^*\\0&-(b+B)\end{matrix}\right)\left(\begin{matrix}\tau^r_\varphi\\ \sigma_\varphi\end{matrix}\right)=
\left(\begin{matrix}\tau^r_{\delta\varphi}\\ \sigma_{\delta\varphi}\end{matrix}\right)
\]
\end{proposition}
\begin{proof}
This is a straightforward computation:
\begin{align*}
b\tau^r_\varphi(A_0,\ldots,A_{k+1})&=\int_{G^{\times k}}{}^b{\rm Tr}_S
\left(\Phi_{A_0}((g_1\cdots g_k)^{-1}h^{-1})\circ \Phi_{A_1}(h)\circ \Phi_{A_2}(g_1)\circ \ldots\circ \Phi_{A_{k+1}}(g_k)\right)\\
&\hspace{7cm}\varphi(e,g_1,g_1g_2,\ldots,g_1\cdots g_k)dg_1\cdots dg_kdh\\
&+
\sum_{i=1}^k(-1)^i\int_{G^{\times {k+1}}}{\rm Tr}_S^b\left((\Phi_{A_0}((g_1\cdots g_k)^{-1})\circ\ldots\circ\Phi_i(g_ih^{-1})\Phi_{i+1}(h)\circ\ldots\circ\Phi_{k+1}(g_k)\right)\\
&\hspace{7cm}\varphi(e,g_1,g_1g_2,\ldots,g_1\cdots g_k)dg_1\cdots dg_kdh\\
&+\int_{G^{\times k}}{}^b{\rm Tr}_S
\left(\Phi_{A_{k+1}}((g_1\cdots g_k)^{-1}h^{-1})\Phi_{A_0}(h)\circ \Phi_{A_2}(g_1)\circ \ldots\circ \Phi_{A_{k+1}}(g_k)\right)\\
&\hspace{7cm}\varphi(e,g_1,g_1g_2,\ldots,g_1\cdots g_k)dg_1\cdots dg_kdh\\
&=\tau^r_{\delta \varphi}(A_0,\ldots,A_{k+1})\\&
+(-1)^{k+1}\int_G^{\times(k+1)}{}^b{\rm Tr}_S\left([\Phi_{A_{k+1}}(g_{k+1}),\Phi_{A_0}((g_1\cdots g_{k+1})^{-1})\circ\Phi_{A_1}(g_1)\circ\ldots\circ\Phi_{A_k}(g_k)]\right)\\&\hspace{7cm}\varphi(e,g_1,\ldots,g_1\cdots g_k)dg_1\cdots dg_{k+1}\\
&=\tau^r_{\delta \varphi}(A_0,\ldots,A_{k+1})\\&+\frac{i(-1)^{k+1}}{2\pi}\int_{G^{\times(k+1)}}\int_\RR{\rm Tr}_S
(I(\Phi_{A_0},(g_1\cdots g_{k+1})^{-1},\lambda)\circ I(\Phi_{A_1},g_1,\lambda)\circ I(\Phi_{A_k},g_k,\lambda)\\&\hspace{3cm} \frac{\partial I(\Phi_{A_{k+1}},g_{k+1},\lambda)}{\partial \lambda}\varphi(e,g_1,\ldots,g_1\cdots g_k)dg_1\cdots dg_{k+1}\\
&= \tau^r_{\delta \varphi}(A_0,\ldots,A_{k+1})+I^*\sigma_\varphi (A_0,\ldots,A_{k+1})
\end{align*}
The fact that $B\tau^r_\varphi=0$ follows just as in the closed manifold case. Finally, injectivity of $I^*$ shows the last identity$(b+B)\sigma_\varphi=0$. 
\end{proof}
\begin{corollary} 
If $\varphi\in Z^{k}_{{\rm diff}} (G)$ is a smooth group {\bf cocycle} then $(\tau^{r}_\varphi,\sigma_\varphi)$
is a relative cyclic cocycle for $ {}^b \mathcal{A}_{G}^c (M)\xrightarrow{I} {}^b \mathcal{A}_{G,\RR^+}^c  (\overline{N^+ \pa M})$.
\end{corollary}

Together with this relative cyclic cocycle we shall also consider the cyclic cocycle $\tau_\varphi$
on the residual algebra
$ \mathcal{A}^c_G (M)$ appearing in the short exact sequence 
$
0\to  \mathcal{A}^c_G (M)\to {}^b \mathcal{A}_{G}^c (M)
\xrightarrow{I} {}^b \mathcal{A}_{G,\RR^+}^c  (\overline{N^+ \pa M})\to 0\,.$
This is exactly as in the closed case, given that elements in $ \mathcal{A}^c_G (M)$ are really kernels 
on $M\times M$ vanishing on $\partial (M\times M)$ (indeed, by definition, their Scwhartz kernel vanishes on the front face
and also vanishes on the left and right boundary).

\noindent
Thus we can give  the following 
\begin{definition}\label{def-unregularized}
Given $\varphi\in Z^{k}_{{\rm diff}} (G)$ we can proceed as in the closed case, see Subsection \ref{section:van-est},
and define a cyclic cocycle $\tau_{\varphi}$ on the algebra $ \mathcal{A}^c_G (M)$.
\end{definition}

\medskip
\noindent
We have worked in the small calculus, but exactly the same steps establish all of the above 
results for the algebras appearing in the calculus with $\delta$-bounds and with $G$-compact support.

\begin{example}
In the spirit of Example \ref{basic-example-1}, let us consider $G=SL(2,\RR)$ acting on itself by left translation. As explained in Part I of this paper, besides the trivial group cocycle, there is an interesting 
degree $2$ cocycle given by the area of a hyperbolic triangle in $\mathbb{H}^2$:
\[
\varphi(g_0,g_1,g_2):={\rm Area}(\Delta(g_0\cdot i,g_1\cdot i,g_2\cdot i)),
\]
where $g_i\in SL(2,\RR)$ acts as usual by M\"obius transformations. The corresponding eta 3-cocycle on $\mathscr{S}(\RR,C^\infty_c(G))$ is given by
\begin{align*}
\sigma_\varphi(B_0,B_1,B_2,B_3):=-\frac{1}{2\pi}\int_{G^3}\int_\RR\hat{B_0}((g_1g_2g_3)^{-1},\lambda)*\hat{B}_1(g_1,\lambda)&*\hat{B}_2(g_2,\lambda)*\frac{\partial \hat{B}(g_3,\lambda)}{\partial\lambda}\\&{\rm Area}(\Delta(i,g_1\cdot i,g_1g_2\cdot i))d\lambda dg_1dg_2dg_3.
\end{align*}
\end{example}

 \section{The higher Atiyah-Patodi-Singer index theorem on G-proper manifolds}\label{sect:aps}
 
 \subsection{Extension of cocycles}
 In this subsection we shall prove that under the same assumptions as in Theorem 
\ref{c*-index-th-part-1}, namely that 
$G$ has finitely many connected components and satisfies the RD condition and that 
$\varphi\in C^{2k}_{{\rm dif},\lambda}(G)$ is a cocycle of polynomial growth, the relative cyclic cocycle
$(\tau^r_\varphi,\sigma_\varphi)$ and the cyclic cocycle $\tau_\varphi$
extend from $$({}^b \mathcal{A}_{G}^{c,\delta} (M,E),{}^b \mathcal{A}^{c,\delta} _{G,\RR^+}  (\overline{N^+ \pa M},E))\;\;\;\text{ and  }
\;\;\;\mathcal{A}_G^{c,\delta}  (M,E)$$ to $$({}^b \mathcal{A}_{G}^{\infty,\delta} (M,E),{}^b \mathcal{A}^{\infty,\delta}_{G,\RR^+}  (\overline{N^+ \pa M},E)) \;\;\;\text{ and  }
\;\;\; \mathcal{A}_G^{\infty,\delta} (M,E)$$ respectively.\\

\medskip
\noindent
{\bf As in previous sections, we expunge the vector bundle $E$ 
from the notation and we omit the $\delta$.}

\medskip
\noindent
As in the closed case, we shall make {\it crucial} use of the fact, proved in Part 1, that under these assumptions
the cyclic cocycle $\tau_\varphi^G$ extends from $C^\infty_c (G)$ to $H^\infty_L(G)$. As we shall see, the proof proceeds
similarly to the closed case, but with some technical complications having to do with the fact that we are considering
Melrose' regularized trace in the definition of $\tau^r_\varphi$.\\

Building on  \cite{GMPi} we shall make use of  a  fundamental observation due to
Lesch-Moscovici-Pflaum, see \cite{LMP2}.  Before stating it, we introduce the relevant notation. Let 
$S$ be a $b$-manifold, built out of a compact manifold with boundary $S_0$. (Eventually
$S$ will be our slice.)
Let $\phi\in C^\infty (S)$ be a function equal to $t$ on the cylindrical end $(-\infty, 0]\times \pa S_0\subset S$. Let $\mathcal{V}$ be a vector field   equal to
 $\pa/\pa t$ on the cylindrical end.
  In particular   $ \mathcal{V} (\phi)=1$ on the cylindrical end.  Let $\chi:= 1-\mathcal{V} (\phi) \in C^\infty_c (S_0\setminus \pa S_0)$.

 \begin{proposition} (Lesch-Moscovici-Pflaum)
 If $P\in  {}^b\Psi^{-\infty,\epsilon} (S) + \Psi^{-\infty,\epsilon} (S) $ then
 \begin{equation}\label{LMP}
  {}^b\Tr (P)= -\Tr (\phi [\mathcal{V},P]) + \Tr (\chi P)\,.
  \end{equation}
 \end{proposition}
 
 \noindent
 Consequently, the $b$-Trace of $P$ is the difference of the traces of two
 trace-class operators
 naturally associated to $P$.  

 \begin{definition}\label{def:triple}
 If $P\in {}^b\Psi^{-\infty,\epsilon} (S) + \Psi^{-\infty,\epsilon} (S) $
 then we define a norm 
 \begin{equation}\label{triple}
 || P ||_b^2 := \| \chi P \|^2_1 + \| \phi [\mathcal{V},P] \|^2_1 + \| [\mathcal{V},P] \|^2_1
 + \| [\phi ,P] \|^2 + \| P \|^2
 \end{equation}
 with the last two norms denoting the $L^2$-operator norm. 
% \textcolor{cyan}{It might be easier 
% to use the HS norm instead of the trace norm. I believe that the following proposition would be still true
% provided $k>0$. The choice of the HS norm would make the proof of Lemma \ref{key-lemma-b}
% easier.}
 \end{definition}

\noindent
The following Proposition is established in \cite{GMPi}:
\begin{proposition}\label{fromGMP}
 If $P_j\in \Psi^{-\infty,\epsilon}_b (S) + \Psi^{-\infty,\epsilon} (S)$, $j\in \{0,1,\dots,k\}$, then there exists $C>0$ such that
 \begin{equation}\label{inequality-lemma1}
 |   {}^b\Tr  (P_0 P_1 \cdots P_k)| \leq  C || P_0 ||_b \cdots || P_k ||_b
 \end{equation}
\end{proposition}

\noindent
In the sequel, we shall need the following analogues of Lemma \ref{key-lemma-closed}:

\begin{lemma}\label{key-lemma-residual}
The map
$\Phi\to ||\Phi(\cdot)||_1$ defines a continuous application 
$\mathcal{A}_G^\infty (M)\to H^\infty_L (G)$. 
\end{lemma}

\begin{lemma}\label{key-lemma-b}
The map
$\Phi\to ||\Phi(\cdot)||_b$ defines a continuous application 
${}^b \mathcal{A}_{G}^\infty (M)\to H^\infty_L (G)$. 
\end{lemma}

%\begin{lemma}\label{key-lemma-invariant}
%For $\j>0$ the map
%$\Phi\to (\int ||\Phi(\cdot)||$ defines a continuous application 
%${}^b \mathcal{A}_{G}^\infty (M)\to H^\infty(G)$. 
%\end{lemma}

\begin{proof}
We begin with Lemma \ref{key-lemma-residual}.
It suffices to show that 
$$\rho_{{\rm ff}}{}^b \Psi^{-\infty,\epsilon} (S)+ \Psi^{-\infty,\epsilon} (S)\xrightarrow{||\cdot||_1} \RR$$
is continuous with respect to the seminorms defining the Fr\'echet topology
of the space on the left hand side. %$ \Psi^{-\infty,\epsilon} (S)$.
  This is proved very much as in the closed case. Write $A_\kappa$ for the operator associated to $\kappa
 \in \rho_{{\rm ff}}{}^b \Psi^{-\infty,\epsilon} (S)+ \Psi^{-\infty,\epsilon} (S)$. Then 
 $$A_\kappa= ((1+\Delta)^{-M} \rho^{\epsilon/2}) (\rho^{-\epsilon/2}(1+\Delta)^{M} A_\kappa)$$
 with $\rho$ a defining function for the boundary.
 The first factor on the right hand side is trace class whereas the second factor is still vanishing on all the boundary hypersurfaces,
 so that, similarly to the closed case,
 $$||A_\kappa||_1 \leq \| (1+\Delta)^{-M} \rho^{\epsilon/2})\|_1 \| (\rho^{-\epsilon/2}(1+\Delta)^{M} A_\kappa)\|_2
\leq C \| (\rho^{-\epsilon/2}(1+\Delta)^{M} A_\kappa)\|_2 $$
The last term is given by the $L^2$ norm of the kernel of the operator over $S\times S$ which is easily bounded
by one of the seminorms of $\rho_{{\rm ff}}{}^b \Psi^{-\infty,\epsilon} (S)+ \Psi^{-\infty,\epsilon} (S)$.

\bigskip
Next we discuss the proof of Lemma \ref{key-lemma-b}.
It suffices to show that
$${}^b \Psi^{-\infty,\epsilon} (S)\xrightarrow{||\cdot||_b} \RR$$
is continuous with respect to the seminorms defining the Fr\'echet topology
of ${}^b \Psi^{-\infty,\epsilon} (S)$, where we recall that 
$|| P ||_b^2 := \| \chi P \|^2_1 + \| \phi [\mathcal{V},P] \|^2_1 + \| [\mathcal{V},P] \|^2_1
 + \| [\phi ,P] \|^2 + \| P \|^2$\\
This follows from the analysis given in \cite{GMPi}, see  {\it End of the proof of Proposition 6.1}, at page 292. We only give a
sketch of the main ideas. 
 For the continuity of the summands $\|P\|^2$ and $ \| [\phi ,P] \|^2$ we use the trick  
explained at page 295 of  \cite{GMPi}, writing $P$ in terms of its indicial operator
and a rest,  which is of course residual, i.e. an element in 
$\Psi^{-\infty,\epsilon}(S)$. Put it differently we take the Taylor series of order 1 at the front face for the kernel of $P$. 
These two terms in the Taylor series are treated separetly. For the first term we recall that 
in \cite{GMPi} we have used freely the observation that 
the bounds appearing in the Fr\'echet seminorms for an element in the $\RR^+$-invariant
calculus with $\epsilon$-bounds on $\overline{N_+ \partial S}$ translate into 
weighted exponential bounds of a translation invariant kernel when we pass 
from $\overline{N_+ \partial S}$ to $\cyl (S)$. For the indicial operator one can estimate its operator $L^2$-norm
in terms  of its kernel and an exponential weight (see the end of Page 295, taking $g=1$ there). The last but summand appearing in the definition of $\|\,\,\|_b$
is treated as in page 296 in  \cite{GMPi}, using the fact that $\phi=t$ along the cylindrical end of the manifold 
with cylindrical ends associated to $S$ or, equivalently, $\phi$ is $\log\rho$, $\rho$ always a boundary
defining function,   
in the compact picture for $S$. The continuity with respect to the trace class norms for the first 3 summands
appearing in the definition of $\|\,\,\|_b$  is treated very much 
as in the proof of  Lemma \ref{key-lemma-residual}. We omit the details. 
%\textcolor{cyan}{Proof to be given.}
\end{proof}

\noindent
Using Lemma \ref{key-lemma-residual} and proceeding exactly as in the closed case we can prove the following:
\begin{proposition}\label{tau-extends}
Let $G$ have a finite number of connected components. Assume property RD for $G$ and let
$\varphi\in Z^{k}_{{\rm diff},\lambda} (G)$ be a cocycle of polynomial growth. Then the
 cyclic cocycle $\tau_\varphi$ extends from $ \mathcal{A}^c_G(M)$  to $ \mathcal{A}^\infty_G(M)$.
 \end{proposition}

\noindent
We now  consider  the extension problem for the regularized cochain $\tau^r_\varphi$.
Given $\Phi\in  {}^b\mathcal{A}^{c}_G(M)$, where we do not write the $\epsilon $ in the notation,
let us  introduce the following norm:
\begin{equation}\label{norm-3}
||| \Phi |||^2_m := \int_G ||\Phi (g)||^2_{b} (1+L(g))^{2m} dg \,.
\end{equation}
The expression $\int_G ||\Phi (g)||^2_{b} (1+L(g))^{2m} dg $ can also be written  for $\Phi\in  {}^b \mathcal{A}^\infty_G (M)$
%and the following lemma holds:
%If $\Phi\in \mathcal{A}^\infty_G(M)$ 
%then %there exists 
%for each $m\in\NN$ we have that $|||\Phi |||_m <\infty$.
%Consequently  %there exists 
%given $m\in\NN$ %such that 
%we have that 
%$$\mathcal{A}^\infty_G(M) \subset \overline{\mathcal{A}^c_G(M)}^{|||\cdot|||_m}\,.$$
and we know from the above  Lemma that
if $\Phi\in {}^b \mathcal{A}^\infty_G(M)$ 
then %there exists 
for each $m\in\NN$ we have that $|||\Phi |||_m <\infty$.
Consequently  %there exists 
given $m\in\NN$ %such that 
we have that 
\begin{equation}\label{inclusion-b}
{}^b \mathcal{A}^\infty_G(M) \subset \overline{{}^b \mathcal{A}^c_G(M)}^{|||\cdot|||_m}\,.
\end{equation}
 Let now $\Phi_0,\Phi_1,\dots,\Phi_k\in {}^b  \mathcal{A}^c_G(M)$. We want to estimate $|\tau^r_\varphi (\Phi_0,\dots,\Phi_k)|$. Proceeding as in the closed case, using crucially Proposition \ref{fromGMP} ,
%, that is, the absolute value of 
%$$\int_{G^{k}} \int_{S^{k+1}} \Phi_0 (g_1)(s_0,s_1)\cdots \Phi_{k-1} (g_{k})(s_{k-1},s_k)
%\Phi_k (g_1 \cdots g_k)^{-1}) (s_k,s_0)  \,c(e,g_1,g_1 g_2,\dots,g_1 g_2 \dots g_{k}) ds_0 \cdots ds_k dg_1\cdots dg_k
%$$
%Bringing the absolute value under the sign of integral, using Lidski's theorem,
%$\Tr K=\int_S K(s,s)ds$ and the well-known estimate 
%$| \Tr (K) |\leq ||K ||_1$ we have
we get 
\begin{equation*}
 |\tau_\varphi^r  (\Phi_0,\dots,\Phi_k)| \leq  \tau_{|\varphi |}^G ( || \Phi_0 (\cdot)||_b,  \cdots  ,||\Phi_{k} (\cdot)||_b)
\end{equation*}
Using Proposition 5.6 in Part 1 we have 

$$\tau_{|\varphi |}^G ( || \Phi_0 (\cdot )||_b,  \cdots  ,||\Phi_{k} (\cdot)||_b)\leq C \nu_{p+\ell} ( || \Phi_0 (\cdot)||_b)
\cdots \nu_{p+\ell} ( || \Phi_{k} (\cdot)||_b)$$
for suitable $p$ and $\ell$.
Now,
$\nu_{p+\ell} ( || \Phi (\cdot)||_b)= ||| \Phi |||_{p+\ell}$.
%with $||\cdot ||_1$ denoting the trace norm. 
Consequently  $$ |\tau_\varphi^r (\Phi_0,\dots,\Phi_k)| \leq C ||| \Phi_0 |||_{p+\ell} \cdots  ||| \Phi_k |||_{p+\ell}$$
and since 
$${}^b \mathcal{A}^\infty_G(M) \subset \overline{{}^b \mathcal{A}^c_G(M)}^{|||\cdot|||_{p+\ell}}$$
we conclude, finally, that $\tau_\varphi^r$ extends continuously from $ {}^b\mathcal{A}^{c}_G(M)$ to ${}^b \mathcal{A}^\infty_G(M)$.\\
%\textcolor{cyan}{Check if there are norm to be squared...}

\noindent
We collect this result in the following 

\begin{proposition}\label{tau-r-extends}
Let $G$ have a finite number of connected components. Assume property RD for $G$ and let
$\varphi\in C^{k}_{{\rm diff},\lambda} (G)$ be a cocycle of polynomial growth. Then the regularized
 cochain $\tau^r_\varphi$ extends to ${}^b \mathcal{A}^\infty_G(M)$.
 \end{proposition}

\noindent
Finally, let now $Y$ be any $G$-proper manifold {\it without} boundary.
We shall deal with the extension of the eta cocycle $\sigma_\varphi$, 
from  ${}^b \mathcal{A}_{G,\RR}^c  ({\rm cyl}(Y))$ to  ${}^b \mathcal{A}_{G,\RR}^\infty  ({\rm cyl}(Y))$.

\noindent
Let us recall the definition of $\sigma_\varphi$ on ${}^b \mathcal{A}_{G,\RR}^c  ({\rm cyl}(Y))$:
\begin{align*}
\sigma_\varphi(B_0,\ldots,B_{k+1})
&:=\frac{(-1)^{k+1}}{2\pi}\int_{G^{k+1}}\int_\RR{\rm Tr}\left(\hat{B}_0((g_1\cdots g_{k+1})^{-1},\lambda)\circ
\hat{B}_1(g_1,\lambda)\circ\cdots \circ\hat{B}_k(g_k,\lambda)\circ\frac{\partial\hat{B}_{k+1}(g_{k+1},\lambda)}{\partial \lambda}\right)
\\&\hspace{7cm} \varphi(e,g_1,g_1g_2,\ldots,g_1\cdots g_{k})
d\lambda dg_1\cdots dg_{k+1}\\
\end{align*}
%\textcolor{cyan}{So, there is no $g_{k+1}$ appearing in the group
%cocycle $c$ , right ?? I would like to point out that this
%is NOT the expression we had at board 7 page 24 of the pdf file of the photos of December 2019 in
%Amsterdam }\\
We thus have:

\begin{align*}
| \sigma_\varphi (B_0,\ldots,B_{k+1})|
&\leq \frac{1}{2\pi}\int_{G^{k+1}}\int_\RR \| \left(\hat{B}_0((g_1\cdots g_{k+1})^{-1},\lambda)\circ
\hat{B}_1(g_1,\lambda)\circ\cdots \circ\hat{B}_k(g_k,\lambda)\circ\frac{\partial\hat{B}_{k+1}(g_{k+1},\lambda)}{\partial \lambda}\right)\|_1
\\&\hspace{7cm} \varphi (e,g_1,g_1g_2,\ldots,g_1\cdots g_{k})
d\lambda dg_1\cdots dg_{k+1}\\
\end{align*}
Let 
$$f_0 (h,\lambda):= \|\hat{B}_0(h,\lambda)\|_1\, \;\;
f_j (h,\lambda):= \|\hat{B}_j(h,\lambda)\|_1 (1+L(h))^\ell\, \;\;j\in \{1,\dots,k\}\,,\;\;
f_{k+1}(h,\lambda):= \| \frac{\partial\hat{B}_{k+1}(h,\lambda)}{\partial \lambda}\|_1$$
Interchanging the two integrals, which is possible given the $G$-compact support and the rapid decay in $\lambda$,
and using the polynomial bounds on the group cocycle $\varphi$, we can employ  the Jolissant-Connes-Moscovici fundamental
 observation
and, as in the proof of  Proposition 5.6 in Part 1, find constant $C$, $D$ and $E$ and a positive integer
$p$ such that 
\begin{align*}
| \sigma_\varphi(B_0,\ldots,B_{k+1}) |\leq &  \frac{C}{2\pi}\int_\RR  
|(f_0 (\cdot,\lambda) * f_1 (\cdot,\lambda) * \cdots f_k (\cdot,\lambda) * f_{k+1} (\cdot,\lambda))(e) |d\lambda \\
\leq & \frac{C}{2\pi}\int_\RR  \| (f_0 (\cdot,\lambda) \|_{C^*_r G}  \| (f_1 (\cdot,\lambda) \|_{C^*_r G}\cdots
 \| (f_k (\cdot,\lambda) \|_{C^*_r G}  \| (f_{k+1} (\cdot,\lambda) \|_{C^*_r G}\\
\leq &  \frac{D}{2\pi}\int_\RR  \nu_p ( f_0 (\cdot,\lambda)) \nu_p ( f_1 (\cdot,\lambda))\cdots \nu_p ( f_k (\cdot,\lambda))
\nu_p ( f_{k+1} (\cdot,\lambda))d\lambda\\
\leq &  \frac{E}{2\pi}\int_\RR  \nu_{p+\ell} ( \|\hat{B}_0(\cdot,\lambda)\|_1) \nu_{p+\ell} ( \|\hat{B}_1(\cdot,\lambda)\|_1)
\cdots \nu_{p+\ell} ( \|\hat{B}_k(\cdot,\lambda)\|_1)
\nu_{p+\ell} (\| \partial_\lambda \hat{B}_{k+1}(\cdot,\lambda)\|_1)d\lambda
\end{align*} 
We now  use the generalized H\"older inequality 
%\footnote{see {\tt https://link.springer.com/content/pdf/10.1186/s13660-019-2048-0.pdf}} 
and get
\begin{align*}
& \int_\RR  \nu_{p+\ell} ( \|\hat{B}_0(\cdot,\lambda)\|_1) \nu_{p+\ell} ( \|\hat{B}_1(\cdot,\lambda)\|_1)
\cdots \nu_{p+\ell} ( \|\hat{B}_k(\cdot,\lambda)\|_1)
\nu_{p+\ell} (\| \partial_\lambda \hat{B}_{k+1}(\cdot,\lambda)\|_1)d\lambda \leq  \\
& (\int_\RR (\nu_{p+\ell} ( \|\hat{B}_0(\cdot,\lambda)\|_1))^{k+2}d\lambda )^{\frac{1}{k+2}} \cdots 
(\int_\RR (\nu_{p+\ell} ( \|\hat{B}_{k+1}(\cdot,\lambda)\|_1))^{k+2} d\lambda )^{\frac{1}{k+2}}
\end{align*}
Now, given $B\in {}^b \mathcal{A}_{G,\RR}^\infty  ({\rm cyl}(Y))$, $m\in \NN^+$ and $j\in \NN^+$ we can prove
as in the previous cases, see \eqref{inclusion-closed}, \eqref{inclusion-b},
that
$$||| B |||_{j,m} : = \left( \int_\RR \left( \int_G \| \hat{B} (g,\lambda)\|_1 (1+ L (g))^{2m} dg \right)^j d\lambda \right)^{\frac{1}{j}}<+\infty$$
This means that
$${}^b \mathcal{A}_{G,\RR}^\infty  ({\rm cyl}(Y))\subset\overline{ {}^b \mathcal{A}_{G,\RR}^c  ({\rm cyl}(Y))}^{|||\cdot|||_{j,m}}$$
As we have proved that there is a constant $F$ and positive integers $\ell$ and $p$ such that
$$| \sigma_\varphi (B_0,\ldots,B_{k+1})|\leq F ||| B_0 |||_{k+2,p+\ell} \cdots ||| B_{k+2} |||_{k+2,p+\ell}$$
we finally conclude that $\sigma_\varphi $ extends %continuously from ${}^b \mathcal{A}_{G,\RR}^c  ({\rm cyl}(Y))$
to ${}^b \mathcal{A}_{G,\RR}^\infty  ({\rm cyl}(Y))$.% as required.

\noindent
Summarizing:

\begin{proposition}\label{sigma-extends}
Let $G$ have a finite number of connected components. Assume property RD for $G$ and let
$\varphi\in C^{k}_{{\rm diff},\lambda} (G)$ be a cocycle of polynomial growth. Let $Y$ be a cocompact
$G$-proper manifold without boundary.
Then the eta 
 cocycle $\sigma_\varphi$ extends continuosly
 from ${}^b \mathcal{A}_{G,\RR}^c  ({\rm cyl}(Y))$
to ${}^b \mathcal{A}_{G,\RR}^\infty  ({\rm cyl}(Y))$.
 \end{proposition}

\noindent
Going back to the case of manifolds with boundary, we notice  that, by continuity, the extended pair $(\tau^r_\varphi,\sigma_\varphi)$ is
still a relative cocycle for ${}^b \mathcal{A}_{G}^\infty  (M)\xrightarrow{I}
 {}^b \mathcal{A}_{G,\RR}^\infty  (\overline{N_+ \pa M})$.

\subsection{The higher APS index formula}\label{PP-aps}
Having proved the extension property for $\tau_\varphi$ and $(\tau^r_\varphi,\sigma_\varphi)$ we can finally tackle the
higher APS index formula for the $C^*$-indices associated to $D$. We recall our geometric data.
We have a $G$-proper manifold with boundary $M_0$ with associated manifold with cylindrical
ends $M$. All our structures are of product type near the boundary.
We assume $M_0$ to be {\em even dimensional}. We have  a $\ZZ_2$-graded odd Dirac operator 
$D$ with boundary operator $D_\partial$ 
satisfying Assumption \ref{assumption:invertibility}:
$$\text{ there exists a } \;\;\alpha>0\;\; \text{ such that }\;\;
{\rm spec}_{L^2} (D_{\partial})\cap [-\alpha,\alpha]=\emptyset\,.$$
We shall assume that $D$ is defined by a bundle of unitary Clifford modules $E$ endowed with 
a hermitian Clifford connection $\nabla^E$.
 We assume 
that $G$ has finitely many connected components and that satisfies the RD condition. We  choose as an admissible Fr\'echet subalgebra $H^\infty_L (G)\subset
C^*_r G$. We fix a $\delta<\alpha$ so as in our treatment of the Connes-Moscovici projector $V(D)$ and its improved version
$V^b (D)$.

\medskip
\noindent
{\bf We shall omit $\delta$ from the notation.}

\medskip
\noindent
We have proved that the  improved Connes-Moscovici projector $V^b (D)$ defines a smooth $C^*$-index class 
$\Ind_\infty (D)\equiv [ V^b (D)]-[e_1]\in K_0 (\mathcal{A}^\infty_G (M))$
and that the triple $(V(D),e_1,V(tD_{{\rm cyl}}))$ defines a smooth $C^*$-relative index class
$\Ind_\infty (D,D_\partial)\in K_0 ( {}^b \mathcal{A}_{G}^\infty  (M), {}^b \mathcal{A}_{G,\RR}^\infty  (\overline{N_+ \pa M}))$.
These two classes corresponds under the excision isomorphism.

\noindent
We now fix a group cocycle $\varphi$ for $G$ of polynomial growth.
Thanks to result of the previous subsection we can consider
 $\tau_\varphi$, the associated extended cyclic cocycle for the algebra $\mathcal{A}^\infty_G (M)$. 
 
 \begin{definition}\label{indeces-b}
 The higher APS $C^*$-index associated to $\varphi$ and $D$ satisfying the 
 previous assumptions are given by the
 numbers
 $$\Ind_{{\rm APS},\varphi} (D):= (-1)^p \frac{2p !}{p!}\,\left<\Ind_\infty (D),[\tau_\varphi]\right>\,,\quad [\varphi]\in H^{2p}_{{\rm diff}} (G)\,.$$
\end{definition}
The main goal of this subsection is to provide a APS index formula for these higher indices. At this point 
of the paper our
arguments follow rather closely those given  in \cite{mp}, \cite{GMPi}; for this reason  we shall be rather brief.

First we define the higher eta invariant that will enter into the higher APS index formula.

\begin{definition}\label{def:higher-eta}
Let $G$ be as above.
Let 
 $\varphi\in Z^k_{{\rm diff}} (G)$, $k=2p$ be  a  group cocycle of polynomial growth.
 We
 define the higher eta invariant associated to $\varphi$ and the boundary operator
 $D_{\partial}$ as
 \begin{equation}
 \eta_{\varphi} (D_{\partial}):=
2 c_p\,
\left[ \sum_{i=0}^{2p}   \int_0^{\infty}\sigma_\varphi (p_t,\dots,[\dot{p}_t, p_t],\dots,p_t)dt
\right]
%$$
 %\frac{2(2p+1)}{p!} \int_0^{+\infty}\sigma_{c} ([\dot{p_t},p_t],p_t,\dots,p_t,p_t)dt
 \end{equation}
 with $p_t=V (t D_{{\rm cyl}})$   and $c_p=(-1)^p \frac{2p !}{p!}$. (We shall see momentarily that this integral is indeed convergent.)
%\footnote{\textcolor{cyan}{this is assuming we don't have 
% to average, otherwise we should take $p_t:=
%  V_{t D_{{\rm cyl}}}\oplus V_{t D_{{\rm cyl}}}^*$.}}
  \end{definition}
  
%  \textcolor{cyan}{It would be desirable to simplify this expression and maybe perform the
%  integration in $\lambda$, leaving an expression that depends solely on $\partial M_0$.
%  For the numeric (lower) eta invariant, this is indeed possible. The point is 
%that the indicial family of $\exp (-D^2)$ is $\exp (-D^2_\partial) \exp (-\lambda^2)$.}

We are now in the position to state the main result of this Section:

\begin{theorem}\label{c*-aps}
Let $G$, $M_0$, $M$, $D$ and $\varphi\in Z^k_{{\rm diff}} (G)$, $k=2p$ as above.
Then the higher eta invariant $ \eta_{\varphi} (D_{\partial})$ is well defined and the following 
higher $C^*$-index formula holds:
\begin{equation}\label{higher-formula}
\Ind_{{\rm APS},\varphi} (D)= \,\int_{M_0} \chi {\rm AS} (M_0)\wedge\omega_\varphi -\frac{1}{2} \eta_{\varphi} (D_\partial)\,.
\end{equation}
where we have normalized the Atiyah-Singer
integrand by multiplying it by $1/(2\pi i)^p$ in degree $\dim M-p$.
\end{theorem}

\begin{proof}
We only sketch the main steps. First, using the fact that $\tau^r_\varphi |_{\mathcal{A}^\infty_G (M)} = \tau_\varphi$
and  that $\Ind_\infty (D)$ and $\Ind_\infty (D,D_\partial)$ correspond under the excision
isomorphism, one proves the following crucial equality

\begin{equation*}
  \langle \Ind_\infty (D), [\tau_{\varphi}] \rangle
%&\equiv  %{\rm const}_{2p}\,
%\langle  [V^b_{\D^\otimes}] -[e_1], [\tau_{c}] \rangle\\
%&= %{\rm const}_{2p} \,
%\langle \beta_{{\rm ex}} ([V^b_{\D^\otimes}] -[e_1]), [(\tau^r_{c},\sigma_{c})] \rangle \\
%&= %{\rm const}_{2p} \,
=\langle [V(D), e_1, V(t D_{{\rm cyl}})], [(\tau^r_{\varphi},\sigma_{\varphi})] \rangle
\end{equation*}
On the left hand side we have the number we want to compute, that is $\Ind_{{\rm APS},\varphi} (D)$
up to the constant $(-1)^p \frac{2p !}{p!}$.
Writing down explicitly the relative pairing between relative K-theory and relative cyclic cohomology
we arrive at the equality
$$   
\left[ \sum_{i=0}^{2p}   \int_1^{\infty}\sigma_\varphi (p_t,\dots,[\dot{p}_t, p_t],\dots,p_t)dt
\right]
=  -  \langle \Ind_\infty (D), [\tau_{\varphi}] \rangle + 
\tau^r_\varphi (V(D))$$
%$$   \frac{1}{2}
%\left[ \sum_{i=0}^{2p}   \int_1^{\infty}\sigma_\varphi (p_t,\dots,[\dot{p}_t, p_t],\dots,p_t)dt
%\right]
%=  -  \langle \Ind_\infty (D), [\tau_{\varphi}] \rangle + 
%\tau^r_\varphi (V(D))$$
with $p_t=V (t D_{{\rm cyl}})$ and 
with the convergence at infinity of the integral on the left hand side 
following from the
fact that the relative pairing is well defined or, alternatively, from the exponential decay
of the heat kernel of the $L^2$-invertible operator $D_{{\rm cyl}}$.
Replace now $D$ by $uD$ with $u>0$. Then after a simple change of variables we obtain:
$$ 
\left[ \sum_{i=0}^{2p}   \int_u^{\infty}\sigma_\varphi (p_t,\dots,[\dot{p}_t, p_t],\dots,p_t)dt
\right]
=  -  \langle \Ind_\infty (uD), [\tau_{\varphi}] \rangle + 
\tau^r_\varphi (V(uD))$$
which is of course equivalent to 
\begin{equation}\label{equality-key} c_p
\left[ \sum_{i=0}^{2p}   \int_u^{\infty}\sigma_\varphi (p_t,\dots,[\dot{p}_t, p_t],\dots,p_t)dt
\right]
=  -  \Ind_{{\rm APS},\varphi} (D) + 
 \,c_p \tau^r_\varphi (V(uD))\end{equation}
 with $c_p=(-1)^p \frac{2p !}{p!}$.
%given that $\langle \Ind_\infty (uD), [\tau_{\varphi}] \rangle $ is independent of $u$ and equal to
 %$\Ind_{{\rm APS},\varphi} (D)$ (up to the constant $(-1)^p \frac{2p !}{p!}$). 
 We now take the limit as $u\downarrow 0$. By a well established general
 principle we can still apply Getzler rescaling to the $b$-supertrace of the heat kernel, see
 \cite[Chapter 8]{Melrose}, and,
 more generally, to the computation of the limit as $u\downarrow 0$ of $\tau^r_\varphi (V(uD))$.
 Thus, by the analogue of Theorem \ref{theo:short-time} in the b-context we have
 $$\lim_{u\downarrow 0} \tau^r_\varphi (V(uD))= c_p^{-1}\int_M \chi {\rm AS}(M) \wedge \omega_\varphi,$$
 where we have normalized the usual Atiyah--Singer integrand by multiplying it by $1/(2\pi i)^p$ in degree $\dim M-p$.
With this, the right hand side is in turn equal to
$$c_p^{-1}\int_{M_0} \chi {\rm AS}(M_0) \wedge \omega_\varphi$$
given that all of our geometric structures are of product type near the boundary. Here we have normalized the Atiyah-Singer
integrand by multiplying it by $1/(2\pi i)^p$ in degree $\dim M-p$.\\
As $\Ind_{{\rm APS},\varphi} (D)$ is a number, we conclude that
$$c_p  \lim_{u\downarrow 0} \left[ \sum_{i=0}^{2p}   \int_u^{\infty}\sigma_\varphi (p_t,\dots,[\dot{p}_t, p_t],\dots,p_t)dt
\right]$$
exists and is equal to $-  \Ind_{{\rm APS},\varphi} (D)+  \int_M \chi {\rm AS}(M) \wedge \omega_\varphi$.
Thus $\eta_\varphi (D_\partial)$ is well defined and 
$$\frac{1}{2}\eta(D_\varphi)=-  \Ind_{{\rm APS},\varphi} (D)+  \int_M \chi {\rm AS}(M) \wedge \omega_\varphi$$
 as required.
\end{proof}

%
%\textcolor{cyan}{1) Can we write down more for, say, a 2-cocycle ? }
%
%\textcolor{cyan}{2) So far we have discussed only the even dimensional case...}
%
 
We point out that thanks to the results of Part 1 this theorem applies, for example, to $G$-proper
manifolds with $G$ a connected  semisimple Lie group. More generally, we can assume that $G$ has finitely
many connected components, satisfies
property RD and is such that $G/K$ has non-positive sectional curvature.

\begin{remark}
The convergence of the higher eta invariant is established here for a boundary operator,
much as in the original work of Atiyah-Patodi-Singer.
If $Y$ is any closed $G$-proper manifold, not necessarily a boundary,
 then
it should be possible to prove, using Getzler rescaling, that the limit
$$\lim_{u\downarrow 0}
\left[ \sum_{i=0}^{2p}   \int_u^{\infty}\sigma_\varphi (p_t,\dots,[\dot{p}_t, p_t],\dots,p_t)dt
\right]
$$
exists. This would allow to define the higher eta invariant  $\eta_{\varphi} (D_{Y})$
in general, even for non-bounding
$G$-proper manifolds.
%A positive answer to this question would have consequences on the treatment
%of the signature operator below.

\end{remark}

\subsection{Pairing with 0-cocycles}
When we apply the above theorem to the 0-cocycle $\tau$ defined by the trivial 0-cocycle on $G$ we 
certainly obtain
a APS index formula. Using the fact that $$I(\exp (-t D^2),\lambda)=\exp (-t D^2_{\partial}) \exp (-t\lambda^2)$$
 and performing the integral in $\lambda$ in the definition of the eta
1-cocycle,   this formula reads:
$$\left<\Ind_{C^*} (D),\tau\right> = \int_M \chi {\rm AS}(M) -\frac{1}{2} \eta_G (D_\partial)$$
with 
$$\eta_G (D_\partial)= \frac{2}{\sqrt{\pi}}\int_0^\infty \tr_G D_\partial e^{-(tD_\partial)^2}dt\,.$$
See \cite{Melrose} and also \cite{mp}, for more on this particular case.\\
Notice however that  this particular result holds under much more general assumptions than the ones
in Theorem \ref{c*-aps} as we shall now briefly explain. 

\smallskip
\noindent
The pairing of the index class with $\tau$ is equal to the pairing of the Morita equivalent 
$C^*_r(G)$-index class $\Ind_{C^*_r(G)}(D)$ with the canonical trace $\tau_e$ on $C^*G$:
$$ \left<\Ind_{C^*} (D),\tau\right> = \left<\Ind_{C^*G}(D),\tau_e\right>.$$ 
Proceeding as  in \cite{Wang},
one  checks that $\left<\Ind_{C^*G}(D),\tau_e\right> $ is equal to 
the von Neumann G-index of $D^+$. A formula for this von Neumann index can be proved
in the von Neumann framework
by mimicking the proof of Vaillant for Galois coverings of manifolds with cylindrical
ends \cite{Vaillant-master-en} (in turn inspired by Melrose'  proof on manifolds with cylindrical ends). 
Thus, assuming only that $G$ is a Lie group
\footnote{presumably one can assume even less}
but keeping the $L^2$-invertibility of the boundary operator $D_\partial$, we obtain that 
$\left<\Ind_{C^*} (D),\tau\right>$ and $\left<\Ind_{C^*G}(D),\tau_e\right>$ are well defined and that
$$\left<\Ind_{C^*} (D),\tau\right> = \left<\Ind_{C^*G}(D),\tau_e\right>=\Ind_{{\rm vN}} (D^+)= \int_M \chi {\rm AS}(M) -\frac{1}{2} \eta_G (D_\partial)$$
with 
$\eta_G (D_\partial)= \frac{2}{\sqrt{\pi}}\int_0^\infty \tr_G D_\partial e^{-(tD_\partial)^2}dt\,.$\\
This formula is the same as the one appearing in \cite{HWW1}: indeed, the von Neumann analytic index 
$\Ind_{{\rm vN}} (D^+)$ appearing on the left hand side of our formula 
and the  analytic APS-index appearing on the left hand side of the index formula in \cite{HWW1} are equal 
(proof as in the classical 
case); moreover, as already remarked, $ \int_M \chi {\rm AS}(M)= \int_{M_0} \chi {\rm AS}(M_0)$
because of the product type assumption of all our geometric structures. Thus the two formulas are completely
equivalent.

Notice that on Galois coverings of manifolds with cylindrical ends
Vaillant proves  a formula for the $L^2$-von-Neumann index without
the  assumption of invertibility on the boundary operator; this formula is crucial
in establishing a signature formula on Galois coverings of manifolds with boundary. 
We expect Vaillant's $L^2$-von Neumannn index formula   to hold also in the $G$-proper context.

For more on these von Neumann index theorems on manifolds with cylindrical ends, also in the context of measured
foliations,  the reader is referred to  \cite{Vaillant-master-en} \cite{lueck-schick} \cite{antonini-bulletin} \cite{antonini-crelle}. 
Needless to say, the focus of this article is on $C^*$-higher APS index formulae and for these higher indices we 
do need the extra assumptions detailed in the statement of Theorem \ref{c*-aps}.

\subsection{Spectral sections}
In the work of Melrose and Piazza on  families of Dirac operators on manifolds with boundary,
the condition that the boundary family be invertible was lifted at the expense of considering {\it generalized} APS
boundary conditions in the incomplete case or, equivalently,  {\it perturbed} Dirac operators in the case of
manifolds with cylindrical ends.
Generalized APS boundary conditions were defined through the notion of {\it spectral section} for the
boundary family.
The notion of spectral section can in fact be given for any family of Dirac operators on closed compact manifolds 
parametrized by a space $B$ and it is proved
in [MP 1,2] that a spectral section exists for a family $(\eth_b)_{b\in B}$ if and only if the associated 
family index  in the K-theory of $B$ vanishes.
It is also proved in   [MP 1,2] that there is a one-to-one correspondence between spectral sections for a family of Dirac operators $(\eth_b)_{b\in B}$
and smoothing perturbations $(A_b)_{b\in B}$ of this family such that $(\eth_b+A_b)_{b\in B}$ 
is invertible; these special perturbations are called
{\it trivializing} perturbations (they trivialize the vanishing index class of the family). By fiber-cobordism invariance of the index we obtain the existence of
 spectral sections, and therefore
trivializing perturbations, for a boundary family. This then  allows for the definition of an index class for
a family $(D_b)_{b\in B}$ of Dirac operators on manifolds with boundary, either in the form of a smoothly varying
family of generalized APS boundary value problems or, equivalently, as the index class of a family
of perturbed Dirac operators on manifolds with cylindrical ends with invertible boundary family.
This index class
 does depend 
on the choice of the spectral section; a formula relating two different choices of spectral sections can be given.

The whole theory was extended to Galois coverings by Leichtnam and Piazza in \cite{LPGAFA} \cite{LP03}, see also  \cite[p. 73]{PSJNCG},
using the notion of noncommutative spectral section (introduced originally by Wu). 
Subsequent improvements of the theory are due to Wahl, see 
\cite{wahl-asian}. Interesting geometric applications are given, for example,
 in \cite{PSJNCG}
and \cite{PS-akt}

Now, in the case of $G$-proper manifolds we do have have a
cobordism invariance of the index. Since the notion of noncommutative spectral section is quite general,
in that it applies to Dirac operators on quite arbitrary Hilbert $C^*$-modules,
we expect that it should be possible to extend the theory of spectral sections to $G$-proper manifolds. As this paper is already quite long we
leave this generalization to future research.

\section{Higher genera on G-proper manifolds with boundary}\label{sect:genera}

We can finally discuss higher genera on $G$-proper manifolds with boundary. 
We assume that $G$ has finitely many connected components, satisfies RD and
is such that $G/K$ is of non-positive sectional curvature. We bound ourselves to the even dimensional
case. The odd dimensional case can be treated by suspension as in Part 1 (but we shall not treat the odd dimensional case
in detail).

\medskip
First we treat the 
higher genera associated to the spin-Dirac operator. We thus assume that $M$, our 
cocompact $G$-proper
manifold with boundary \footnote{We do not use
$M_0$ in this part of the paper}, admits a $G$-invariant spin structure. 
Let $g$ be a $G$-invariant metric 
on $M$ which is, as usual, of product type near the boundary, $g=dx^2 + g_{\partial}$. 
%We refer to Part 1 for comments on the existence of such metrics.
We obtain a well-defined
spin-Dirac operator $D_g$.
 We assume 
additionally that $g_{\partial}$ is of positive scalar curvature. It then follows from  Lichnerowicz '
formula that $D_\partial$ is $L^2$-invertible. All the hypothesis of our $C^*$-higher APS index theorem
are thus fulfilled and we define the higher $\widehat{A}$-genera associated to $M$ as the numbers
\begin{equation}\label{higher -aroof}
\left\{\left(\int_{M} \chi \widehat{A}(M) \wedge \omega_\varphi -\frac{1}{2} \widehat{\eta}_\varphi (D_\partial)\right),\;\;[\varphi]\in H^{2*}_{{\rm diff}} (G)\right\},
\end{equation}
where we have rescaled the eta cocycle as 
\begin{equation}
\label{resceta}
\widehat{\eta}_\varphi (D_\partial):= (2\pi i)^p \eta_\varphi (D_\partial),\quad [\varphi]\in H^{2p}_{{\rm diff}} (G).
\end{equation}
We set 
\begin{equation}\label{higher -aroof-bis}
\widehat{A}(M,\partial M;\varphi):=\int_{M} \chi \widehat{A}(M) \wedge \omega_\varphi -\frac{1}{2} \widehat{\eta}_\varphi (D_\partial)\,,\;\;[\varphi]\in H^{2*}_{{\rm diff}} (G).
\end{equation}

Using the fact that  $\widehat{\eta}_\varphi (D_\partial)$
changes sign if we change the orientation of the boundary, we obtain at once 
the following additivity result: assume that
$X$ is a co-compact $G$-proper manifold without boundary and that
\begin{equation}\label{union}X=M_+\cup_H M_-
\end{equation}
with 
$M_\pm$
cocompact $G$ manifolds with boundary and 
$$\partial M_+=H=-
\partial M_-\,.$$ Assume that an equivariant metric
metric has been fixed on $X$ and that this metric is of product-type near 
$H$, inducing equivariant metrics  
on $M_\pm$ that are of product-type near the boundary.
Then, if the induced metric on $H$ is of psc, we immediately obtain the following additivity
result for higher $\widehat{A}$-genera:
$$\widehat{A}(X;\varphi)=\widehat{A}(M_+,\partial M_+;\varphi)
+ \widehat{A}(M_-,\partial M_-;\varphi).$$
\begin{definition}\label{def:extandable}
We call {\bf extendable}  a $G$-invariant metric of psc on $\partial M$ that admits an extension to a $G$-invariant 
metric of psc on $M_0$
which is  of product type near the boundary.
\end{definition}

\noindent
We have:

\begin{theorem}\label{a-roof-vanishing}
Let $M$ and $g$ as above. 
If  $g_{\partial}$  is extendable then 
$$ \widehat{A}(M,\partial M;\varphi)
%:=\int_{M} \chi \widehat{A}(M) \wedge \omega_\varphi -\frac{1}{2} \widehat{\eta}_\varphi (D_\partial)
=0,\;\;\forall [\varphi]\in H^{2*}_{{\rm diff}} (G)\,.
$$
More generally, if 
 $g_{\partial}$ is isotopic to an extendable metric of psc then 
$$ \widehat{A}(M,\partial M;\varphi)=0,\;\;\forall [\varphi]\in H^{2*}_{{\rm diff}} (G)\,.
$$
\end{theorem}

\begin{proof}
We prove directly the last, more general statement.
By our higher APS index formula it suffices to show that $\Ind (D_g)=0$ in $K_0 (C^*_r G)$. Recall that the double
of $M$ is a spin-boundary of a $G$-proper manifold with boundary.
Set $g_{\partial}:=g^0$. By hypothesis, there exists a path of metrics $g^t$ on $\partial M$ that are 
all $G$-invariant and of psc. Moreover $g^1$ extends to a metric $\tilde{g}$ on $M$ that is $G$-invariant, of product type
near the boundary and of psc. Consider  the cylinder $[0,1]\times \partial M$ with the metric $g^{{\rm cyl}}$ induced by
$\{g^t\}$. We can arrange that $g^{{\rm cyl}}$ is of product-type near the boundary.
We consider the union of $M$ with the cylinder $[0,1]\times \partial M$ and then we
glue an extra copy of $M$, obtaing a $G$-proper manifold diffeomorphic to the double of 
$M$.
By Bunke's additivity formula \cite{Bunke} (clearly true also in the $G$-proper context) we have that
$$0=\Ind (D_g)+\Ind (D_{g^{{\rm cyl}}}) + \Ind (D_{\tilde{g}})$$
with the first equality following from the bordism invariance of the index.
%\textcolor{cyan}{Also true in the $G$-proper context ?? I would guess so....}.
Since the last two summands on the right hand side are equal to 0 we conclude that also $\Ind (D_g)=0$, as required.
\end{proof}

\begin{corollary} Assume $g_{\partial}$ to be of psc. Then the higher genera 
%$\{(\,\int_{M} \chi \widehat{A}(M) \wedge \omega_\varphi -\frac{1}{2} \widehat{\eta}_\varphi (D_\partial)),\;\;[\varphi]\in H^{2*}_{{\rm diff}} (G)\}$
$ \{\widehat{A}(M,\partial M;\varphi),\;\;[\varphi]\in H^{2*}_{{\rm diff}} (G)\}$
are obstructions to the existence of an isotopy from $g_\partial$ to an extendable metric.
\end{corollary}

\bigskip
Next we deal with higher signatures on $G$-proper manifolds with boundary. As for the case of higher
$\widehat{A}$-genera, we bound ourselves to the case of even dimensional manifolds.
In what follows, until the end of this subsection, we  denote by $D$ the signature operator associated
to a $G$-invariant riemannian metric $g$ on an orientable $G$-proper manifolds with boundary $M$.
As usual we assume $g$ to be of product type near the boundary and denote by $g_{\partial}$ the induced
boundary metric.
We denote by $D_\partial$ the associated boundary operator. Depending on conventions, this is by definition
the odd signature operator of the boundary (conventions differ slightly but in a inconsequential way).
Similarly to the case of Galois coverings of compact manifolds with boundary, the situation for the signature
operator is complicated by the fact that it is too restrictive to assume that the signature operator on the boundary  is
$L^2$-invertible. The case of Galois coverings was treated successfully in \cite{LLP}, see \cite{LPFOURIER} for a survey, with simplifications
in \cite{wahl-product} \cite{Wahl_higher_rho} (in turn based crucially on \cite{LPAGAG}).
We shall build on these results. In order to keep the paper to a  reasonable size, we have decided
to simply  sketch carefully a program of work, leaving the details to future work.

\medskip
\noindent
Let $\dim M=2k$. 
Consider the following definitions, adapted from \cite{LPAGAG} and  \cite{wahl-product}.  Assume that 
there exists an orthogonal  $G$-invariant decomposition 
$$\Omega_{L^2} (\partial M)= V_{\partial M} \oplus W_{\partial M}$$
with the property that $D_\partial$  and the Hodge involution $\tau_{\partial M}$ are diagonal with respect to this decomposition, $D_\partial$ restricted to $V_{\partial M}$ admits a bounded $G$-equivariant inverse, 
and there exists a bounded
$G$-equivariant operator 
$\mathcal{I}$ on $\Omega_{L^2} (\partial M)$ which vanishes on $V_{\partial M}$, is an involution on 
$W_{\partial M}$ and anticommutes there with (the restrictions of) $\tau_{\partial M}$ and $D_\partial$.\\
The operator $A^\mathcal{I}_{\partial}:=i \mathcal{I}\tau_{\partial}$ is a $G$ equivariant bounded operator with the property that 
$D_\partial +A^\mathcal{I}_{\partial}$ is invertible. We can lift this perturbation 
to a perturbation $A^\mathcal{I}$ of the signature operator on the manifold with cylindrical ends 
associated to $M$ which, by construction, has an invertible boundary operator
and it is therefore $C^*$-Fredholm.

\medskip
\noindent
We  choose $V_{\partial M}=\overline{d\Omega^{k-1}_{L^2} (\partial M)}\oplus \overline{d^*\Omega^{k+1}_{L^2} (\partial M)}$ where  $d$ and $d^*$ are really the closure of $d$ and $d^*$ respectively and they are of course defined on their
respective domains \footnote{Being in the complete case, these operators are essentially closed}. We choose $W_{\partial M} = (V_{\partial M})^\perp$ and define $\Omega^{<}$ and $\Omega^{>}$
as the
subspaces of $W_{\partial M} $ made of forms of degree $\leq k-1$ and $\geq k$ respectively.

\medskip
\noindent
We now make the following

\medskip
\noindent
\begin{assumption}\label{assume-middle} The differential form Laplacian of $(\partial M,g_{\partial})$ is $L^2$-invertible in degree $k$.
\end{assumption}

%\medskip
%\noindent
%\textcolor{cyan}{From Paolo: in the Galois covering case this condition is a condition on the k-th 
%Novikov-Shubin invariant, and so, by Gromov-Shubin, is a homotopy invariant condition.
%In our case ??...}

\medskip
\noindent
Under this assumption we will be in the position to proceed as follows:

\begin{itemize}
\item We define $\mathcal{I}$ as the operator equal to $-1$ on $\Omega^{<}$, equal to $1$ on $\Omega^{>}$
and equal to $0$ on $V_{\partial M}$
and prove, thanks to our assumption, that $\Omega_{L^2} (\partial M)= V_{\partial M} \oplus W_{\partial M}$, $D_{\partial}$, $\tau_{\partial M}$
and $\mathcal{I}$ 
satisfy the above requirements.
\item We consider the $C^*_rG$-index class $\Ind (D,\mathcal{I})\in K_0 ((C^*( M,\Lambda^* M)^G)\equiv 
K_0 (C^*_r G)$ associated to the perturbed operator 
$D+ A^\mathcal{I}$. 
\item We prove, following \cite{Wahl_higher_rho} and \cite{fukumoto2}, that  $\Ind (D,\mathcal{I})$ is a $G$-proper-homotopy invariant of the pair $(M,\partial M)$
\item We now assume that $G$ has finitely
many connected components, satisfies
property RD and is such that $G/K$ has non-positive sectional curvature; by applying (a  refinement of)  our higher  APS index formula, we prove that for a $2p$-cocycle $\varphi\in Z^{2p}_{{\rm diff}} (G)$ we have 
$$\Ind_\varphi (D,\mathcal{I})\equiv (-1)^p \frac{2p !}{p!} \left<\tau_\varphi, \Ind (D,\mathcal{I})\right>=  \frac{1}{(2\pi i)^p}\int_{M} \chi L (M) \wedge \omega_\varphi -\frac{1}{2} \eta_\varphi (D_\partial, \mathcal{I})$$
where
$$ \eta_\varphi (D_\partial, \mathcal{I}):=2c_p \left[ \sum_{i=0}^{2p}   \int_0^{\infty}\sigma_\varphi (p_t,\dots,[\dot{p}_t, p_t],\dots,p_t)dt
\right]$$
Here $p_t= V(\widehat{D}_{\cyl} (t))$, with $\widehat{D}_{\cyl} (t)= tD_{\cyl} + g(t)  A^\mathcal{I}_\partial$, $g$ a smooth
function on $\RR$ such that $g(t)=0$
for $t<1/2$, $g(t)=t$ for $t\geq 1$. (The function $g$ is needed so as to ensure the use of Getzler rescaling
and thus convergence for small $t$.) 
\item
Rescaling $\eta_\varphi (D_\partial, \mathcal{I})$ by $(2\pi i)^p$ as in \eqref{resceta}, we obtain $\widehat{\eta}_\varphi (D_\partial, \mathcal{I})$ and define $$\sigma (M,\partial M;\varphi ):= %\frac{1}{(2\pi i)^p}
\int_{M} \chi L (M) \wedge \omega_\varphi -\frac{1}{2} \widehat{\eta}_\varphi (D_\partial, \mathcal{I})\,,
$$ call them {\it the higher 
signatures} of $(M,\partial M)$ and conclude from all of the above that under the stated hypothesis these are 
$G$-proper homotopy invariants of the pair $(M,\partial M)$.
\end{itemize}

\medskip

Let  $X=M_+\cup_H M_-$ as in \eqref{union} and assume that $H$ satisfies 
assumption \ref{assume-middle}; then we obtain immediately the
following additivity formula for higher signatures
$$\sigma (X;\varphi)=\sigma (M_+,\partial M_+;\varphi)
+ \sigma(M_-,\partial M_-;\varphi).$$
Finally, proceeding as in \cite[Corollary 0.4]{LLP} we can prove the following cut-and-paste 
invariance of the higher signatures of a $G$-proper manifold without boundary:\\
let, as before, $X=M_+\cup_{(H,{\rm id}_H)} M_-$ with $H$ satisfying assumption \ref{assume-middle};
if  $F:H\to H$ is a $\Gamma$-equivariant diffeomorphism, then
for $X_F:= M_+\cup_{(H,F)}  M_-$ we have
$$\sigma (X;\varphi)=\sigma (X_F;\varphi)\,.$$
Notice that already on Galois coverings the {\it higher} signatures are not, in general, cut-and-paste invariants.

\appendix

\section{On heat kernels and the Connes-Moscovici projector}

\subsection{Rapid exponential decay}
\label{red}
%\textcolor{cyan}{I have reviewed this section on August 9.}\\
Let $M$  be a closed smooth manifold carrying a smooth proper action of a Lie group $G$ with finitely many connected components and with compact quotient. 
%We denote the action of $G$ on $M$ by $(g,x)\mapsto gx$, $g\in G$, $x\in M$, and recall that
%the fact that the action is proper means that the map $G\times M\to M\times M$ given by $(g,x)\mapsto (x,gx)$, is proper. 
As in Subsection \ref{subsect:3algebras}  we choose an invariant complete Riemannian metric denoted by $h$, with associated distance function denoted by $d_{M}(x,y)$ for $x,y\in M$, and volume form $d{\rm vol}(x)$. 
In the following discussion the volume growth properties of the Lie group $G$ will be important: for any left-invariant (pseudo-)metric $d_{G}$ on $G$, denote by $B_{r}$ the closed ball 
of radius $r$ around the unit element. Let $\mu(g)$ be a left-invariant Haar measure on $G$, and write ${\rm Vol}_{G}(r)$ for the volume of $B_{r}$ with respect to this Haar measure. One says that $G$ has 
\begin{itemize}
\item[$i)$] {\em polynomial growth} if there are constants $C_{1}, C_{2}$ such that
${\rm Vol}_{G}(r)\leq C_{1}r^{C_{2}}$,
 
\item[$ii)$] 
{\em exponential growth} if there are constants $C_{1}, C_{2}$ such that
$
{\rm Vol}_{G}(r)\leq C_{1}e^{C_{2}r}.
$
\end{itemize}
It is known, see for example \cite{jenkins}, that Lie groups have at most exponential growth.
For the action of $G$ on $M$, choosing a basepoint $x_{0}\in M$, the function $d_{M}(g_{1}x_{0},g_{2}x_{0})$ on $G\times G$ defines a left invariant pseudo-metric on $G$. 
With this pseudometric, since we are assuming the quotient $M/G$ to be compact, it is easy to see that $G$ being of polynomial or exponential growth implies that $M$ is of polynomial resp. exponential growth.
%\footnote{\textcolor{blue}{$\bullet$ $\longrightarrow$ We must check the preceding discussion....}}

Recall the following definitions
\begin{definition}
\label{def-alg-kern}
Consider $C^{\infty}(M\times M)^{G}$. For $K\in C^{\infty}(M\times M)^{G}$ we  say that
\begin{itemize}
%\item[$i)$] $K$ is compactly supported if $\supp(K)\subset M\times M$ is $G$-compact with respect to
%the diagonal $G$-action,
\item[$i)$] $K$ is exponentially rapidly decreasing if 
\[
\forall q\in\mathbb{N}~\exists\, C_{q}>0~\mbox{\rm such that} \sup_{x,y\in M}\left| e^{qd(x,y)}\nabla^{m}_{x}\nabla^{n}_{y}k(x,y)\right| < C_{q}~\forall m,n\in\mathbb{N},
\]
\item[$ii)$] $K$ is polynomially rapidly decreasing if
\[
\forall q\in\mathbb{N}~\exists\, C_{q}>0~\mbox{\rm such that} \sup_{x,y\in M}\left| (1+d(x,y))^{q}\nabla^{m}_{x}\nabla^{n}_{y}k(x,y)\right| < C_{q}~\forall m,n\in\mathbb{N}.
\]
%\item[$iv)$] $K$ is exponentially decreasing  with weight $\delta$ if 
%\[
%\exists \,C_{\delta}>0~\mbox{\rm such that} \sup_{x,y\in M}\left| e^{\delta d(x,y)}\nabla^{m}_{x}\nabla^{n}_{y}k(x,y)\right| < C_{\delta}~\forall m,n\in\mathbb{N},
%\]

\end{itemize}
%\footnote{\textcolor{blue}{In the last definition, iv), do we really want to require the estimate for the derivatives ? Actually,
%it is not clear that we need iv)}}
We write %$\mathcal{A}^{c}_{G}(M)$, 
$\mathcal{A}^{\rm exp}_{G}(M)$ and $\mathcal{A}^{\rm pol}_{G}(M)$ for the set of elements in $C^{\infty}(M\times M)^{G}$ satisfying $i)$ and  $ii)$  above, and call them  ``rapidly exponentially" resp. ``rapidly polynomially''  decreasing  kernels. We shall also refer to elements of $\mathcal{A}^{\rm exp}_{G}(M)$ as ``kernels of rapid exponential decay''. 
\end{definition}

For integral operators with smooth kernels acting on the sections of a vector bundle $E$ we can similarly
define  $\mathcal{A}^{\rm exp}_{G}(M,E)$ and $\mathcal{A}^{\rm pol}_{G}(M,E)$.
We shall often omit the vector bundle $E$ from the notation.
 
The following Proposition is well-known but as we could not find an explicit reference for a proof we provide one for
the benefit of the reader:

\begin{proposition}\label{prop:exp-rap-decay-algebra}
Convolution of kernels give $\mathcal{A}_{G}^{\rm exp}(M)$  the structure of an algebra. If $G$ 
 has polynomial growth then $\mathcal{A}_{G}^{{\rm pol}}(M)$ is an algebra.
\end{proposition}

\begin{proof} 
Let us start by showing that the composition 
\[
(k_{1}*k_{2})(x,y)=\int_{M}k_{1}(x,z)k_{2}(z,y) d {\rm vol}(z)
\]
of two elements $k_{1}$ and $k_{2}$ in $\mathcal{A}^{\rm exp}_{G}$ is defined, i.e., that the integral converges.
This follows from the following
\begin{lemma}
Suppose that $f\in C^{\infty}(M)$ is such that 
\[
\forall q\in\mathbb{N}~\exists~ C_{q}>0~\mbox{such that}~ \left|\sup_{x\in M}f(x)e^{qd_M(x_{0},x)}\right|<C_{q},
\]
for some fixed $x_{0}\in M$. Then $f$ is integrable.
\end{lemma}
\begin{proof}
We start by writing
\[
\int_{M}|f(x)|d{\rm vol}(x)=\lim_{r\to\infty}\int_{B_{r}(x_{0})}|f(x)|d{\rm vol}(x),
\]
where $B_{r}(x_{0})$ denotes the geodesic ball of radius $r>0$ around $x_{0}$. It therefore suffices to  show that the sequence
\[
\alpha(n):=\int_{B_{n}(x_{0})}|f(x)|d{\rm vol}(x)
\]
is Cauchy. For this we use the fact that $M$ has at most exponential growth. When $n>m$, we have
\begin{align*}
|\alpha(n)-\alpha(m)|&=\int_{m\leq d(x,x_{0})\leq n}|f(x)|d{\rm vol}(x)\\
&\leq C \int_{m\leq d(x,x_{0})\leq n}e^{-qd(x,x_{0})}d{\rm vol}(x)\quad\mbox{for all}~q\in\mathbb{N},\\
&\leq C' e^{-qm+np},
\end{align*}
where $C$ and $C'$ are positive constants, and for some $p$ bounding the growth of $M$. But since this inequality holds true for any $q$, we can surely get the right hand side as small
as we want. It follows that the sequence $\{\alpha(n)\}$ is Cauchy and therefore $f$ is integrable.
\end{proof}
This Lemma shows that composition of elements $k_{1}$ and $k_{2}$ in $\mathcal{A}_{G}^{\rm exp}(M)$ produces a new kernel $k_{1}*k_{2}$ which is obviously $G$-invariant. To show that the composition
in fact lies in $\mathcal{A}_{G}^{\rm exp}(M)$, we have to show that it satisfies the exponential estimates defining $\mathcal{A}_{G}^{\rm exp}(M)$. We start by the trivial observation that if $k_{1}*k_{2}$ satisfies the exponential 
estimate for some fixed $Q>0$, it automatically satisfies the exponential estimates for $q<Q$. It suffices therefore to prove  the estimate for large enough $q\in\mathbb{N}$.

Because the $G$-action on $M$ is proper, we can choose a cut-off function: this is a $c\in C_{c}^{\infty}(M,\mathbb{R}_{+})$ satisfying
\begin{equation}
\label{prop-cut-off}
\int_{G}c(g^{-1}x)d\mu(g)=1,\quad\mbox{for all}~x\in M,
\end{equation}
where $\mu$ denotes the Haar measure on $G$. Inserting $1$ in this way in the equation defining the composition $k_{1}*k_{2}$, the Lemma shows us that we can interchange the 
two integrations and obtain
\[
(k_{1}*k_{2})(x,y):=\int_{G}\int_{M}k_{1}(gx,z)c(z)k_{2}(z,gy)d{\rm vol}(z)d\mu(g).
\]
From this we derive
\begin{align*}
\left|e^{qd(x,y)}(k_{1}*k_{2})(x,y)\right|&\leq \int_{G}\int_{M}\left|e^{qd(gx,z)}k_{1}(gx,z)c(z)k_{2}(z,gy)e^{qd(z,gy)}\right|d{\rm vol}(z)d\mu(g)\\
&\leq C\int_{G}e^{-qd(gx,x_{0})}e^{-qd(x_{0},gy)}d\mu(g),\quad \mbox{for some}~x_{0}\in{\rm supp}(c).
\end{align*}
For fixed $x$ and $y$, the functions $d(gx,x_{0})$ and $d(x_{0,}gy)$ define equivalent left-invariant pseudo-metrics on $G$ and since Lie groups have at most exponential growth, we see that
the integral converges for large $q$. This shows that the composition $k_{1}*k_{2}$ satisfies the exponential inequality for $q$ large enough. By the remark above, this implies the 
estimates for smaller, and therefore for all $q\in\mathbb{N}$. Including the derivatives $\nabla_{x}^{m}\nabla_{y}^{n}$ in the estimates, the argument proving these is the same as above.

\end{proof}

\noindent
The following result is also  well-known, but once again, we could not find a detailed proof in the
literature. We give one in the proof of  Proposition  \ref{prop:CM-rapid-appendix} below.

\begin{proposition}
\label{hk}
The heat kernel of a $G$-invariant Dirac  Laplacian, $\exp (- t D^2)$, is an element in  $\mathcal{A}_{G}^{{\rm exp}}(M)$.\\
\end{proposition}

%\begin{proposition}
%Recall the definition
%\[
%\mathcal{A}_{G}^{{\rm pol}}(M):=\{k\in C^{\infty}(M\times M)^{G},~\forall \ell\in\mathbb{N}~\exists C_{\ell}>0~\mbox{such that} \sup_{x,y\in M}\left| (1+d(x,y))^{\ell}\nabla^{m}_{x}\nabla^{n}_{y}k(x,y)\right| < C_{\ell}~\forall m,n\in\mathbb{N}\}.
%\]
%If  $G$ has polynomial growth then $\mathcal{A}_{G}^{{\rm pol}}(M)$ is an algebra under the convolution product.
%\end{proposition} 
%
%\begin{proof}
%The proof  follows by exactly the same arguments as above; one now uses the fact that $G$ of polynomial growth implies, by cocompactness and properness of the action, that $M$ is also of polynomial growth.\\
%\end{proof}

%Later on, in defining index classes, we shall need the following 
%
%\begin{proposition}\label{prop:semi-ideal} 
%Let $k_1, k_2\in \mathcal{A}_{G}^{{\rm exp}}(M)$ and let $W\in  \mathcal{A}_{G}(M)$ be exponentially decreasing 
%with weight $\delta>0$. Then $k_1 * W * k_2 \in  \mathcal{A}_{G}^{{\rm exp}}(M)$.
%\end{proposition}
%
%\begin{proof}
%First notice that if $k$ is rapidly exponentially decreasing then $k * W$ and $W * k$ are exponentially decreasing, so that
%the composition is indeed well defined. ( \textcolor{m}{to be doubled checked})
%\end{proof}

%\subsection{The case of manifolds with boundary}
%HERE WE SHOULD CONSTRUCT THE EXACT SEQUENCE
%\[
%0\to\mathcal{A}_{G}^{{\rm pol},\epsilon}(M)\to ~^b\!\mathcal{A}_{G}^{{\rm pol},\epsilon}(M)\to ~^{b}\!\mathcal{A}_{G}^{{\rm pol},\epsilon}(\overline{N^{+}\partial M}_{0})\to 0.
%\]

\subsection{The Connes-Moscovici projector}
Recall the Connes-Moscovici parametrix 
\begin{equation}
 Q
:= \frac{I-\exp(-\frac{1}{2} D^- D^+)}{D^- D^+} D^+
\end{equation}
with
$I-Q D^+ = \exp(-\frac{1}{2} D^- D^+)$, $I-D^+ Q_{V} =  \exp(-\frac{1}{2} D^+ D^-)$
and  the
idempotent
\begin{equation}
\label{cm-idempotent}
V=\left( \begin{array}{cc} e^{-D^- D^+} & e^{-\frac{1}{2}D^- D^+}
\left( \frac{I- e^{-D^- D^+}}{D^- D^+} \right) D^-\\
e^{-\frac{1}{2}D^+ D^-}D^+& I- e^{-D^+ D^-}
\end{array} \right)
\end{equation}
The following Proposition, well known to the experts, clarifies in which algebra
this idempotent lives. As we could not find a detailed proof, we supply one below.

\begin{proposition}\label{prop:CM-rapid-appendix}
The Connes-Moscovici idempotent  is an element in  $M_{2\times 2} (\mathcal{A}_G^{{\rm exp}} (M))$
(with the identity adjoined).

%In particular, if $G$ has finitely many connected components and satisfies the RD condition 
% then the Connes-Moscovici projector is an element in
%$M_{2\times 2} (\mathcal{A}^\infty_G (M,E))$.
\end{proposition}

\begin{proof}
The most problematic entry  in the idempotent \eqref{cm-idempotent} is the one in the right upper corner. This entry has the form $f(D)D^-$, where
\begin{equation}
\label{function}
f(z):=e^{-z^2/2}\frac{(1-e^{-z^2})}{z^2}.
\end{equation}
Using finite propagation speed methods, we can get estimates for the kernel $k_{f(D)}(x,y)$ of the operator $f(D)$. Following \cite[\S 4.1]{taylor}, we define the following class of functions for $W>0$:
\[
\mathcal{S}^{m}_W:=\{f\in C^\infty(\bar{\Omega}_W)~\mbox{even and holomorphic on} ~\Omega_W,~|f^{(k)}(z)|\leq C_k(1+|z|)^{m-k},~\forall k.\},
\]
where $\Omega_W$ is the infinite strip $|{\rm Im}(z)|< W$ in the complex plane.
The main result in {\em loc. cit.} gives that
\[
f\in \mathcal{S}^{m}_W\Longrightarrow f(D)=f(D)^\#+f(D)^b,
\]
with $f(D)^\#$ an order $m$ Pseudodifferential operator whose Schwartz kernel $k^\#_{f(D)}\in\mathcal{D}'(M\times M)$ is supported in 
\begin{equation}
\label{supp}
\{(x,y)\in M\times M,~d(x,y)<1\},
\end{equation}
and with singular support on the diagonal. The operator $f(D)^b$ is smoothing with kernel satisfying
\[
|\nabla_x^l\nabla_y^nk^b_{f(D)}(x,y)|\leq C_j(1+d(x,y))^{-j}e^{-Wd(x,y)},\quad\mbox{for all}~j,l,n.
\]
To prove the Lemma, it therefore suffices to show that $f\in \mathcal{S}^{-\infty}_W:=\bigcap_m\mathcal{S}^{m}_W$ for all $W>0$: if this is the case the 
kernel $k^b_{f(D)}(x,y)$ obviously satisfies the required estimates of Definition \ref{def-alg-kern} $ii)$ to be in $\mathcal{A}_G^{{\rm exp}} (M)$. On the other hand,
since $f\in \mathcal{S}^{-\infty}_W$, the operator $f(D)^\#$ is smoothing as well, so its kernel is a $G$-invariant smooth function on $M\times M$ with support in \eqref{supp}.
Passing to the quotient space $(M\times M)\slash G$, it defines a smooth function with support in the compact subset $d^{-1}([0,1])$ which is therefore bounded. We can therefore 
easily obtain the estimates of Definition \ref{def-alg-kern} $ii)$ for $k^\#_{f(D)}(x,y)$. The estimates for the derivatives of $k^\#_{f(D)}(x,y)$ follow by the same argument.

We are therefore left to show that the complex function \eqref{function} belongs to the class $\mathcal{S}^{-\infty}_W$ defined above. For this we write $f=f_1f_2$ with 
$f_1:=e^{-z^2/2}$ and $f_2(z):=(1-e^{-z^2})/ z^2$. For the function $f_1(z)$ on the strip $\Omega_W$ we clearly have the estimates
\[
|f_1(z)|=|e^{-z^2}|\leq e^{W^2}e^{-{\rm Re}(z)^2}\leq C^{f_1}_m(1+|z|)^m,\quad \mbox{for all}~m\in\mathbb{Z}.
\]
The $k$-the derivative of $f_1(z)$ has the form $P_k(z)e^{-z^2/2}$, where $P_k(z)$ is a polynomial of degree $k$. In a similar way we then get the estimates
\[
|f_1^{(k)}(z)|\leq |P_k(z)|e^{W^2}e^{-{\rm Re}(z)^2}\leq C^{f_1}_{m,k}(1+|z|)^{m-k},\quad \mbox{for all}~m\in\mathbb{Z},
\]
because of the dominating term $e^{-{\rm Re}(z)^2}$. This shows that $f_1\in \mathcal{S}^{-\infty}_W$ for all $W>0$. Remark that this argument shows that the heat kernel
$e^{-tD^2}$ is rapidly exponentially decaying, i.e., belongs to $\mathcal{A}_G^{{\rm exp}} (M)$.

For the function $f=f_1f_2$, with $f_1$ being an element of $\mathcal{S}^{-\infty}_W$ for all $W>0$, it suffices to show that $f_2(z)$, together with all its derivatives is bounded
on $\Omega_W$ for all $W>0$. Indeed, in that case when $|f_2(z)|\leq M_k^{f_2}$, we have
\begin{align*}
|f^{(k)}(z)|&=\left|\sum_{i=0}^k \binom{k}{i}f_1^{(i)}(z)f_2^{(k-i)}(z)\right|\\
&\leq \sum_{i=0}^k \binom{k}{i}|f_1^{(i)}(z)||f_2^{(k-i)}(z)|\\
&\leq \sum_{i=0}^k \binom{k}{i} C^{f_1}_{m-k+i,i}(1+|z|)^{m-k}M^{f_2}_{k-i}=:C'_{m,k}(1+|z|)^{m-k},\quad \mbox{for all}~m\in\mathbb{Z},
\end{align*}
proving the desired inequalities. 

Since $f_2(z)$ and all its derivatives are holomorphic on $\Omega_W$, we only have to check the behaviour as $z$ approaches $\pm\infty$.
For the specific $f_2(z)$ we have above, one easily proves by induction that its $k$'th derivative has the general form
\[
f^{(k)}_2(z)=\frac{a_0-e^{-z^2}(a_0+a_2z^2+\ldots+a_{2k}z^{2k})}{z^{k+2}},
\]
with $a_i\in\mathbb{Z}$. With this formula it is not difficult to see that $\lim_{z\to\pm\infty}f_2^{(k)}(z)=0$, so $f^{(k)}$ is indeed bounded for each $k$.
This completes the proof.
\end{proof}

{\small \bibliographystyle{plain}
\bibliography{rel-coverings}
}

\end{document}